\documentclass[reqno,11pt]{amsart}
\usepackage{amsmath,amssymb,dsfont,amsthm,mathrsfs,amsfonts}
\usepackage[title]{appendix}
\usepackage[a4paper,  margin=3.3cm]{geometry}

\newtheorem{theorem}{Theorem}[section]
\newtheorem{corollary}[theorem]{Corollary}
\newtheorem{lemma}[theorem]{Lemma}
\newtheorem{proposition}[theorem]{Proposition}
\theoremstyle{definition}
\newtheorem{definition}[theorem]{Definition}
\theoremstyle{remark} \theoremstyle{remark}
\newtheorem{assumption}{Assumption}
\newtheorem{remark}[theorem]{Remark}
\numberwithin{equation}{section}
\DeclareMathOperator*{\esssup}{\mathrm{ess\,sup}}

\allowdisplaybreaks[3]

\usepackage{graphicx}
\title[States of an infinite fission-death system]{Evolution of states of an infinite fission-death system}

\author{Yuri  Kozitsky}

\address{Instytut Matematyki, Uniwersytet Marii Curie-Sk{\l}odowskiej, 20-031 Lublin, Poland}
\email{jkozi@hektor.umcs.lublin.pl}

\author{Agnieszka Tana{\'s}}
\address{Instytut Matematyki, Uniwersytet Marii Curie-Sk{\l}odowskiej, 20-031 Lublin, Poland}
\email{agnieszka.puchacz@interia.eu}
\begin{document}



\begin{abstract}

The evolution of an infinite system of interacting point entities
with traits $x\in \mathds{R}^d$ is studied. The elementary acts of
the evolution are state-dependent death of an entity with rate that
includes a competition term and independent fission in the course of
which an entity gives birth to two new entities and simultaneously
disappears. The states of the system are probability measures on the
corresponding configuration space and the main result of the paper
is the construction of the evolution $\mu_0\to \mu_t$, $t>0$, of
states in the class of sub-Poissonian measures.

\end{abstract}

\maketitle 

\section{Introduction}
\subsection{Posing}

In recent years, there has been a lot of studies of the stochastic
dynamics of structured populations, see, e.g.,
\cite{Ba,BL,Chris,JK11,KK,KK2}. Typically, the structure is
introduced by assigning to each entity a trait $x\in X$. Then the
population dynamics consists in changing the traits of its members
that includes also their appearance and disappearance. Usually, one
endows the trait space $X$ with a locally-compact topology and
assumes that: (a) the populations are locally finite, i.e., compact
subsets of $X$ may contain traits of finite sub-populations only;
(b) the dynamics of a given entity is mostly affected by the
interaction with entities whose traits belong to a compact
neighborhood of its own trait. Then the local structure of the
population is determined by the network of such interactions. Since
the traits of a finite population lie in a compact subset of $X$,
each of its members has a compact neighborhood containing the traits
of the rest of population. In view of this, in order to clear
distinguish between global and local effects one should deal with
infinite populations and noncompact trait spaces. In the statistical
mechanics of interacting physical particles, this conclusion had led
to the concept of the thermodynamic (infinite-volume) limit, see,
e.g., \cite[pp. 5,6]{Simon}, and, thereby, to the description of the
states of thermal equilibrium as probability measures on the space
of particle configurations. Such states are constructed from local
conditional states and are Gibbsian, i.e., they satisfy a specific
consistency condition.

In this article, we study the Markov evolution of a possibly
infinite system of point entities (particles) with trait space
$X=\mathds{R}^d$, $d\geq 1$. The pure states of the system are
locally finite configurations $\gamma \subset \mathds{R}^d$, see,
e.g., \cite{Berns,KK,KK1,KK2}, whereas the general states are
probability measures on the space of all such configurations. The
elementary acts of the evolution are: (a) state-dependent
disappearance (death) with rate $m(x) + \sum_{y\in \gamma\setminus
x} a(x-y)$; (b) independent fission with rate $b(x|y_1, y_2)$ in the
course of which the particle with trait $x\in \gamma$ gives birth to
two particles, with traits $y_1, y_2 \in \mathds{R}^d$, and
simultaneously disappears from $\gamma$. The model with this kind of
death and budding instead of fission, cf. \cite{Chris}, is known as
the Bolker-Pacala model. Its recent study can be found in
\cite{KK,KK1}, see also the literature quoted therein. A similar
model with fission (fragmentation) in which each particle produces a
(random) finite number of new particles was introduced and studied
in \cite{tan}. The main result of the present work is the
construction of the global in time evolution of states in a certain
class of probability measures.

\subsection{The overview}
As mentioned above, the state space of the model is the set $\Gamma$
of all subsets $\gamma \subset \mathds{R}^d$ such that the set
$\gamma_\Lambda:=\gamma\cap\Lambda$ is finite whenever $\Lambda
\subset \mathds{R}^d$ is compact. For compact $\Lambda$, we define
the map $\Gamma\ni \gamma \mapsto N_\Lambda (\gamma) =
|\gamma_\Lambda|\in \mathds{N}_0$, where $|\cdot|$ denotes
cardinality and $\mathds{N}_0$ stands for the set of nonnegative
integers. Then $\mathcal{B}(\Gamma)$ will denote the smallest
$\sigma$-field of subsets of $\Gamma$ with respect to which all
these maps are measurable. That is, $\mathcal{B}(\Gamma)$ is
generated by the family of sets
\begin{equation}
  \label{o1}
 \Gamma^{\Lambda,n} :=\{ \gamma \in \Gamma: N_\Lambda (\gamma) =
 n\},  \qquad n\in \mathds{N}_0, \ \ \Lambda - {\rm compact}.
\end{equation}
It is known \cite{KK,KK2} that $(\Gamma,\mathcal{B}(\Gamma))$ is a
standard Borel space. The set of $n$-point configurations $\Gamma^n$
and the set of all finite configurations $\Gamma_0$ then are
\begin{equation*}
 \Gamma^n = \{ \gamma\in \Gamma: |\gamma|=n\}, \qquad  \Gamma_0 := \bigcup_{n=0}^\infty \Gamma^n \in \mathcal{B}(\Gamma).
\end{equation*}
For compact $\Lambda$, we let $\Gamma_\Lambda =\{ \gamma:
\gamma\subset \Lambda\} \subset
  \Gamma_0$
and define
\begin{equation*}
 \mathcal{B}(\Gamma_\Lambda) = \{ \mathbb{A}\cap \Gamma_\Lambda: \mathbb{A} \in
 \mathcal{B}(\Gamma)\} \subset \mathcal{B}(\Gamma_0) = \{ \mathbb{A}\cap \Gamma_0: \mathbb{A} \in
 \mathcal{B}(\Gamma)\} \subset \mathcal{B}(\Gamma).
\end{equation*}
Clearly, $(\Gamma_0, \mathcal{B}(\Gamma_0))$ and $(\Gamma_\Lambda,
\mathcal{B}(\Gamma_\Lambda))$ are standard Borel spaces. By
$\mathcal{P}(\Gamma)$, $\mathcal{P}(\Gamma_0)$,
$\mathcal{P}(\Gamma_\Lambda)$ we denote the sets of all probability
measures on $(\Gamma,\mathcal{B}(\Gamma))$,
$(\Gamma_0,\mathcal{B}(\Gamma_0))$ and
$(\Gamma_\Lambda,\mathcal{B}(\Gamma_\Lambda))$, respectively.

For a compact $\Lambda$ and $\mathbb{A}\in
\mathcal{B}(\Gamma_\Lambda)$, we set $\mathbb{C}_{\mathbb{A}} = \{
\gamma \in \Gamma: \gamma_\Lambda \in \mathbb{A}\}$ and let
$\mathcal{B}_\Lambda (\Gamma)$ be the sub-$\sigma$-field of
$\mathcal{B}(\Gamma)$ generated by all such \emph{cylinder} sets
$\mathbb{C}_{\mathbb{A}}$. A \emph{cylinder} function $F : \Gamma
\to \mathds{R}$ is a $\mathcal{B}_\Lambda
(\Gamma)/\mathcal{B}(\mathds{R})$-measurable function for some
compact $\Lambda$. Here by $\mathcal{B}(\mathds{R})$ we denote the
Borel $\sigma$-field of subsets of $\mathds{R}$. For a compact
$\Lambda$ and a given $\mu \in \mathcal{P}(\Gamma)$, by setting
\begin{equation}
\label{Rel}
\mu(\mathbb{C}_{\mathbb{A}}) = \mu^\Lambda (\mathbb{A})
\end{equation}
 we determine
$\mu^\Lambda \in \mathcal{P}(\Gamma_\Lambda)$ -- the
\emph{projection} of $\mu$. Note that all such projections
$\{\mu^\Lambda\}_{\Lambda}$ of a given $\mu \in \mathcal{P}(\Gamma)$
are consistent in the Kolmogorov sense.

Each $\mu\in \mathcal{P}(\Gamma)$ is characterized by its values on
the sets (\ref{o1}); in particular, by their local moments
\begin{equation}
  \label{o2}
\int_\Gamma N_\Lambda^m d\mu =: \mu(N_\Lambda^m) = \sum_{n=0}^\infty
n^m \mu(\Gamma^{\Lambda, n}), \qquad m\in \mathds{N}.
\end{equation}
This characterization naturally includes the dependence of
$\mu(\Gamma^{\Lambda, n})$ on $n$. A homogeneous Poisson measure
$\pi_\varkappa \in \mathcal{P}(\Gamma)$ with density $\varkappa >0$
has the property $\pi_\varkappa(\Gamma_0) = 0$. For this measure, it
follows that
\begin{equation}
  \label{o3}
  \pi_\varkappa (\Gamma^{\Lambda, n}) = \frac{\left(\varkappa |\Lambda|\right)^n}{n!} \exp\left( - \varkappa |\Lambda| \right),
\end{equation}
where $|\Lambda|$ stands for the volume of $\Lambda$. In our
consideration, the set of sub-Poissonian measures $\mathcal{P}_{\rm
exp}(\Gamma)$ plays an important role, see Definition \ref{0df}
below and the corresponding discussion in \cite{KK,KK1}. For each
$\mu \in \mathcal{P}_{\rm exp}(\Gamma)$, there exists $\varkappa >0$
such that
\begin{equation}
 \label{o3a}
 \mu(N_\Lambda^m) \leq \pi_\varkappa (N_\Lambda^m),
\end{equation}
holding for all compact $\Lambda$ and $m\in \mathds{N}$.

The Markov evolution is described by the Kolmogorov equation
\begin{equation}
  \label{1}
  \dot{F}_t = L F_t , \qquad F_t|_{t=0} = F_0,
\end{equation}
where $\dot{F}_t$ denotes the time derivative of an
\emph{observable} $F_t:\Gamma\to \mathds{R}$. The operator $L$
determines the model, and in our case it is
\begin{gather}
  \label{2}
 (LF)(\gamma) = \sum_{x\in \gamma}\left( m (x) + \sum_{y\in \gamma\setminus x} a (x-y)\right)
\left[ F(\gamma\setminus x) - F(\gamma)
 \right] \\[.2cm] \nonumber + \sum_{x\in \gamma}
 \int_{(\mathds{R}^d)^2} b(x|y_1 , y_2)\left[ F(\gamma\setminus x \cup\{y_1, y_2\}) - F(\gamma)
 \right] dy_1 d y_2.
\end{gather}
In expressions like $\gamma \cup x$, we treat $x$ as the singleton
$\{x\}$. The first term in (\ref{2}) describes the death of the
particle with trait $x$ occurring: (i) independently with rate $m(x)
\geq 0$; (ii) under the influence (competition) of the rest of the
particles in $\gamma$ occurring with rate
\begin{equation}
  \label{3}
E^a (x, \gamma\setminus x) := \sum_{y\in \gamma\setminus x} a
(x-y)\geq 0.
\end{equation}
The second term in (\ref{2}) describes independent fission  with rate
$b(x|y_1 , y_2)\geq 0$.

The evolution of states $\mu_0 \to \mu_t$ is defined by the
Fokker-Planck equation
\begin{equation}
  \label{4}
  \dot{\mu}_t = L^* \mu_t, \qquad \mu_t|_{t=0} =\mu_0,
\end{equation}
where $L^*$ is related to (\ref{2}) according to the rule $(L^*
\mu)(\mathbb{A})= \mu(L \mathds{1}_{\mathbb{A}} )$, $\mathbb{A}\in
\mathcal{B}(\Gamma)$; $\mathds{1}_{\mathbb{A}}$ is the indicator
function. Both evolutions are in the duality $\mu_0 (F_t) =
\mu_t(F_0)$. Here and in the sequel, we use the notation $\mu(F)=
\int F d \mu$, cf. (\ref{o2}).

The direct use of $L$ and/or $L^*$ as linear operators in
appropriate Banach spaces is possible only if one restricts the
consideration to states on $\Gamma_0$. Otherwise, the sums in
(\ref{2}) and (\ref{3}) -- taken over infinite configurations -- may
not exist. At the same time, constructing evolutions of finite
sub-populations contained in compact sets followed by taking the
`infinite-volume' limit -- as it is done in the theory of Gibbs
fields \cite{Simon} -- can hardly be realized here as the evolution
usually destroys the consistency of the local states. Instead of
trying to construct global states from local ones, we will we
proceed as follows. Let $C_0 (\mathds{R}^d)$ stand for the set of
continuous real-valued functions with compact support. Then the map
\[
\Gamma \ni \gamma \mapsto F^\theta (\gamma) := \prod_{x\in \gamma}
(1+ \theta (x)), \qquad \theta \in \varTheta:=\{\theta\in C_0
(\mathds{R}^d): \theta(x) \in (-1,0]\},
\]
is clearly measurable and satisfies $0<F^\theta (\gamma) \leq 1$ for
all $\gamma$. The set $\varTheta$ clearly has the following
properties: (a) for each pair of distinct $\gamma, \gamma' \in
\Gamma$, there exists $\theta \in \varTheta$ such that $F^\theta
(\gamma)\neq F^\theta (\gamma')$; (b) for each pair $\theta ,
\theta'\in \varTheta$, the point-wise combination $\theta + \theta'
+\theta \theta'$ is also in $\varTheta$; (c) the zero function
belongs to $\varTheta$. From this it follows that $\{F^\theta:
\theta \in \varTheta\}$ is a measure defining class, i.e.,
$\mu(F^\theta)= \nu(F^\theta)$, holding for all $\theta \in
\varTheta$, implies $\mu=\nu$ for each $\mu, \nu \in
\mathcal{P}(\Gamma)$, see \cite[Proposition 1.3.28, page 113]{AKKR}.
Noteworthy, for each $\theta \in \varTheta$, $\mu(F^\theta) =
\mu^{\Lambda_\theta} (F^\theta)$, where a compact $\Lambda_\theta$
is such that $\theta(x) = 0$ for $x\in
\Lambda^c_{\theta}:=\mathds{R}^d \setminus \Lambda_\theta$.

Our results related to (\ref{2}), (\ref{4}) consist in the
following:
\begin{itemize}
  \item[1.] Constructing the evolution $[0,+\infty)\ni t \mapsto \mu_t \in \mathcal{P}(\Gamma_0)$, $\mu_t|_{t=0}=\mu_0\in \mathcal{P}(\Gamma_0)$,
  by proving the existence of a unique classical solution of (\ref{4}) in
  the Banach space $\mathcal{M}$ of signed measures on $\Gamma_0$ with bounded
  variation.
\item[2.] Constructing the evolution $[0,+\infty) \ni t
  \mapsto \mu_t \in \mathcal{P}_{\rm exp}(\Gamma)$, $\mu_t|_{t=0}=\mu_0\in  \mathcal{P}_{\rm exp}(\Gamma)$, such that:
\begin{itemize}
  \item[2.1.] for each compact $\Lambda$ and  $t\geq 0$, $\mu_t^{\Lambda}$ -- as a measure on $\Gamma_0$ --  lies in the
  domain $\mathcal{D}(L^*) \subset \mathcal{M}$;
 \item[2.2.]  for
  each $\theta \in \varTheta$, the map $(0,+\infty) \ni t \mapsto \mu_t (F^\theta)$ is continuously
  differentiable and the following holds
\begin{equation}
  \label{Jan}
 \frac{d}{dt}
  \mu_t (F^\theta) = (L^* \mu^{\Lambda_\theta}_t)(F^\theta).
\end{equation}
\end{itemize}
\end{itemize}
Item 1 is realized in Theorem \ref{1ftm}. The main idea of how to
construct the evolution $\mu_0 \to \mu_t$ stated in item 2 is to
obtain it from the evolution $B_0 (\theta)\to B_t(\theta)$, $\theta
\in \varTheta$ by solving the evolution equation related to those in
(\ref{1}) and (\ref{4}). Here $B_0 (\theta) = \mu_0 (F^\theta)$ with
$\mu_0 \in \mathcal{P}_{\rm exp}(\Gamma)$. This is realized in
Theorem \ref{1tm} and Corollary \ref{Jaco}. One of the hardest
points of this scheme is to prove that $B_t(\theta)= \mu_t
(F^\theta)$ for a unique sub-Poissonian measure. At this stage, we
deal with the evolution of local states constructed in realizing
item 1.

\section{Preliminaries and the Model}
We begin by briefly introducing the relevant aspects of the
technique used in this work. Its more detailed description
(including the notations) can be found in \cite{KK,KK2} and in the
publications quoted therein.

\subsection{Measures and functions on configuration spaces}
It is know that
\begin{equation*}
B_{\pi_\varkappa} (\theta):= \pi_\varkappa (F^\theta) = \exp\left(
\varkappa\int_{\mathds{R}^d} \theta (x)
 dx\right).
\end{equation*}
Obviously, $B_{\pi_\varkappa}$ can be continued to an exponential
type entire function of $\theta \in L^1(\mathds{R}^d)$.
\begin{definition}
 \label{0df}
The set of sub-Poissonian measures $\mathcal{P}_{\rm exp} (\Gamma)$
consists of all those $\mu\in \mathcal{P}(\Gamma)$ for each of which
 $\mu(F^\theta)$ can be continued to an exponential type entire function of
$\theta \in L^1(\mathds{R}^d)$.
\end{definition}
It can be shown that $\mu\in \mathcal{P}_{\rm exp} (\Gamma)$ if and
only if  $\mu(F^\theta)$ might be written in the form
\begin{equation}
 \label{6b}
 \mu(F^\theta) = 1 + \sum_{n=1}^\infty \frac{1}{n!} \int_{(\mathds{R}^d)^n} k^{(n)}_\mu (x_1 , \dots , x_n) \theta (x_1) \cdots \theta(x_n) d x_1 \cdots d x_n,
\end{equation}
where $k^{(n)}_\mu$ is the $n$-th order \emph{correlation function}
of $\mu$. Each $k^{(n)}_\mu$ is a symmetric element of
$L^\infty((\mathds{R}^d)^n)$, and the collection $\{k^{(n)}_\mu
\}_{n\in \mathds{N}}$ satisfies
\begin{equation}
 \label{6c}
 \|k^{(n)}_\mu\|_{L^\infty((\mathds{R}^d)^n} \leq  \varkappa^{ n}, \qquad n \in \mathds{N},
\end{equation}
holding with some $\varkappa >0$. Note that $k^{(n)}_\mu $ is
positive and $k^{(n)}_{\pi_\varkappa} = \varkappa^n$; hence,
(\ref{6c}) means that $k^{(n)}_\mu (x_1 , \dots , x_n)\leq
\varkappa^n$ by which one gets (\ref{o3a}).

Now we turn to functions $G:\Gamma_0 \to \mathds{R}$. It can be
proved that such a function is
$\mathcal{B}(\Gamma_0)/\mathcal{B}(\mathds{R})$-measurable if and
only if there exists the collection of symmetric Borel functions
$G^{(n)}:(\mathds{R}^d)^n \to \mathds{R}$, $n\in \mathds{N}$, such
that
\begin{equation}
 \label{7}
 G(\eta) = G^{(n)}(x_1 , \dots , x_n), \qquad {\rm for}  \ \ \eta = \{x_1 , \dots , x_n\}.
\end{equation}
\begin{definition}
 \label{1df}
 A measurable function $G:\Gamma_0 \to \mathds{R}$ is said to have bounded support if: (a) there exists compact
 $\Lambda \subset \mathds{R}^d$ such that $G(\eta) =0$ whenever $\eta\cap\Lambda \neq \eta$; (b) there exists
 $N\in \mathds{N}$ such that $G(\eta) =0$ whenever $|\eta|>N$. By $B_{\rm bs}(\Gamma_0)$ we denote the set of all bounded functions with bounded support.
 For each $G\in B_{\rm bs}(\Gamma_0)$, by $\Lambda_G$ and $N_G$ we denote the smallest $\Lambda$ and $N$ with the properties just mentioned,
 and use the notations $C_G = \sup_{\eta\in \Gamma_0} |G(\eta)|$.
\end{definition}
The Lebesgue-Poisson measure $\lambda$ on $(\Gamma_0,
\mathcal{B}(\Gamma_0))$ is defined by the integrals
\begin{equation}
 \label{8}
\int_{\Gamma_0} G(\eta) \lambda (d \eta) = G(\varnothing) +
\sum_{n=1}^\infty \frac{1}{n!} \int_{(\mathds{R}^d)^n} G^{(n)}(x_1 ,
\dots x_n) d x_1 \cdots d x_n,
\end{equation}
with all $G\in B_{\rm bs}(\Gamma_0)$. For such $G$, we set
\begin{equation}
 \label{9}
 (KG)(\gamma) = \sum_{\eta \Subset \gamma} G(\eta), \qquad \gamma \in \Gamma,
\end{equation}
where $\eta\Subset \gamma$ means that $\eta\subset \gamma$ and $\eta\in \Gamma_0$. Clearly, cf. Definition \ref{1df}, we have that
\begin{equation}
 \label{10}
 |(KG)(\gamma)| \leq C_G \left(1 + |\gamma\cap \Lambda_G| \right)^{N_G}, \qquad G\in B_{\rm bs}(\Gamma_0).
\end{equation}
Like in (\ref{7}), we introduce the function $k_\mu:\Gamma_0 \to
\mathds{R}$ such that $k_\mu (\eta) = k^{(n)}_\mu (x_1 , \dots,
x_n)$ for $\eta =\{x_1 , \dots, x_n\}$, $n\in \mathds{N}$, and
$k_\mu(\varnothing) =1$. Then we rewrite (\ref{6b}) as follows
\begin{equation}
  \label{11}
\mu (F^\theta) = \int_{\Gamma_0} k_\mu(\eta) e(\theta;\eta)
  \lambda (d \eta), \qquad e(\theta;\eta):= \prod_{x\in
  \eta}\theta(x).
\end{equation}
For $\mu\in \mathcal{P}_{\rm exp}(\Gamma)$ and a compact $\Lambda$,
let $\mu^\Lambda$ be the corresponding projection. It is possible to
show that $\mu^\Lambda$, as a measure on $(\Gamma_\Lambda,
\mathcal{B}(\Gamma_\Lambda))$, is absolutely continuous with respect
to the Lebesgue-Poisson measure $\lambda$. Hence, we may write
\begin{equation}
  \label{12}
\mu^\Lambda ( d\eta) = R^\Lambda_\mu (\eta) \lambda (d\eta), \qquad
\eta \in \Gamma_\Lambda.
\end{equation}
For each compact $\Lambda$, the Radon-Nikodym derivative
$R^\Lambda_\mu$ and the correlation function $k_\mu$ satisfy
\begin{equation}
  \label{13}
k_\mu(\eta) = \int_{\Gamma_\Lambda} R^\Lambda_\mu (\eta \cup \xi)
\lambda ( d \xi), \qquad \eta \in \Gamma_\Lambda.
\end{equation}
For each $G\in B_{\rm bs}(\Gamma_0)$ and $k:\Gamma_0 \to \mathds{R}$
such that $k^{(n)}\in L^\infty ((\mathds{R}^d)^n)$ the integral
\begin{equation}
  \label{14}
  \langle \! \langle G, k \rangle \! \rangle := \int_{\Gamma_0}
  G(\eta) k(\eta) \lambda ( d \eta)
\end{equation}
surely exists. By (\ref{6b}), (\ref{9}), (\ref{12}) and (\ref{14})
we then obtain
\begin{equation}
  \label{15}
 \int_\Gamma \left(KG \right)(\gamma) \mu (d\gamma)  =  \langle \! \langle G, k_\mu \rangle \! \rangle
\end{equation}
holding for all $G\in B_{\rm bs}(\Gamma_0)$ and $\mu \in
\mathcal{P}_{\rm exp}(\Gamma)$. Set
\begin{equation}
  \label{16}
B_{\rm bs}^{\star}(\Gamma_0) = \{ G \in B_{\rm bs}(\Gamma_0): \left(
KG \right) (\gamma) \geq 0 \ \ {\rm for} \  \ {\rm all} \ \ \gamma
\in \Gamma\}.
\end{equation}
By \cite[Theorems 6.1, 6.2 and Remark 6.3]{Tobi} we know that the
following is true.
\begin{proposition}
  \label{1pn}
Let a measurable function $k: \Gamma_0 \to \mathds{R}$ have the
following properties:
\begin{eqnarray*}
&(a) & \ \ \langle \! \langle G, k_\mu \rangle \! \rangle \geq 0,
\qquad {\rm for} \ \ {\rm all} \ \  G \in B_{\rm
bs}^{\star}(\Gamma_0);
\\[.2cm]
 &(b) & \ \ k(\varnothing) =1; \\ &(c) & \ \ k(\eta) \leq
C^{|\eta|}, \qquad {\rm for} \ \ {\rm some} \ \ C>0.
\end{eqnarray*}
Then there exists a unique $\mu\in \mathcal{P}_{\rm exp}(\Gamma)$
such that $k$ is its correlation function.
\end{proposition}
Throughout the paper we use the following easy to check identities
holding for appropriate functions $g:\mathds{R}^d \to \mathds{R}$
and $G:\Gamma_0 \to \mathds{R}$:
\begin{equation}
  \label{17}
\forall x\in \gamma\qquad   \sum_{\eta \Subset \gamma}\prod_{z\in
\eta} g(z)= (1+ g(x))\sum_{\eta \Subset \gamma\setminus
x}\prod_{z\in \eta} g(z),
\end{equation}
\begin{equation}
  \label{18}
  \int_{\Gamma_0} \sum_{\xi \subset \eta} G(\xi, \eta, \eta
  \setminus \xi) \lambda ( d\eta)= \int_{\Gamma_0} \int_{\Gamma_0} G(\xi, \eta\cup \xi,
  \eta) \lambda ( d\xi) \lambda ( d\eta).
\end{equation}

\subsection{The model}
As mentioned above, the model which we consider in this work is
described by the generator given in (\ref{2}). Its entries are
subject to the following
\begin{assumption}
  \label{ass1}
The nonnegative measurable $a$, $b$ and $m$ satisfy:
\begin{itemize}
\item[(i)] $a$ is integrable and bounded; hence, we may set $$\sup_{x\in \mathds{R}^d}a(x) =
a^*, \qquad \int_{ \mathds{R}^d}a(x) dx = \langle a \rangle.$$
\item[(ii)]
There exist positive $r$ and $a_*$ such that $a(x) \geq a_*$
whenever $|x|\leq r$.
\item[(iii)] For each $x\in \mathds{R}^d$, $b(x|y_1, y_2) d y_1 d y_2$ is a symmetric finite measure on $(\mathds{R}^d)^2$; hence, we may set
\[
\langle b \rangle = \int_{(\mathds{R}^d)^2} b(x|y_1, y_2) d y_1 d
y_2,
\]
where, for simplicity, we consider the translation invariant case.
The mentioned symmetry means that $b(x|y_1, y_2) = b(x|y_2, y_1)$.
\item[(iv)]
The function $$\beta (y_1 - y_2) = \int_{\mathds{R}^d} b(x|y_1 ,
y_2) d x
$$ is supposed
to be such that $\sup_{x\in \mathds{R}^d}\beta(x) =:
\beta^*<\infty$. By the translation invariance it follows that
$$\int_{ \mathds{R}^d}\beta(x) dx = \langle b \rangle.$$
\end{itemize}
\end{assumption}
Noteworthy, we do not exclude the case where $b$ is a distribution.
For instance, by setting $$b(x|y_1, y_2)= \frac{1}{2} \left(\delta
(x-y_1) + \delta (x-y_2) \right)\beta (y_1-y_2),$$ we obtain the
Bolker-Pacala model  \cite{KK1} as a particular case of our model.
\begin{remark}
  \label{1rk}
The function $\beta$ describes the dispersal of siblings, which
compete with each other. As in the Bolker-Pacala model, here the
following situations may occur:
\begin{itemize}
  \item \emph{short dispersal:} there exists $\omega >0$
  such that $a(x)
   \geq \omega \beta(x)$ for all $x\in \mathds{R}^d$;
  \item \emph{long dispersal:} for each $\omega >0$, there
  exists $x\in \mathds{R}^d$ such that $ a(x)
  < \omega \beta(x)$.
\end{itemize}
\end{remark}
For $\eta \in \Gamma_0$, we set, cf. (\ref{3}),
\begin{eqnarray}
  \label{19}
E^a(\eta) & = & \sum_{x\in \eta} E^a(x, \eta\setminus x) =
\sum_{x\in
\eta} \sum_{y\in \eta\setminus x} a(x-y), \\[.2cm] \nonumber
E^b(\eta) & = & \sum_{x\in \eta} \sum_{y\in \eta\setminus x}
\beta(x-y) = \sum_{x\in \eta} \sum_{y\in \eta\setminus x}
\int_{\mathds{R}^d} b(z|x, y) d z.
\end{eqnarray}
The properties mentioned in (ii) and (iv) of Assumption \ref{ass1}
imply the following fact, proved in \cite[Lemma 3.1]{KKa}. For the
reader convenience, we repeat the proof in Appendix below.
\begin{proposition}
  \label{2pn}
There exist $\omega>0$ and $\upsilon\geq 0$ such that the following
holds
\begin{equation}
\label{2pnN} \upsilon |\eta| + E^a(\eta) \geq \omega E^b(\eta),
\qquad \eta\in \Gamma_0.
\end{equation}
\end{proposition}
The inequality in (\ref{2pnN}) can be rewritten in the form
\begin{equation}
\Phi_{\omega}(\eta):= \sum_{x\in \eta} \sum_{y\in \eta\setminus x} \left[ a(x-y)-\omega \int_{\mathbb{R}^d}b(z|x,y)dz \right] \ge -\upsilon |\eta|.
\label{fi}
\end{equation}
\begin{proposition}
\label{pfi} Assume that (\ref{fi}) holds for some $\omega_0>0$ and
$\upsilon_0>0$. Then for each $\omega < \omega_0$, it holds also for
$\upsilon=\upsilon_0\omega/\omega_0$.
\end{proposition}
\begin{proof}
For $\omega \in [0, \omega_0]$ by adding and subtracting $\frac{\omega}{\omega_0}E^a(\eta)$
we obtain
$$\Phi_\omega(\eta)=\frac{\omega}{\omega_0}
\left[ \left(\frac{\omega_0}{\omega}-1 \right)
E^a(\eta)+\Phi_{\omega_0}(\eta) \right]\ge -\frac{\omega}{\omega_0}\upsilon_0|\eta|.$$
\end{proof}

\section{The Evolution of States of the Finite System}

Here we assume that the initial state in (\ref{4}) has the property
$\mu_0(\Gamma_0)=1$, i.e., the system in $\mu_0$ is finite. Then the
evolution will be constructed in the Banach space of signed measures
with bounded variation, where the generator $L^*$ can be defined as
an unbounded linear operator and $C_0$-semigroup techniques can be
applied.

\subsection{The statement}

As just mentioned, we will solve (\ref{4}) in the Banach space
$\mathcal{M}$ of all signed measures on
$(\Gamma_0,\mathcal{B}(\Gamma_0))$ with bounded variation. Let
$\mathcal{M}^{+}$ stand for the cone of positive elements of
$\mathcal{M}$. By means of the Hahn-Jordan decomposition $\mu =
\mu^{+} - \mu^{-}$, $\mu^{\pm}\in \mathcal{M}^{+}$,  the norm of
$\mu\in \mathcal{M}$ is set to be $\|\mu\|_{\mathcal{M}}= \mu^{+}
(\Gamma_0) + \mu^{-}(\Gamma_0)$. Then $\mathcal{P}(\Gamma_0)$ is a
subset of $ \mathcal{M}^{+}$. The linear functional
$\varphi_{\mathcal{M}}(\mu) := \mu(\Gamma_0)= \mu^{+}(\Gamma_0) -
\mu^{-}(\Gamma_0)$ has the property $\varphi_{\mathcal{M}}(\mu) =
\|\mu\|_{\mathcal{M}}$ for each $\mu\in\mathcal{M}^{+}$. That is,
$\|\cdot \|_{\mathcal{M}}$ is additive on the cone $\mathcal{M}^{+}$
and hence $\mathcal{M}$ is an $AL$-space, cf. \cite{TV}.

For a strictly increasing function $\chi: \mathds{N}_0 \to [0,
+\infty)$, we set
\begin{equation}
  \label{18a}
\mathcal{M}_\chi =\left\{ \mu \in \mathcal{M}: \int_{\Gamma_0}\chi(
|\eta|) \mu^{\pm} (d \eta)<\infty\right\}, \qquad
\mathcal{M}^{+}_\chi = \mathcal{M}_\chi \cap \mathcal{M}^{+},
\end{equation}
and introduce
\begin{eqnarray}
  \label{18b}
  \varphi_{\mathcal{M}_\chi} (\mu) =  \int_{\Gamma_0} \chi(|\eta|) \mu^{+}
(d \eta)- \int_{\Gamma_0} \chi( |\eta|) \mu^{-} (d \eta), \qquad \mu
\in \mathcal{M}_\chi.
\end{eqnarray}
Note that $\mathcal{M}_\chi$ is a proper subset of $\mathcal{M}$ and
the corresponding embedding is continuous. Set, cf. Assumption
\ref{ass1} and (\ref{19}),
\begin{equation}
  \label{22}
\Psi(\eta) = M(\eta) + E^a (\eta) + \langle b \rangle |\eta|, \qquad
M(\eta):= \sum_{x\in \eta} m(x) \leq m^* |\eta|,
\end{equation}
and then
\begin{equation}
  \label{22a}
 \mathcal{D}= \left\{ \mu \in \mathcal{M}: \int_{\Gamma_0} \Psi(\eta)
 \mu^{\pm}(d\eta)<\infty\right\}.
\end{equation}
By (\ref{19}) we have that $\Psi(\eta)\leq C|\eta|^2$ for an
appropriate $C>0$; hence, $\mathcal{M}_{\chi_2} \subset
\mathcal{D}$, where $\chi_m (n) = (1+n)^m$, $m\in \mathds{N}$. Then,
for $\mu \in \mathcal{D}$, we define
\begin{equation}
  \label{22o}
(A\mu)(d\eta)  =  - \Psi (\eta) \mu(d\eta), \qquad (B\mu)(d\eta)  =
 \int_{\Gamma_0} \Xi (d \eta|\xi) \mu( d \xi),
\end{equation}
where the measure kernel $\Xi$ is
\begin{eqnarray}
  \label{22b}
\Xi (\mathbb{A}|\xi) &=& \sum_{x\in \xi} \left( m(x) + E^a (x, \xi\setminus x) \right)\mathds{1}_{\mathbb{A}}(\xi \setminus x) \\[.2cm] \nonumber
& + & \sum_{x\in \xi}\int_{(\mathds{R}^d)^2} b(x|y_1, y_2)
\mathds{1}_{\mathbb{A}} (\xi \setminus x\cup\{y_1 , y_2\}) d y_1 d
y_2, \qquad \mathbb{A} \in \mathcal{B}(\Gamma_0),
\end{eqnarray}
and $\mathds{1}_{\mathbb{A}}$ is the indicator of $\mathbb{A}$. Then
we set $L^* = A+B$. By direct inspection one checks that $L^*$
satisfies $\mu(LF) = (L^* \mu)(F)$ holding for all $\mu \in
\mathcal{D}$ and appropriate $F:\Gamma_0 \to [0, +\infty)$, see
(\ref{2}).

Along with $\chi_m$ defined above we also consider $\chi^\kappa (n)
:= e^{\kappa n}$, $\kappa >0$, and the space
$\mathcal{M}_{\chi^\kappa}$. By a global solution of (\ref{4}) in
$\mathcal{M}$ with $\mu_0 \in \mathcal{D}$ we understand a
continuous map $[0,+\infty) \ni t \mapsto \mu_t \in
\mathcal{D}\subset \mathcal{M}$, which is continuously
differentiable in $\mathcal{M}$ on $(0,+\infty)$ and is such that
both equalities in (\ref{4}) hold.
\begin{theorem}
  \label{1ftm}
The problem in (\ref{4}) with $\mu_0 \in \mathcal{D}$ has a unique
global solution $\mu_t\in\mathcal{M}$, which has the following
properties:
\begin{itemize}
  \item[{\it(a)}] for each $m\in \mathds{N}$, $\mu_t \in
  \mathcal{M}_{\chi_m}\cap\mathcal{P}(\Gamma_0)$ for all $t>0$
  whenever $\mu_0 \in
  \mathcal{M}_{\chi_m}\cap\mathcal{P}(\Gamma_0)$;
  \item[{\it (b)}] for each $\kappa>0$ and $\kappa' \in (0,
  \kappa)$,  $\mu_t \in \mathcal{M}_{\chi^{\kappa'}}\cap\mathcal{P}(\Gamma_0)$ for all $t\in (0,T(\kappa, \kappa'))$
  whenever $\mu_0 \in
  \mathcal{M}_{\chi^{\kappa}}\cap\mathcal{P}(\Gamma_0)$, where
  \begin{equation}
    \label{22T}
   T(\kappa, \kappa') = \frac{\kappa -
   \kappa'}{\langle b \rangle}e^{-\kappa}  ;
  \end{equation}
\item[{\it (c)}] for all $t>0$, $\mu_t (d\eta) = R_t (\eta) \lambda (d\eta)$ whenever
$\mu_0 (d\eta) = R_0(\eta) \lambda (d\eta)$.
\end{itemize}
\end{theorem}

\subsection{The proof}
To prove Theorem \ref{1ftm}, as well as to elaborate tools for
studying the evolution of infinite systems, we use the Thieme-Voigt
perturbation technique \cite{TV}, the basic elements of which we
present here in the form adapted to the context.

To prove claim (c) along with the space $\mathcal{M}$ we will
consider its subspace consisting of measures absolutely continuous
with respect to the Lebesgue-Poisson measure defined in (\ref{8}).
This is $\mathcal{R}:=L^1 (\Gamma_0 , d \lambda)$ in which we have a
similar functional $\varphi_{\mathcal{R}}(R) = \int_{\Gamma_0}
R(\eta) \lambda (d\eta)$. Then we define $\mathcal{R}^{+}$ and
$\mathcal{R}^{+}_1$ consisting of positive elements and probability
densities, respectively. Note that $\varphi_{\mathcal{R}}(R) =
\|R\|_{\mathcal{R}}$ for $R\in \mathcal{R}^{+}$ and hence
$\mathcal{R}$ is also and $AL$-space. For $\chi: \mathds{N}_0 \to
[0,+\infty)$ as in (\ref{18a}), we set
\begin{eqnarray}
  \label{18c}
& & \mathcal{R}_\chi  = \left\{ R\in \mathcal{R}: \int_{\Gamma_0}
\chi(|\eta|) |R(\eta)|\lambda (d \eta ) < \infty\right\}, \\[.2cm]
\nonumber & & \varphi_{\mathcal{R}_\chi}(R)  =  \int_{\Gamma_0}
\chi(|\eta|) R(\eta)\lambda (d \eta ), \qquad R\in \mathcal{R}_\chi,
\\[.2cm]
\nonumber & & \mathcal{R}_\chi^{+}  =   \mathcal{R}_\chi \cap
\mathcal{R}^{+}, \qquad \mathcal{R}_{\chi,1}^{+} = \{ R\in
\mathcal{R}_\chi^{+}: \varphi_{\mathcal{R}}(R) = 1\}.
\end{eqnarray}
Now let $\mathcal{E}$ be either $\mathcal{M}$ or $\mathcal{R}$, and
$\|\cdot \|_{\mathcal{E}}$ stand for the corresponding norm. The
sets $\mathcal{E}^{+}$, $\mathcal{E}_1^{+}$, $\mathcal{E}_\chi$,
$\mathcal{E}_\chi^{+}$, $\mathcal{E}_{\chi,1}^{+}$, and the
functionals $\varphi_{\mathcal{E}}$, $\varphi_{\mathcal{E}_\chi}$
are defined analogously, i.e., they should coincide with the
corresponding objects introduced above if  $\mathcal{E}$ is replaced
by $\mathcal{M}$ or $\mathcal{R}$ (by $\mathcal{M}_{1}^{+}$ we then
understand $\mathcal{P}(\Gamma_0)$). Let $\mathcal{D}\subset
\mathcal{E}$ be a linear subspace, $\mathcal{D}^{+} =\mathcal{D}\cap
\mathcal{E}^{+}$ and $(A,\mathcal{D})$, $(B,\mathcal{D})$ be
operators on $\mathcal{E}$. Set also $\mathcal{D}_\chi =\{ u\in
\mathcal{D}\cap \mathcal{E}_\chi: A u \in \mathcal{E}_\chi\}$ and
denote by $A_\chi$ the {\it trace} of $A$ in $\mathcal{E}_\chi$,
i.e., the restriction of $A$ to $\mathcal{D}_\chi$. Recall that a
$C_0$-semigroup of bounded linear operators $S=\{S(t)\}_{t\geq 0}$
in $\mathcal{E}$ is called \emph{positive} if
$S(t):\mathcal{E}^{+}\to \mathcal{E}^{+}$ for each $t\geq 0$. A
\emph{sub-stochastic} (resp. \emph{stochastic}) semigroup in
$\mathcal{E}$ is a positive $C_0$-semigroup such that
$\varphi_{\mathcal{E}} (S(t)u) \leq \varphi_{\mathcal{E}} (u)$
(resp. $\varphi_{\mathcal{E}} (S(t)u) = \varphi_{\mathcal{E}} (u)$)
whenever $u\in \mathcal{E}^{+}$.
\begin{proposition}\cite[Proposition 2.2]{TV}
  \label{TV0pn}
Let $(A,\mathcal{D})$ be the generator of a positive $C_0$-semigroup
in $\mathcal{E}$, and $(B,\mathcal{D})$ be positive, i.e.,
$B:\mathcal{D}^{+}\to \mathcal{E}^{+}$. Suppose also that
\begin{equation}
  \label{40}
\forall u\in\mathcal{D}^{+} \qquad  \varphi_{\mathcal{E}}((A+B)u)
\leq 0.
\end{equation}
Then, for each $r\in (0,1)$, the operator $(A+rB, \mathcal{D})$ is
the generator of a sub-stochastic semigroup in $\mathcal{E}$.
\end{proposition}
\begin{proposition}\cite[Proposition 2.7]{TV}
  \label{TVpn}
Assume that:
\begin{itemize}
  \item[(i)] $-A:\mathcal{D}^{+} \to \mathcal{E}^{+}$ and $B:\mathcal{D}^{+}
\to \mathcal{E}^{+}$;
\item[(ii)] $(A,\mathcal{D})$ be the generator of a
sub-stochastic semigroup $S=\{S(t)\}_{t\geq 0}$ on $\mathcal{E}$
such that $S (t):\mathcal{E}_\chi \to \mathcal{E}_\chi$ for all
$t\geq 0$ and the restrictions $S (t)|_{\mathcal{E}_\chi}$
constitute a $C_0$-semigroup on $\mathcal{E}_{\chi}$ generated by
$(A_\chi, \mathcal{D}_\chi)$;
\item[(iii)] $B:\mathcal{D}_\chi \to \mathcal{E}_\chi$ and
$ \varphi_{\mathcal{E}} \left( (A+B) u\right) = 0$,  for $u\in
\mathcal{D}^{+}$;
\item[(iv)] there exist $c>0$ and $\varepsilon >0$ such that
\[
\varphi_{\mathcal{E}_\chi} \left( (A+B) u\right) \leq c
\varphi_{\mathcal{E}_\chi} (u) - \varepsilon \|A u\|_{\mathcal{E}},
\qquad {\rm for}  \ \ u\in \mathcal{D}_\chi \cap \mathcal{E}^{+}.
\]
\end{itemize}
Then the closure of $(A+B,\mathcal{D})$ in $\mathcal{E}$ is the
generator of a stochastic semigroup $S_{\mathcal{E}}=
\{S_{\mathcal{E}}(t)\}_{t\geq 0}$ on $\mathcal{E}$ which leaves
$\mathcal{E}_\chi$ invariant. The restrictions
$S_{\mathcal{E}_\chi}(t):=S_{\mathcal{E}}(t)|_{\mathcal{E}_\chi}$,
$t\geq 0$, constitute a $C_0$-semigroup $S_{\mathcal{E}_\chi}$ on
$\mathcal{E}_\chi$ generated by the trace of the generator of
$S_{\mathcal{E}}$ in $\mathcal{E}_\chi$.
\end{proposition}
{\it Proof of Theorem \ref{1ftm}.} Along with $L^*=A+B$ defined in
(\ref{22a}) and (\ref{22b}) we consider the operator in
$\mathcal{R}$ defined according to the rule $ (L^* \mu) (d\eta) =
(L^\dagger R_\mu) (\eta) \lambda ( d\eta)$. Then $L^\dagger =
A^\dagger + B^\dagger$ with
\begin{eqnarray}
  \label{L}
(A^\dagger R)(\eta) & = & - \Psi (\eta) R(\eta), \\[.2cm] \nonumber
(B^\dagger R)(\eta) & = & \int_{\mathds{R}^d} \left(m(x) +
E^a(x,\eta) \right) R(\eta \cup x) d x \\[.2cm] \nonumber & + &
\int_{\mathds{R}^d}  \sum_{y_1\in \eta}\sum_{y_2 \in \eta \setminus
y_1} b(x|y_1 , y_2) R(\eta \cup x \setminus \{y_1 , y_2\}) d x,
\end{eqnarray}
the domain of which is, cf. (\ref{22a}),
\begin{equation}
  \label{L2}
\mathcal{D}^\dagger = \left\{ R \in \mathcal{R}: \int_{\Gamma_0}
\Psi(\eta) |R(\eta)| \lambda ( d \eta) < \infty \right\}.
\end{equation}
For $R\in \mathcal{D}^\dagger \cap \mathcal{R}^{+}$, by (\ref{18})
and (\ref{22}) we obtain from (\ref{L})
\begin{eqnarray}
  \label{L4}
\varphi_{\mathcal{R}} ( B^\dagger R) & = & \int_{\Gamma_0}
\left(\sum_{x\in \eta} [m(x) + E^a(x,\eta \setminus x)]  \right)
R(\eta)\lambda (d \eta) \\[.2cm] \nonumber & + & \int_{\Gamma_0}
\left(\sum_{x\in \eta} \int_{(\mathds{R}^d)^2} b(x|y_1 , y_2) d y_1
d y_2 \right)
R(\eta)\lambda (d \eta) \\[.2cm] \nonumber & = & \int_{\Gamma_0}
\Psi (\eta) R(\eta)\lambda (d \eta) = - \varphi_{\mathcal{R}}(
A^\dagger R).
\end{eqnarray}
By (\ref{L2}) and (\ref{L4}) we then get that: (a) $B^\dagger :
\mathcal{D}^\dagger \to \mathcal{R}$ and $B^\dagger :
\mathcal{R}^{+} \cap \mathcal{D}^\dagger \to \mathcal{R}^{+}$; (b)
$\varphi_{\mathcal{R}} ((A^\dagger + B^\dagger)R) =0$ for each $R\in
\mathcal{R}^{+} \cap \mathcal{D}^\dagger$. In the same way, we prove
that the operators defined in (\ref{22a}) and (\ref{22o}) satisfy:
(a) $B : \mathcal{D} \to \mathcal{M}$ and $B : \mathcal{D}^{+} \to
\mathcal{M}^{+}$; (b) $\varphi_{\mathcal{M}} ((A + B)\mu) =0$ for
each $\mu \in \mathcal{D}^{+}$. Thus, both pairs $(A, \mathcal{D})$,
$(B, \mathcal{D})$ and $(A^\dagger, \mathcal{D}^\dagger)$,
$(B^\dagger, \mathcal{D}^\dagger)$ satisfy item (i) of Proposition
\ref{TVpn}. We proceed further by setting
\begin{eqnarray}
  \label{L3}
(S(t) \mu) (d \eta) & = & \exp\left(- t \Psi(\eta)\right) \mu(
d\eta), \quad \mu \in \mathcal{M}, \quad t>0, \\[.2cm] \nonumber (S^\dagger(t) R) (\eta)& = & \exp\left(- t \Psi(\eta)\right)
R( \eta), \quad R\in \mathcal{R}.
\end{eqnarray}
Obviously, $S=\{S(t)\}_{t\geq 0}$ and
$S^\dagger=\{S^\dagger(t)\}_{t\geq 0}$ are sub-stochastic semigroups
on $\mathcal{M}$ and $\mathcal{R}$, respectively. They are generated
respectively by $(A, \mathcal{D})$ and $(A^\dagger,
\mathcal{D}^\dagger)$. Clearly, the restrictions
$S(t)|_{\mathcal{M}_\chi}$ and $S^\dagger(t)|_{\mathcal{R}_\chi}$
constitute positive $C_0$-semigroups for $\chi_m$ and $\chi^\kappa$
as in Theorem \ref{1ftm}. Likewise, $B:\mathcal{D}_\chi\to
\mathcal{M}_\chi$ and $B^\dagger:\mathcal{D}^\dagger_\chi\to
\mathcal{R}_\chi$. Thus, the conditions in items (ii) and (iii) of
Proposition \ref{TVpn} are satisfied in both cases.

Now we turn to item (iv) of Proposition \ref{TVpn}. By (\ref{18b})
we have
\begin{eqnarray*}
\varphi_{\mathcal{M}_\chi} ((A+B)\mu) & = &
\varphi_{\mathcal{M}_\chi}
(L^* \mu) = \int_{\Gamma_0} (LF_\chi)(\eta) \mu(d\eta), \quad F_\chi(\eta):= \chi(|\eta|), \qquad \\[.2cm]
\nonumber \varphi_{\mathcal{R}_\chi} ((A^\dagger+B^\dagger)R) & = &
\varphi_{\mathcal{R}_\chi} (L^\dagger R) = \int_{\Gamma_0} (L
F_\chi)(\eta) R(\eta) \lambda (d\eta).
\end{eqnarray*}
Then the condition in item (iv) is satisfied if, for some positive
$c$ and $\varepsilon$ and all $\eta$, the following holds
\begin{equation}
  \label{L5a}
  (L F_\chi)(\eta) + \varepsilon \Psi (\eta) \leq c \chi(|\eta|).
\end{equation}
For $\chi_m (n) = (1+n)^m$, $m\in \mathds{N}$, by (\ref{2}) we have,
cf. (\ref{22}),
\begin{eqnarray}
  \label{L6}
(L F_{\chi_m})(\eta) & = & - \left( M(\eta) + E^a (\eta)\right)
\epsilon_m (|\eta|) + \langle b \rangle |\eta| \epsilon_{m}
(|\eta|+1), \\[.2cm] \nonumber \epsilon_m (n) & := & (n+1)^m - n^m =
(n+1)^{m-1} + (n+1)^{m-2}n + \cdots + n^{m-1} \\[.2cm] \nonumber &
\leq & m(n+1)^{m-1}.
\end{eqnarray}
For $\chi^\kappa(n) = e^{\kappa n}$, we have
\begin{eqnarray*}
(L F_{\chi^\kappa})(\eta)  =  - \left( M(\eta) + E^a (\eta)\right)
e^{\kappa |\eta|} (1- e^{-1}) + \langle b \rangle |\eta| e^{\kappa
|\eta|} (e-1).
\end{eqnarray*}
By (\ref{L6}) the condition in (\ref{L5a}) takes the form
\begin{equation}
  \label{L8}
 - \left( M(\eta) + E^a (\eta)\right)\left(
\epsilon_m (|\eta|) -\varepsilon\right) + \langle b \rangle |\eta|
\left(\epsilon_{m} (|\eta|+1) + \varepsilon\right) \leq c
\left(|\eta|+1 \right)^m.
\end{equation}
since $\epsilon_m (|\eta|) \geq 1$. For $\varepsilon < 1$, the
validity of (\ref{L8}) will follow whenever $c$ satisfies
\[
c \geq m \langle b \rangle \left( 2^{m-1} + 1\right).
\]
Hence, for $\chi=\chi_m$, all the conditions of Proposition
\ref{TVpn} are met for both choices of $\mathcal{E}$ and the
corresponding operators. Therefore, we have two semigroups:
$S_{\mathcal{M}}$ and $S_{\mathcal{R}}$, with the properties
described in the mentioned statement. Then $\mu_t = S_\mathcal{M}(t)
\mu_0$ is the unique solution of the Fokker-Planck equation with
$\mu_0 \in \mathcal{D}$, which proves claim (a) of Theorem
\ref{1tm}. At the same time, $R_t = S_{\mathcal{R}}(t) R_0(\eta )$
is the unique solution of
\begin{equation}
  \label{L9}
\dot{R}_t = L^\dagger R_t, \qquad R_t|_{t=0} = R_{\mu_0}\in
\mathcal{D}^\dagger.
\end{equation}
By (\ref{L2}) we have that $R_{\mu_0}\in \mathcal{D}^\dagger$ and
$\mu_0\in \mathcal{D}$ are equivalent. By direct inspection one
checks that $\mu_t(d \eta) = R_t (\eta) \lambda (d \eta)$ solves
(\ref{4}) if $R_t$ solves (\ref{L9}). Then the unique solution
$\mu_t = S_{\mathcal{M}}(t) \mu_0$  of (\ref{4}) has the mentioned
form, which proves claim (c).

To complete the proof we fix $\kappa >0$ and consider the trace of
$A$ in $\mathcal{M}_{\chi^\kappa}$, cf. (\ref{22o}), defined on the
domain
\begin{equation*}
\mathcal{D}_{\kappa}:=\left\{ \mu \in \mathcal{M}_{\chi^\kappa}:
\int_{\Gamma_0} \Psi (\eta) e^{\kappa |\eta|}\mu^{\pm }(d\eta ) <
\infty\right\}.
\end{equation*}
First, we split $B$ into the sum $B_1 + B_2$, where for $\mathbb{A}
\in
 \mathcal{B}( \Gamma_0)$ we set, cf.
(\ref{22b}),
\begin{equation}
  \label{L11}
 (B_1 \mu)(\mathbb{A}) = \int_{\Gamma_0} \left( \sum_{x\in \eta}[m(x) + E^a(x, \eta \setminus
 x)] \mathds{1}_{\mathbb{A}} (\eta\setminus x)\right) \mu(d\eta) ,
\end{equation}
and
\begin{equation}
  \label{L12}
(B_2 \mu)(\mathbb{A}) = \int_{\Gamma_0} \left(\sum_{x\in \eta}
\int_{(\mathds{R}^d)^2} b(x|y_1 , y_2) \mathds{1}_{\mathbb{A}} (\eta
\setminus x \cup\{y_1 , y_2\}) d y_1  d y_ 2\right) \mu( d\eta).
\end{equation}
For $\mu\in \mathcal{D}_\kappa^{+} :=\mathcal{D}_\kappa \cap
\mathcal{M}^{+}$,  from (\ref{L11}) we have
\begin{eqnarray}
  \label{L13}
\varphi_{\mathcal{M}_{\chi^\kappa}} (B_1 \mu) & = & \int_{\Gamma_0}
 e^{\kappa |\xi|} \int_{\Gamma_0} \sum_{x\in \eta}[m(x) +
E^a(x, \eta \setminus
 x)]\delta_{\eta\setminus x}( d \xi) \mu(d\eta) \\[.2cm] \nonumber
 & = & \int_{\Gamma_0} e^{\kappa (|\eta|-1)} \left(M(\eta) +
E^a(\eta) \right) \mu(d\eta) \\[.2cm] \nonumber
 & \leq & - e^{-\kappa} \varphi_{\mathcal{M}_{\chi^\kappa}} ( A \mu)
.
\end{eqnarray}
For $r = e^{-\kappa}$, by (\ref{L13}) we have that
$\varphi_{\mathcal{M}_{\chi^\kappa}} (A+ r^{-1} B_1 \mu)\leq 0$ for
each $\mu \in \mathcal{D}_\kappa^{+}$. Then by Proposition
\ref{TV0pn} we obtain that $(A+ B_1, \mathcal{D}_\kappa)$ generates
a sub-stochastic semigroup $S_\kappa$ on
$\mathcal{M}_{\chi^\kappa}$. For $\kappa'\in (0,\kappa)$, let us
show now that $B_2$ acts as a bounded linear operator from
$\mathcal{M}_{\chi^\kappa}$ to $\mathcal{M}_{\chi^{\kappa'}}$. In
view of the Hahn-Jordan decomposition, it is enough to consider the
action of $B_2$ on positive elements of $\mathcal{M}_{\chi^\kappa}$.
Since $B_2$ is positive, cf. (\ref{L12}), for $\mu\in
\mathcal{M}^{+}_{\chi^\kappa}$, we have
\begin{eqnarray}
  \label{L14}
\|B_2 \mu\|_{\mathcal{M}_{\chi^{\kappa'}}} & = & \int_{\Gamma_0}
e^{\kappa'|\xi|} \int_{\Gamma_0} \sum_{x\in
\eta}\int_{(\mathds{R}^d)^2} b(x|y_1, y_2)  \delta_{\eta\setminus x
\cup\{y_1 , y_2\}} (d \xi) dy_1 dy_2 \mu ( d\eta) \qquad \\[.2cm]
\nonumber & = & e^{\kappa'} \int_{\Gamma_0} e^{\kappa'|\eta|}
\sum_{x\in \eta}\int_{(\mathds{R}^d)^2} b(x|y_1, y_2)
 dy_1 dy_2 \mu (
d\eta) \\[.2cm]
\nonumber & = & e^{\kappa'} \langle b \rangle \int_{\Gamma_0} |\eta|
e^{- (\kappa-\kappa')|\eta|} e^{\kappa|\eta|} \mu ( d\eta) \\[.2cm]
\nonumber & \leq & \frac{e^{\kappa'} \langle b \rangle}{e (\kappa -
\kappa')} \|\mu\|_{\mathcal{M}_{\chi^{\kappa}}}.
\end{eqnarray}
Let $(B_2)_{\kappa'\kappa}: \mathcal{M}^{+}_{\chi^\kappa} \to
\mathcal{M}^{+}_{\chi^{\kappa'}}$ be the operator as just described.
For $n\in \mathds{N}$, we set
\begin{equation}
  \label{L15}
\kappa_l = \kappa - (\kappa - \kappa')l /n, \qquad l=0, 1, \dots ,
n.
\end{equation}
By means of (\ref{L14}) and (\ref{L15}) we then estimate of the
operator norm
\begin{equation}
  \label{L16}
\|(B_2)_{\kappa_{l+1}\kappa_l}\| \leq \frac{ e^{\kappa} n \langle b
\rangle}{e (\kappa - \kappa')}.
\end{equation}
Next, for $t>0$ and $0\leq t_n \leq \cdots \leq t_0 = t$, we
consider the following bounded linear operator acting from
$\mathcal{M}_{\chi^\kappa}$ to $\mathcal{M}_{\chi^{\kappa'}}$
\begin{equation*}
T_{\kappa' \kappa}^{(n)} (t,t_1, t_2 , \dots , t_n) =
S_{\kappa_n}(t-t_1) (B_2)_{\kappa_n \kappa_{n-1}}
S_{\kappa_{n-1}}(t_1-t_2) \cdots (B_2)_{\kappa_1 \kappa}
S_{\kappa}(t_n),
\end{equation*}
where $S_{\kappa_{l}}$ is the sub-stochastic semigroup in
$\mathcal{M}_{\chi^{\kappa_l}}$ generated by $(A+B_1,
\mathcal{D}_{\kappa_l})$. By the latter fact we have that
$T_{\kappa' \kappa}^{(n)} (t,t_1, t_2 , \dots , t_n):
\mathcal{M}_{\chi^\kappa}\to \mathcal{D}_{\kappa'}$ and
\begin{eqnarray}
  \label{L17a}
 \frac{d}{dt} T_{\kappa' \kappa}^{(n)} (t,t_1, t_2 , \dots , t_n) & =
&  (A+ B_1)T_{\kappa' \kappa}^{(n)} (t,t_1, t_2 , \dots ,
t_n),\\[.2cm] \nonumber T_{\kappa' \kappa}^{(n)} (t,t, t_2 , \dots ,
t_n) & = & (B_2)_{\kappa' \kappa_{n-1}} T_{\kappa_{n-1}
\kappa}^{(n-1)} (t, t_2 , \dots , t_n).
\end{eqnarray}
As $(B_2)_{\kappa' \kappa_{n-1}}$ is the restriction of $(B_2,
\mathcal{D}_{\kappa'})$ to $\mathcal{M}_{\chi^{\kappa_{n-1}}}
\subset \mathcal{D}_{\kappa'}$ and $T_{\kappa' \kappa}^{(n-1)}
(t,t_2, t_2 , \dots , t_n): \mathcal{M}_{\chi^\kappa}\to
\mathcal{D}_{\kappa'}$, the second line in (\ref{L17a}) can be
rewritten as
\begin{equation}
  \label{L17b}
T_{\kappa' \kappa}^{(n)} (t,t, t_2 , \dots , t_n)  = B_2 T_{\kappa'
\kappa}^{(n-1)} (t, t_2 , \dots , t_n).
\end{equation}
On the other hand, since all the semigroups $S_{\kappa_{l}}$ are
sub-stochastic and $(B_2)_{\kappa' \kappa}$ are positive, by
(\ref{L16}) we get the following estimate of its operator norm
\begin{equation}
  \label{L18}
\|T^{(n)}_{\kappa' \kappa} (t,t_1, t_2 , \dots , t_n)\| \leq \left(
\frac{ e^{\kappa} n \langle b \rangle}{e (\kappa -
\kappa')}\right)^n.
\end{equation}
We also set $T^{(0)}_{\kappa' \kappa}(t)= S_{\kappa'}
(t)|_{\mathcal{M}_{\chi^\kappa}}$, and then consider
\begin{equation}
  \label{L19}
Q_{\kappa'\kappa}(t) := \sum_{n=0}^\infty \int_0^t\int_0^{t_1}
\cdots \int_0^{t_{n-1}} T^{(n)}_{\kappa'\kappa} (t,t_1, t_2 , \dots
, t_n) d t_n d t_{n-1} \cdots d t_1.
\end{equation}
By (\ref{L18}) we conclude that the series in (\ref{L19}) converges
uniformly on compact subsets of $[0, T(\kappa, \kappa'))$, see
(\ref{22T}), to a continuously differentiable function
\[
(0,T(\kappa, \kappa')) \ni t \mapsto Q_{\kappa'\kappa}(t) \in
\mathcal{L}(\mathcal{M}_{\chi^\kappa},
\mathcal{M}_{\chi^{\kappa'}}),
\]
where the latter is the Banach space of all bounded linear operators
acting from $\mathcal{M}_{\chi^\kappa}$ to
$\mathcal{M}_{\chi^{\kappa'}}$. By (\ref{L17a}) and (\ref{L17b}) we
obtain
\begin{equation}
  \label{L20}
\frac{d}{dt} Q_{\kappa'\kappa}(t) = (A + B_1 + B_2)
Q_{\kappa'\kappa}(t) = L^* Q_{\kappa'\kappa}(t).
\end{equation}
Thus, assuming that $\mu_0 \in \mathcal{M}_{\chi^\kappa}$ we get
that $\tilde{\mu}_t := Q_{\kappa'\kappa}(t) \mu_0$, for $t \in
[0,T(\kappa, \kappa'))$, lies in $\mathcal{M}_{\chi^{\kappa'}}$ and
solves (\ref{4}). Therefore, $\tilde{\mu}_t$ coincides with $\mu_t =
S_{\mathcal{M}}(t)\mu_0$, which completes the proof.
  {\hfill$\square$}

\section{The Evolution of States of the Infinite System: Posing}
\label{Sec3}
 In this section, we begin to construct the evolution of
states $\mu_0\to \mu_t$ assuming that the system in $\mu_0$ is
infinite and hence the method developed in Sect. 3 does not work
anymore. Instead, we will obtain $\mu_0\to \mu_t$ from the evolution
$B_0 \to B_t$, where $B_0(\theta)=\mu_0 (F^\theta)$ and $\mu_0\in
\mathcal{P}_{\rm exp} (\Gamma)$, see Definition \ref{0df}. In view
of (\ref{11}), the evolution $B_0 \to B_t$ can be constructed as the
evolution of correlation functions. The latter will be performed in
the following three steps: (a) constructing $k_0\to k_t$ for $t< T$
(for some $T<\infty$) (Sect. \ref{Sec4}); (b) proving that $k_t$ is
the correlation function of a unique $\mu_t \in \mathcal{P}_{\rm
exp}(\Gamma)$ (Sect. \ref{Sec5}); (c) continuing $k_t$ to all $t>0$
(Sect. \ref{Sec6}).

To make the first step, we derive from (\ref{1}) the corresponding
evolution equation with the operator $L^\Delta$ obtained from
(\ref{2}) by (\ref{17}), (\ref{18}) and the following rule
\begin{equation}
  \label{20}
\mu(L F^\theta) = \int_{\Gamma_0} (L^\Delta k_\mu)(\eta) e(\theta;
\eta) \lambda (d\eta).
\end{equation}
Then we prove that the equation $\dot{k}_t = L^\Delta k_t$ has a
unique solution $k_t$, $t<T$, in a scale of Banach spaces such that
$k_t^{(n)}$ satisfies (\ref{6c}) with $\varkappa$ dependent on $t$.
The restriction $t<T$ arises from the proof as no direct semigroup
method can be applied here. The proof just mentioned does not
guarantee that the solution $k_t$ is a correlation function, and
even its usual positivity is not certain. Step (b) is made by
constructing suitable approximations $k_t^{\rm app}$ to the
mentioned solution $k_t$. By this construction $k_t^{\rm app}$
satisfies condition (a) of Proposition \ref{1pn}. Then we prove
that, for all $G\in B_{\rm bs}(\Gamma_0)$,  $\langle \! \langle G,
k_t^{\rm app} \rangle \! \rangle$ converges to $\langle \! \langle
G, k_t \rangle \! \rangle$ as the approximations are eliminated.
This yields that also $k_t$ satisfies condition (a) of Proposition
\ref{1pn}. The remaining conditions (b) and (c) are checked
directly. Then $k_t = k_{\mu_t}$ for a unique $\mu_t\in
\mathcal{P}_{\rm exp}(\Gamma)$. This also implies the usual
positivity of $k_t$ which is then used to obtain the continuation to
all $t>0$.

\subsection{The operators} To make the first step mentioned above
we calculate $L^\Delta$ according to (\ref{20}) and obtain it in the
following form
\begin{eqnarray}
  \label{21}
L^\Delta & = & A_1^\Delta + A_2^\Delta + B_1^\Delta + B_2^\Delta, \\[.2cm] \nonumber
(A_1^\Delta k)(\eta) & = & - \Psi(\eta) k(\eta),\\[.2cm] \nonumber
( A_2^\Delta k)(\eta) & = & \int_{\mathds{R}^d} \sum_{y_1\in
\eta}\sum_{y_2\in \eta\setminus y_1} k(\eta \cup x \setminus \{y_1,
y_2\}) b(x|y_1 , y_2) d x,\\[.2cm] \nonumber
(B_1^\Delta k)(\eta) & = & - \int_{\mathds{R}^d} k(\eta \cup x)
E^a(x, \eta) d x,\\[.2cm] \nonumber
(B_2^\Delta k)(\eta) & = & 2 \int_{(\mathds{R}^d)^2} \sum_{y_1 \in
\eta}k(\eta \cup x \setminus y_1) b(x|y_1 , y_2) d y_2 d x,
\end{eqnarray}
where $\Psi$ is as in (\ref{22}). Since the correlation functions of
measures from $\mathcal{P}_{\rm exp}(\Gamma)$ satisfy (\ref{6c}), we
introduce
\begin{equation}
\label{nk} \|k \|_{\alpha} = \esssup_{\eta \in \Gamma_0}e^{-\alpha
|\eta|} |k(\eta)|, \qquad \alpha \in \mathds{R},
\end{equation}
and the corresponding $L^\infty$-like Banach spaces
\begin{equation}
  \label{23}
  \mathcal{K}_\alpha = \{k:\Gamma_0 \to \mathds{R}: \|k\|_\alpha
  <\infty\}.
\end{equation}
For $\alpha' < \alpha$, we have that $\| k\|_{\alpha'} \ge \|
k\|_{\alpha}$. Therefore, $\mathcal{K}_{\alpha'} \hookrightarrow
\mathcal{K}_{\alpha}$, where ``$\hookrightarrow$'' denotes
continuous embedding. Thus, $\{\mathcal{K}_\alpha\}_{\alpha \in
\mathds{R}}$ is an ascending scale of Banach spaces.

Our aim now is to define linear operators which act as in
(\ref{21}), cf. (\ref{22}). First, for a given $\alpha \in
\mathds{R}$, we define an unbounded operator $(L^\Delta_\alpha,
\mathcal{D}_\alpha^\Delta)$, where
\begin{equation}
 \label{24}
 \mathcal{D}_\alpha^\Delta = \{ k \in \mathcal{K}_\alpha: \Psi k \in  \mathcal{K}_\alpha\}.
\end{equation}
Thus, $A_1^\Delta$ maps $\mathcal{D}_\alpha^\Delta$ to
$\mathcal{K}_\alpha$. Furthermore, for each $k\in
\mathcal{D}_\alpha^\Delta$, one finds $C>0$ such that
$(1+\Psi(\eta)) |k(\eta)| \leq e^{\alpha |\eta|} C$. We apply this
fact and item (iv) of Assumption \ref{ass1} to get
\begin{gather*}
\left\vert(A^\Delta_2 k)(\eta) \right\vert \leq \frac{Ce^{-\alpha +
\alpha|\eta|}}{1 + \Psi(\eta)} \sum_{y_1\in \eta} \sum_{y_2 \in
\eta\setminus y_1} \beta (y_1 - y_2) \leq C\beta^* e^{-\alpha +
\alpha|\eta|},
\end{gather*}
which means that $A_2^\Delta:\mathcal{D}_\alpha^\Delta \to
\mathcal{K}_\alpha$. In a similar way, we prove that
$B_i^\Delta:\mathcal{D}_\alpha^\Delta \to \mathcal{K}_\alpha$,
$i=1,2$. Thus, the expression in (\ref{21}) defines
$(L^\Delta_\alpha, \mathcal{D}_\alpha^\Delta)$. By the inequality
\begin{equation}
 \label{25}
n^p e^{-\sigma n} \le \left( \frac{p}{e\sigma}\right)^p , \qquad
p\ge 1, \quad \sigma>0, \quad n\in \mathds{N},
\end{equation}
one readily proves that
\begin{equation}
 \label{26}
 \forall \alpha' < \alpha \qquad \mathcal{K}_{\alpha'} \subset \mathcal{D}^\Delta_\alpha.
\end{equation}
The next step is to introduce bounded operators $L_{\alpha
\alpha'}^{\Delta}: \mathcal{K}_{\alpha'} \to \mathcal{K}_\alpha$. To
this end, by means of (\ref{25}) and the inequality $|k(\eta) | \le
e^{\alpha |\eta|} \|k\|_{\alpha}$ (see (\ref{nk})), for $\alpha' <
\alpha$ we obtain from (\ref{21}) the following estimate
\begin{eqnarray}
  \label{27}
\|A_1^\Delta k\|_\alpha & \leq & \esssup_{\eta \in \Gamma_0}
e^{-\alpha
|\eta|}\Psi (\eta) |k(\eta)| \\[.2cm] \nonumber  & \leq &  \bigg{(}  ( m^* + \langle b \rangle +
a^*)\esssup_{\eta\in \Gamma_0}\left[|\eta|^2 e^{-(\alpha-
\alpha')|\eta|} \right] \bigg{)} \|k\|_{\alpha'} \\[.2cm] \nonumber  & =
& \frac{4 ( m^* + \langle b \rangle + a^*) }{e^2(\alpha-\alpha')^2}
\|k\|_{\alpha'}.
\end{eqnarray}
In a similar way, one estimates $\|A_2^\Delta k\|_\alpha$ and
$\|B_i^\Delta k\|_\alpha$, $i=1,2$, which then yields, cf.
(\ref{21}),
\begin{equation}
  \label{28}
\|L^\Delta k\|_\alpha \leq \left(4\frac{ m^* + \langle b \rangle +
a^* + \beta^* e^{-\alpha'} }{e^2(\alpha-\alpha')^2} + \frac{\langle
a \rangle e^{\alpha'} + 2 \langle b \rangle}{e(\alpha-\alpha')}
\right)\|k\|_{\alpha'}.
\end{equation}
Then we define a bounded operator $L^\Delta_{\alpha \alpha'}:
\mathcal{K}_{\alpha'} \to \mathcal{K}_\alpha$, the norm of which is
estimated by means of (\ref{28}). In view of (\ref{26}), we have
that each $k\in \mathcal{K}_{\alpha'}$ lies in
$\mathcal{D}^\Delta_\alpha$, and
\begin{equation}
  \label{29}
L^\Delta_{\alpha \alpha'} k = L^\Delta_{\alpha }k.
\end{equation}
In the sequel, we consider two types of operators with the action as
in (\ref{21}): (a) unbounded operators $(L^\Delta_\alpha,
\mathcal{D}(L^\Delta_\alpha))$, $\alpha\in \mathds{R}$, with the
domains as in (\ref{24}); (b) bounded operators $L^\Delta_{ \alpha
\alpha'}$ just described. These operators are related to each other
by (\ref{29}), i.e., $L^\Delta_{\alpha\alpha'}$ can be considered as
the restriction of $L^\Delta_{\alpha }$ to $\mathcal{K}_{\alpha'}$.

\subsection{The statements}

For $\alpha \in \mathds{R}$, we set, cf. (\ref{15}), (\ref{16}) and
Proposition \ref{1pn},
\begin{equation}
  \label{32}
\mathcal{K}^\star_\alpha = \{ k \in \mathcal{K}_\alpha:
k(\varnothing) =1 \ {\rm and} \  \langle \! \langle G, k \rangle \!
\rangle \geq 0 \  {\rm for}  \ {\rm all}  \ G\in B^\star_{\rm bs}
(\Gamma_0) \}.
\end{equation}
Note that
\begin{equation}
  \label{32a}
\mathcal{K}^\star_\alpha \subset \mathcal{K}_{\alpha}^{+} :=\{ k\in
\mathcal{K}_\alpha: k(\eta ) \geq 0\}.
\end{equation}
Since the spaces defined in (\ref{23}) form an ascending scale, we
have that $k\in \mathcal{K}_{\alpha_0}$ lies in all
$\mathcal{K}_\alpha$ with $\alpha>\alpha_0$. Recall that the model
parameters satisfy Assumption \ref{ass1} which, in particular, imply
the validity of Proposition \ref{2pn}.
\begin{theorem}
  \label{1tm}
There exists $c\in \mathds{R}$ dependent on the model parameters
only such that, for each $\mu_0\in \mathcal{P}_{\rm exp}(\Gamma_0)$,
there exists a unique map $[0,+\infty) \ni t \mapsto k_t \in
\mathcal{K}^{\star}_{\alpha_t}$ with $\alpha_t = \alpha_0 + ct$ and
$\alpha_0> - \log \omega$ such that  $k_0=k_{\mu_0}\in
\mathcal{K}^\star_{\alpha_0}$, which has the following properties:
\begin{itemize}
  \item[(i)]
For each $T>0$ and all $t\in [0,T)$, the map $$ [0,T)\ni t \mapsto
k_t \in \mathcal{K}_{\alpha_t}  \subset
 \mathcal{D}(L^\Delta_{\alpha_T}) \subset \mathcal{K}_{\alpha_T}$$
is continuous on $[0,T)$ and continuously differentiable on $(0,T)$
in $\mathcal{K}_{\alpha_T}$.
\item[(ii)] For all $t\in (0,T)$ it satisfies
\begin{equation*}
  \dot{k}_t = L^\Delta_{\alpha_T} k_t.
\end{equation*}
\end{itemize}
\end{theorem}
\begin{corollary}
  \label{Jaco}
Let $k_t\in \mathcal{K}^\star_{\alpha_t}$, $t\geq 0$, be as in
Theorem \ref{1tm}, and then $\mu_t\in\mathcal{P}_{\rm exp}(\Gamma)$
be the measure corresponding to this $k_t$ according to Proposition
\ref{1pn}. Then the map $t \mapsto \mu_t$ is such that
\begin{itemize}
\item[1.] for each compact $\Lambda$ and  $t\geq 0$, $\mu_t^{\Lambda}$ lies in the
  domain $\mathcal{D}\subset \mathcal{M}$ defined in (\ref{22a});
 \item[2.]  for
  each $\theta \in \varTheta$, the map $[0,+\infty) \ni t \mapsto \mu_t (F^\theta)$ is continuous and
continuously differentiable on $(0,+\infty)$ and the following
holds, cf. (\ref{Jan}),
\begin{equation}
  \label{Ja}
 \frac{d}{dt} \mu_t (F^\theta) = (L^* \mu_t^{\Lambda_\theta}) (F^\theta) =
 \langle\!\langle e(\theta, \cdot), L^\Delta_{ \alpha_T}
 k_t\rangle\!\rangle,
\end{equation}
where the latter equality holds for all $T>t$, see (\ref{11}) and
(\ref{14}).
\end{itemize}
\end{corollary}
The proof of these statements is done in the remainder of the paper.
Its main steps are: (a) constructing the evolution $k_{\mu_0}\to
k_t$ for $t<T$ for some $T<\infty$; (b) proving that $k_t$ belongs
to $\mathcal{K}^\star_\alpha$ with an appropriate $\alpha$, that by
Proposition \ref{1pn} will allow us to associate $k_t$ with a unique
$\mu\in \mathcal{P}_{\rm exp}(\Gamma)$; (c) proving that $k_t$ lies
in $\mathcal{K}_{\alpha_t}$ on the mentioned time interval, which
will be used to continue $k_t$ to all $t>0$.

\section{The solution on a bounded time interval}

\label{Sec4}

Here we make step (a) of the program formulated at the end of Sect.
\ref{Sec3}.

\subsection{The statement}

Let us fix some $\alpha_1\in \mathds{R}$, take $\alpha_2 >\alpha_1$
and consider the following Cauchy problem in
$\mathcal{K}_{\alpha_2}$
\begin{equation}
 \label{33}
 \dot{k}_t = L^\Delta_{\alpha_2} k_t , \qquad k_t|_{t=0} = k_0 \in \mathcal{K}_{\alpha_1}.
\end{equation}
By its solution on a time interval $[0, T)$ we mean a continuous (in
$\mathcal{K}_{\alpha_2}$) map  $[0, T)\ni t \mapsto k_t\in
\mathcal{D}^\Delta_{\alpha_2}$, which is continuously differentiable
on $(0, T)$ and satisfies both equalities in (\ref{33}). For
$\alpha, \alpha'\in \mathds{R}$ such that $\alpha'< \alpha$ and for
$\upsilon\geq 0$ as in Proposition \ref{2pn}, we set
\begin{equation}
 \label{34}
 T(\alpha, \alpha') = \frac{\alpha - \alpha'}{2 \langle b \rangle + \upsilon + \langle a \rangle e^{\alpha}}.
\end{equation}
\begin{lemma}
 \label{1lm}
Let $\omega$ and $\upsilon$ be as in Proposition \ref{2pn}. Then for
each $\alpha_1 > - \log \omega$ and an arbitrary $k_0 \in
\mathcal{K}_{\alpha_1}$, the problem in (\ref{33}) has a unique
solution $k_t\in \mathcal{D}^\Delta_{\alpha_2}$ on the time interval
$[0, T(\alpha_2, \alpha_1))$.
\end{lemma}
In contrast to the case of finite configurations described in
Theorem \ref{1ftm}, the construction of a $C_0$-semigroup that
solves (\ref{33}) is rather hopeless. In view of this, the proof of
Lemma \ref{1lm} will be done in the following steps:
\begin{itemize}
  \item[(i)] the operator $L^\Delta$ will be written in the form
  $L^\Delta = A^\Delta_\upsilon + B^\Delta_\upsilon$, see
  (\ref{45}), in such a way that
   $A^\Delta_\upsilon:=A^\Delta_{1,\upsilon} + A^\Delta_{2}$ can be used to
construct a certain (sun-dual) $C_0$-semigroup in
$\mathcal{K}_{\alpha_2}$;
\item[(ii)] this semigroup and $B^\Delta_{\upsilon} := B^\Delta_1 +
B^\Delta_{2,\upsilon}$, see (\ref{46}), will be used to construct
the family of operators $\{Q_{\alpha \alpha'} (t): t\in [0,
T(\alpha, \alpha'))\}$, see (\ref{34}) and Lemma \ref{3lm}, such
that $Q_{\alpha \alpha'} (t)\in \mathcal{L}(\mathcal{K}_{\alpha'},
\mathcal{K}_\alpha)$ and $k_t = Q_{\alpha_2 \alpha_1}(t)k_0$ is the
solution in question. $\mathcal{L}(\mathcal{K}_{\alpha'},
\mathcal{K}_\alpha)$ stands for the Banach space of all bounded
operators acting from $\mathcal{K}_{\alpha'}$ to
$\mathcal{K}_{\alpha}$.

\end{itemize}

\subsection{The predual semigroup} Here we make the first step in constructing the
semigroup mentioned in item (i) above. For $\alpha \in \mathbb{R}$,
the space predual to $\mathcal{K}_\alpha$ is
\begin{equation}
 \label{35}
\mathcal{G}_\alpha := L^1(\Gamma_0, e^{\alpha|\cdot|}d \lambda),
\end{equation}
which for $\alpha>0$ coincides with $\mathcal{R}_\chi$ defined in
(\ref{18c}) with $\chi(n) = e^{\alpha n}$. Here, however, we allow
$\alpha$ to be any real number. The norm in $\mathcal{G}_\alpha$ is
\begin{equation}
 \label{36}
|G|_\alpha=\int_{\Gamma_0}|G(\eta)|e^{\alpha |\eta|}\lambda(d \eta).
 \end{equation}
Clearly,  $|G|_{\alpha'} \le |G|_{\alpha}$ whenever
$\alpha'<\alpha$. Then $\mathcal{G}_{\alpha} \hookrightarrow
\mathcal{G}_{\alpha'}$, and this embedding is also dense. In order
to use Proposition \ref{2pn} we modify the operators introduced in
(\ref{21}) by adding and subtracting the term $\upsilon |\eta|$.
This will lead also to the corresponding reconstruction of the
predual operators. For an appropriate $G:\Gamma_0 \to \mathds{R}$,
set, cf. (\ref{22}),
\begin{eqnarray}
 \label{37}
(A_{1,\upsilon}G)(\eta) & = & - \Psi_\upsilon (\eta) G(\eta) = - \left(\upsilon |\eta| + E^a(\eta) + M(\eta) + \langle b \rangle |\eta| \right)
G(\eta), \\[.2cm] \nonumber
(A_2 G)(\eta)& = &\sum_{x \in \eta} \int_{(\mathbb{R})^2}G(\eta
\setminus x \cup y_1 \cup y_2)b(x|y_1,y_2)dy_1 dy_2,\\[.2cm]
\nonumber \mathcal{D}_\alpha & = & \{ G:\in \mathcal{G}_\alpha:
\Psi_\upsilon G \in \mathcal{G}_\alpha\}.
\end{eqnarray}
By Proposition \ref{2pn} we have that
\begin{equation}
  \label{37J}
\Psi_\upsilon (\eta) \geq \omega E^b(\eta).
\end{equation}
The operator $(A_{1,\upsilon}, \mathcal{D}_\alpha)$ is the generator
of the semigroup $S_{0,\alpha} = \{S_{0,\alpha}\}_{t\geq 0}$ of
multiplication operators which act in $\mathcal{G}_\alpha$ as
follows, cf. (\ref{L3}),
\begin{equation}
 \label{38}
( S_{0,\alpha}(t) G)(\eta) =
\exp\left(- t \Psi_\upsilon (\eta) \right)  G(\eta).
\end{equation}
Let $\mathcal{G}_\alpha^{+}$ be the  cone of positive elements of
$\mathcal{G}_\alpha $ The semigroup defined in (\ref{38}) is
obviously \emph{sub-stochastic}. Set $\mathcal{D}_\alpha^{+} =
\mathcal{D}_\alpha \cap \mathcal{G}_\alpha ^{+}$. By (\ref{18}),
(\ref{36}) and (\ref{37}) we get
\begin{eqnarray}
\label{39}
|A_2G|_\alpha & = & \int_{\Gamma_0} e^{\alpha|\eta|}|(A_2G)(\eta)|\lambda(d \eta)  \\[.2cm]
& \leq & \int_{\Gamma_0 } e^{\alpha|\eta|} \int_{(\mathbb{R}^d)^2}
\sum_{x\in \eta} |G(\eta \setminus x \cup y_1
\cup y_2)| b(x|y_1,y_2) dy_1dy_2 \lambda(d \eta)\nonumber\\[.2cm]
& = & \int_{\Gamma_0} \int_{\mathbb{R}^d} \sum_{y_1 \in
\eta}\sum_{y_2 \in \eta \setminus y_1}e^{\alpha (|\eta|-1)}|G(\eta)|
b(x|y_1,y_2) dx \lambda(d \eta) \nonumber\\[.2cm] \nonumber
& = & e^{-\alpha} \int_{\Gamma_0} e^{\alpha|\eta|}E_b(\eta)
|G(\eta)|\lambda(d \eta) \leq (e^{-\alpha}/ \omega)
|A_{1,\upsilon}G|_{\alpha}.
\end{eqnarray}
The latter estimate follows by  (\ref{37J}), see also (\ref{19}).
\begin{lemma}
  \label{2lm}
Let $\upsilon$ and $\omega$ be as in Proposition \ref{2pn} and
$A_{1,\upsilon}$, $A_2$ and $\mathcal{D}_\alpha$ be as in
(\ref{37}). Then for each $\alpha> - \log \omega$, the operator
$(A_{\upsilon} , \mathcal{D}_\alpha):=(A_{1,\upsilon} + A_2,
\mathcal{D}_\alpha)$ is the generator of a sub-stochastic semigroup
$S_\alpha=\{S_\alpha(t)\}_{t\geq 0}$ on $\mathcal{G}_\alpha$.
\end{lemma}
\begin{proof}
We apply Proposition \ref{TV0pn} with  $\mathcal{E}=
\mathcal{G}_\alpha$, $\mathcal{D}=\mathcal{D}_\alpha$ and $A=
A_{1,\upsilon}$. For some $r\in (0, 1)$, we set $B = r^{-1} A_2$,
which is clearly positive. By (\ref{39}) $B$ is defined on
$\mathcal{D}_\alpha$. To show that (\ref{40}) holds we take $G\in
\mathcal{D}_\alpha^{+}$ and proceed as in (\ref{39}). That is,
\begin{eqnarray*}
& & \int_{\Gamma_0} \left( (A_{1,\upsilon} +r^{-1}
A_2)G\right)(\eta) e^{\alpha|\eta|}\lambda ( d \eta)  =  -
\int_{\Gamma_0} \Psi_\upsilon (\eta) G(\eta) e^{\alpha|\eta|}
\lambda ( d\eta) \\[.2cm]\nonumber &  &+ r^{-1}
\int_{\Gamma_0} \sum_{x\in \eta} \int_{(\mathds{R}^d)^2} G(\eta
\setminus x\cup \{y_1, y_2\})b(x|y_1 , y_2) e^{\alpha|\eta|} d y_1 d
y_2 \lambda (d\eta) \\[.2cm] \nonumber
& & \leq - \int_{\Gamma_0} \left(\upsilon |\eta| + E^a (\eta) -
r^{-1}e^{-\alpha}E^b (\eta) \right) G(\eta) e^{\alpha|\eta|} \lambda
(d \eta).
\end{eqnarray*}
Now, for $\alpha > - \log \omega$, we pick $r\in (0,1)$ in such a
way that $r^{-1}e^{-\alpha} \leq \omega$, which by Proposition
\ref{2pn} implies that (\ref{40}) holds for this choice. Then the
operator $A_{1,\upsilon} + r (r^{-1} A_2)$ satisfies Proposition
\ref{TV0pn} by which the proof follows.
\end{proof}
By the definition of the sub-stochasticity of $S_\alpha$ we have
that $|S_\alpha(t) G|_\alpha \leq |G|_\alpha$ whenever $G\in
\mathcal{G}_\alpha^{+}$. Let us show now that the same estimate
holds also for all $G\in \mathcal{G}_\alpha$. Each such $G$ in a
unique way can be decomposed $G=G^{+} - G^{-}$ with $G^{\pm} \in
\mathcal{G}_\alpha^{+}$. Moreover, by (\ref{36}) we have that
\[
|G|_\alpha = \int_{\Gamma_0} e^{\alpha |\eta|} \left( G^{+}(\eta) +
G^{-}(\eta)\right) \lambda (d\eta)= |G^{+}|_\alpha + |G^{-}|_\alpha.
\]
Then
\begin{eqnarray}
  \label{39J}
|S_\alpha (t) G|_\alpha & = & |S_\alpha (t) (G^{+}-G^{-})|_\alpha
\leq
|S_\alpha (t) G^{+}|_\alpha + |S_\alpha (t) G^{-}|_\alpha \\[.2cm]
\nonumber & \leq & |G^{+}|_\alpha + | G^{-}|_\alpha = |G|_\alpha .
\end{eqnarray}

\subsection{The sun-dual semigroup}

Let $S_\alpha(t)$ be an element of the semigroup as in Lemma
\ref{2lm}. Then its adjoint $S^*_\alpha(t)$ is a bounded linear
operator in $\mathcal{K}_\alpha$. Clearly, $\{S^*_\alpha(t)\}_{t\geq
0}$ is a semigroup. However, it is not strongly continuous and hence
cannot be directly used to construct (classical) solutions of
differential equations. This obstacle is usually circumvented as
follows, see \cite{P}. Set, cf. (\ref{14}),
\begin{equation*}
\mathcal{D}_\alpha^* = \{ k\in \mathcal{K}_\alpha: \exists \hat{k}
\in\mathcal{K}_\alpha \ \forall G \in \mathcal{D}_\alpha \ \langle
\! \langle A_\upsilon G, k\rangle \! \rangle = \langle \! \langle G,
\hat{k}\rangle \! \rangle\}.
\end{equation*}
Then the operator $(A^*_{\upsilon},\mathcal{D}_\alpha^*)$ is adjoint
to $(A_{\upsilon},\mathcal{D}_\alpha)$. It acts as follows
\begin{eqnarray*}
(A^*_{\upsilon} k)(\eta) & = & - \Psi_\upsilon (\eta) k(\eta)
\\[.2cm] \nonumber & + & \int_{\mathds{R}^d} \sum_{y_1 \in \eta}\sum_{y_2 \in \eta\setminus y_1
} k(\eta\cup x\setminus \{y_1,y_2\}) b(x|y_1 , y_2)  d x.
\end{eqnarray*}
By direct inspection one obtains that $\mathcal{K}_{\alpha'} \subset
\mathcal{D}_\alpha^*$ whenever $\alpha'< \alpha$. Let
$\mathcal{Q}_\alpha$ be the closure of $\mathcal{D}_\alpha^*$ in
$\mathcal{K}_\alpha$. Then we have
\begin{equation}
  \label{43}
  \mathcal{K}_{\alpha'}\subset \mathcal{D}_\alpha^* \subset
  \mathcal{Q}_\alpha \subsetneq \mathcal{K}_\alpha, \qquad
  \alpha'<\alpha.
\end{equation}
Now we set
\begin{equation*}
\mathcal{D}_\alpha^\odot= \{ k\in \mathcal{D}_\alpha^*: A_\upsilon^*
k \in \mathcal{Q}_\alpha\},
\end{equation*}
and denote by $A^\odot_\upsilon$ the restriction of $A_\upsilon^*$
to $\mathcal{D}_\alpha^\odot$. Then $(A^\odot_\upsilon,
\mathcal{D}_\alpha^\odot)$ is the generator of a $C_0$-semigroup,
which we denote by $S^\odot_\alpha=\{S^\odot_\alpha (t)\}_{t \geq
0}$. This is the semigroup which we have aimed to construct. It has
the following property, see \cite[Lemma 10.1]{P}.
\begin{proposition}
  \label{Papn}
for each $k\in \mathcal{Q}_\alpha$ and $t\geq 0$, it follows that
$\|S^\odot_\alpha (t) k\|_\alpha = \|S^*_\alpha(t)k\|_\alpha \leq
\|k\|_\alpha$. Moreover, for each $\alpha'< \alpha$ and $k\in
\mathcal{K}_{\alpha'}$, the map $[0,+\infty)\ni t \mapsto
S^\odot_\alpha (t) k\in \mathcal{Q}_\alpha$ is continuous.
\end{proposition}
The estimate $\|S^*_\alpha(t)k\|_\alpha \leq \|k\|_\alpha$ is
obtained by means of (\ref{39J}). The continuity follows by
(\ref{43}) and the fact that $S^\odot_\alpha$ is a $C_0$-semigroup.
\subsection{The resolving operators: proof of Lemma \ref{1lm}} Now
we construct the family of operators $\{Q_{\alpha \alpha'}(t)\}$
such that the solution of (\ref{33}) is obtained in the form $k_t =
Q_{\alpha_2 \alpha_1}(t) k_0$. This construction,  in which we
employ $S^\odot$, resembles the one used to get (\ref{L19}). We
begin by rearranging the operators in (\ref{21}) as follows
\begin{equation}
  \label{45}
L^\Delta = A^\Delta + B^\Delta = A^\Delta_\upsilon +
B^\Delta_\upsilon,
\end{equation}
where $A^\Delta_\upsilon = A^\Delta_{1,\upsilon} + A^\Delta_2$, see
(\ref{37}), and
\begin{eqnarray}
  \label{46}
  B^\Delta_\upsilon &= & B_1^\Delta+
  B^\Delta_{2,\upsilon},\\[.2cm]\nonumber
  (B^\Delta_{2,\upsilon}k)(\eta) &=& (B^\Delta_{2}k)(\eta) +
  \upsilon |\eta|k(\eta)\\[.2cm]\nonumber & = & 2
  \int_{(\mathds{R}^d)^2} \sum_{y_1\in \eta} b(x|y_1, y_2) k(\eta \cup
  x \setminus y_1) d x d y_2 + \upsilon |\eta|k(\eta),
\end{eqnarray}
whereas $B_1^\Delta$ is as in (\ref{21}). By means of (\ref{46}),
for $\alpha \in \mathds{R}$ and $\alpha' < \alpha$, we define
$(B^\Delta_\upsilon)_{\alpha\alpha'}\in
\mathcal{L}(\mathcal{K}_{\alpha'}, \mathcal{K}_{\alpha})$ the norm
of which can be estimated similarly as in (\ref{27}), (\ref{28}),
which yields
\begin{equation}
  \label{47}
 \| (B^\Delta_\upsilon)_{\alpha \alpha'}\| \leq \frac{2 \langle b \rangle + \upsilon + \langle a \rangle e^{\alpha'}}{e(\alpha
 -\alpha')}.
\end{equation}
Now let $\mathbf{B}$ be either $B^\Delta_\upsilon$ or
$B^\Delta_{2,\upsilon}$, and $\mathbf{B}_{\alpha \alpha'}$ be the
corresponding bounded operator. Then, cf. (\ref{47}),
\begin{equation}
  \label{48}
\|\mathbf{B}_{\alpha \alpha'}\| \leq
\frac{\varpi(\alpha;\mathbf{B})}{e(\alpha-\alpha')},
\end{equation}
where
\begin{equation}
  \label{49}
  \varpi(\alpha;B^\Delta_{\upsilon}) = 2 \langle b \rangle + \upsilon + \langle a \rangle
  e^{\alpha}, \quad \varpi(\alpha;B^\Delta_{2,\upsilon}) = 2 \langle b \rangle +
  \upsilon.
\end{equation}
For some $\alpha_1, \alpha_2$ such that $\alpha_1 < \alpha_2$,  we
then set $\Sigma_{\alpha_2 \alpha_1}(t) =
S^{\odot}_{\alpha_2}(t)|_{\mathcal{K}_{\alpha_1}}$, $t>0$, where
$S^{\odot}_{\alpha}$ is the sub-stochastic semigroup as in
Proposition \ref{Papn}. Let also $\Sigma_{\alpha_2 \alpha_1}(0)$ be
the embedding operator $\mathcal{K}_{\alpha_1}\to
\mathcal{K}_{\alpha_2}$. Hence, see Proposition \ref{Papn}, the
operator norm satisfies
\begin{equation}
  \label{50}
 \| \Sigma_{\alpha_2 \alpha_1}(t)\|\leq 1, \qquad t\geq 0.
\end{equation}
We also have
\begin{eqnarray}
  \label{51}
\Sigma_{\alpha_2 \alpha_1}(t) & = & \Sigma_{\alpha_2 \alpha_1}(0)
S^{\odot}_{\alpha_1}(t), \\[.2cm] \nonumber \Sigma_{\alpha_3 \alpha_1}(t+s) &
= & \Sigma_{\alpha_3 \alpha_2}(t)\Sigma_{\alpha_2 \alpha_1}(s),
\quad \ \alpha_3 > \alpha_2,
\end{eqnarray}
holding for all $t, s\geq 0$. Moreover,
\begin{equation*}
\frac{d}{dt} \Sigma_{\alpha_2 \alpha_1}(t) = A^\Delta_\upsilon
\Sigma_{\alpha_2 \alpha_1}(t),
\end{equation*}
which follows by Lemma \ref{2lm} and the construction of the
semigroup $S^{\odot}_\alpha$. Now we set
\begin{equation}
  \label{53}
  T(\alpha_2 , \alpha_1;\mathbf{B}) = \frac{\alpha_2 -
  \alpha_1}{\varpi(\alpha_2;\mathbf{B})},
\end{equation}
see (\ref{48}), (\ref{49}), and also
\begin{equation}
  \label{54}
\mathcal{A}(\mathbf{B})= \{ (\alpha_1 , \alpha_2, t): - \log \omega
< \alpha_1 < \alpha_2 , \ t\in[0,T(\alpha_2, \alpha_1;
\mathbf{B}))\}.
\end{equation}
Note that $T(\alpha_2 , \alpha_1;B^\Delta_\upsilon)$ coincides with
$T(\alpha_2 , \alpha_1)$ defined in (\ref{34}).
\begin{lemma}
  \label{3lm}
For both choices of $\mathbf{B}$, there exist the corresponding
families $\lbrace Q_{\alpha_2 \alpha_1}(t;\mathbf{B}): (\alpha_1,
\alpha_2,t) \in \mathcal{A}(\mathbf{B}) \rbrace$, each element of
which has the following properties:
\begin{itemize}
  \item[{\it(a)}] $Q_{\alpha_2 \alpha_1}(t;\mathbf{B}) \in \mathcal{L}(\mathcal{K}_{\alpha_1},
  \mathcal{K}_{\alpha_2})$;
\item[{\it(b)}] the map $[0, T(\alpha_2, \alpha_1;\mathbf{B})) \ni t \mapsto
Q_{\alpha_2 \alpha_1}(t;\mathbf{B})\in
\mathcal{L}(\mathcal{K}_{\alpha_1}, \mathcal{K}_{\alpha_2})$ is
continuous;
\item[{\it(c)}] the operator norm of $Q_{\alpha_2 \alpha_1}(t;\mathbf{B}) \in
\mathcal{L}(\mathcal{K}_{\alpha_1}, \mathcal{K}_{\alpha_2})$
satisfies
$$\|Q_{\alpha_2 \alpha_1}(t;\mathbf{B})  \| \le \frac{T(\alpha_2, \alpha_1;\mathbf{B})}{T(\alpha_2, \alpha_1;\mathbf{B})-t},$$
\item[{\it(d)}]  for each $\alpha_3 \in (\alpha_1, \alpha_2)$ and $t <T(\alpha_3,
\alpha_1;\mathbf{B})$, the following holds
\begin{equation}
  \label{54b}
\frac{d}{dt} Q_{\alpha_2 \alpha_1}(t;\mathbf{B})  =
((A^{\Delta}_\upsilon)_{\alpha_2 \alpha_3} + \mathbf{B}_{\alpha_2
\alpha_3})Q_{\alpha_3 \alpha_1}(t;\mathbf{B}),
\end{equation}
which yields, in turn, that
\begin{eqnarray}
  \label{54a}
\frac{d}{dt} Q_{\alpha_2 \alpha_1}(t;B^\Delta_\upsilon ) & = &
L^\Delta_{\alpha_2}Q_{\alpha_2 \alpha_1}(t;B^\Delta_\upsilon ) \\[.2cm]\frac{d}{dt} Q_{\alpha_2 \alpha_1}(t;B^\Delta_{2,\upsilon} ) & = &
((A^{\Delta}_\upsilon)_{\alpha_2} +
(B^\Delta_{2,\upsilon})_{\alpha_2} )Q_{\alpha_2
\alpha_1}(t;B^\Delta_{2,\upsilon} ), \nonumber
\end{eqnarray}
where $L^\Delta_{\alpha_2}$ is as in (\ref{33}), see also
(\ref{45}), and $(B^\Delta_{2,\upsilon})_{\alpha_2} $ denotes
$(B^\Delta_{2,\upsilon} , \mathcal{D}^\Delta_{\alpha_2})$, see
(\ref{24}).
\end{itemize}
\end{lemma}
\begin{proof}
Fix some $T< T(\alpha_2, \alpha_1;\mathbf{B})$ and then take $\alpha
\in (\alpha_1, \alpha_2]$  and positive $\delta < \alpha- \alpha_1$
such that
$$T< T_\delta:= \frac{\alpha - \alpha_1 - \delta}{\beta(\alpha_2;\mathbf{B})}.$$
Then take some $l\in \mathds{N}$ and divide $[\alpha_1, \alpha]$
into $2l+1$ subintervals in the following way: $\alpha_1=\alpha^0$,
$\alpha = \alpha^{2l+1}$ and
\begin{equation}
  \label{55}
  \alpha^{2s}=\alpha_1+\frac{s}{l+1}\delta + s \epsilon, \qquad \alpha^{2s+1}=\alpha_1+\frac{s+1}{l+1}\delta + s \epsilon,
\end{equation}
where $\epsilon = (\alpha - \alpha_1 - \delta)/l$ and $s=0, 1, ...,
l$. Now for $0\leq t_l \leq t_{l-1}\cdots \leq t_1 \leq t_0 := t$,
define
\begin{eqnarray}
  \label{56}
  \Pi^{(l)}_{\alpha \alpha_1}(t, t_1, t_2, ... , t_l;\mathbf{B}) & = & \Sigma_{\alpha \alpha^{2l}}(t-t_1)\mathbf{B}_{\alpha^{2l}\alpha^{2l-1}}
  \cdots \Sigma_{\alpha^{2s+1}
  \alpha^{2s}}(t_{l-s}-t_{l-s+1})\mathbf{B}_{\alpha^{2s}\alpha^{2s-1}}
  \qquad
 \nonumber \\[.2cm]  & \times & \Sigma_{\alpha^3 \alpha^{2}}(t_{l-1}-t_l)\mathbf{B}_{\alpha^{2}\alpha^{1}}\Sigma_{\alpha^1 \alpha_1}(t_l).
\end{eqnarray}
By the very construction we have that $ \Pi^{(l)}_{\alpha
\alpha_1}(t, t_1, t_2, ... , t_l;\mathbf{B}) \in
\mathcal{L}(\mathcal{K}_\alpha , \mathcal{K}_{\alpha_1})$, and the
map
$$(t,t_1,...,t_l) \mapsto \Pi^{(l)}_{\alpha \alpha_1}(t, t_1, t_2, ... , t_l;\mathbf{B})$$
is continuous (Proposition \ref{Papn} and the fact that each
$\mathbf{B}_{\alpha^{2s}\alpha^{2s-1}}$ is bounded). Moreover, by
(\ref{50}) and (\ref{48})  we have
\begin{eqnarray}
  \label{56a}
\| \Pi^{(l)}_{\alpha \alpha_1}(t, t_1, t_2, ... , t_l;\mathbf{B})\|
& \leq &  \prod_{s=0}^l \|\mathbf{B}_{\alpha^{2s}\alpha^{2s-1}} \|
\leq  \prod_{s=0}^l
\frac{\varpi(\alpha^{2s};\mathbf{B})}{e(\alpha^{2s}-\alpha^{2s-1})}
\\[.2cm] \nonumber &  \leq & \left(\frac{l\upsilon(\alpha_2;\mathbf{B})}{e(\alpha - \alpha_1-\delta)}
\right)^l \leq \left(\frac{l}{eT_\delta}\right)^l.
\end{eqnarray}
By (\ref{51}) we also have that
\[
\Sigma_{\alpha^{2s+1} \alpha^{2s}}(t_{l-s}-t_{l-s+1})=
\Sigma_{\alpha^{2s+1}
\alpha^{2s}}(0)S^\odot_{\alpha^{2s}}(t_{l-s}-t_{l-s+1}). \] Taking
the derivative of both sides of the latter we obtain
\begin{eqnarray*}
\frac{d}{dt}\Sigma_{\alpha^{2s+1}
\alpha^{2s}}(t)=(A_\upsilon^{\Delta})_{\alpha^{2s+1} \alpha''}
\Sigma_{\alpha'' \alpha^{2s}}(t) =
(A_\upsilon^{\Delta})_{\alpha^{2s+1}} \Sigma_{\alpha^{2s+1}
\alpha^{2s}}(t),
\end{eqnarray*}
holing for each $\alpha''\in (\alpha^{2s}, \alpha^{2s+1})$. Here
$(A_\upsilon^{\Delta})_{\alpha}$ stands for the unbounded operator
defined in (\ref{37}). Then we obtain from (\ref{56})  the following
\begin{eqnarray}
  \label{57}
\frac{d}{dt}  \Pi^{(l)}_{\alpha \alpha_1}(t, t_1, t_2, ... ,
t_l;\mathbf{B}) & = & (A^{\Delta}_\upsilon)_{\alpha \alpha'}
\Pi^{(l)}_{\alpha' \alpha_1}(t, t_1, t_2, ... , t_l;\mathbf{B})
\\[.2cm]\nonumber & = & (A^{\Delta}_\upsilon)_{\alpha}
\Pi^{(l)}_{\alpha \alpha_1}(t, t_1, t_2, ... , t_l;\mathbf{B}).
\end{eqnarray}
Now we set
\begin{equation}
  \label{58}
Q_{\alpha \alpha_1}(t;\mathbf{B}) = \Sigma_{\alpha \alpha_1}(t)+
\sum_{l=1}^{\infty}\int_{0}^{t}
\int_0^{t_1}...\int_0^{t_{l-1}}\Pi^{(l)}_{\alpha \alpha_1}(t, t_1,
t_2, ... , t_l;\mathbf{B})dt_l...dt_1.
\end{equation}
By (\ref{56a}) the series in (\ref{58}) converges uniformly of
compact subsets of $[0,T_\delta)$, which proves claims (a) and (b).
The estimate in (c) follows directly from (\ref{56a}). Finally,
(\ref{54a}) follows by (\ref{57}), cf. (\ref{L20}).\end{proof}

By solving (\ref{54b}) with the initial condition $Q_{\alpha_2
\alpha_1}(t+s;\mathbf{B})|_{t=0} = Q_{\alpha_2
\alpha_1}(s;\mathbf{B})$ we obtain the following `semigroup'
property of the family $\lbrace Q_{\alpha_2 \alpha_1}(t;\mathbf{B}):
(\alpha_1, \alpha_2,t) \in \mathcal{A}(\mathbf{B}) \rbrace$.
\begin{corollary}
  \label{JKco}
For each $\alpha\in (\alpha_1, \alpha_2)$ and $t,s>0$ such that
\[
s< T(\alpha, \alpha_1;\mathbf{B}), \quad t < T(\alpha_2,
\alpha;\mathbf{B}), \quad t+s < T(\alpha_2, \alpha_1;\mathbf{B}),
\]
the following holds
\[
Q_{\alpha_2 \alpha_1}(t+s;\mathbf{B}) = Q_{\alpha_2
\alpha}(t;\mathbf{B})Q_{\alpha\alpha_1}(s;\mathbf{B}).
\]
\end{corollary}
\begin{remark}
 \label{Jan10rk}
Since $B^\Delta_{2,\upsilon}$ is positive, by (\ref{56}) we obtain
that $Q_{\alpha_2
\alpha_1}(t;B^\Delta_{2,\upsilon}):\mathcal{K}_{\alpha_1}^+
\to\mathcal{K}_{\alpha_2}^+$.  This positivity will be used to
continue $k_t$ to all $t>0$. It is the only reason for us to use
$Q_{\alpha_2 \alpha_1}(t;B^\Delta_{2,\upsilon})$ since
$B^\Delta_{\upsilon}$ is not positive, and hence the positivity of
$Q_{\alpha_2 \alpha_1}(t;B^\Delta_{\upsilon})$ cannot be secured.
\end{remark}

\noindent {\it Proof of Lemma \ref{1lm}.} Set
\begin{equation}
  \label{Jan3}
Q_{\alpha_2 \alpha_1}(t) = Q_{\alpha_2
\alpha_1}(t;B^\Delta_\upsilon), \qquad t < T(\alpha_2 ,
\alpha_1;B^\Delta_\upsilon) = T(\alpha_2 , \alpha_1)
\end{equation}
Then the solution in question is obtained by setting $k_t =
Q_{\alpha_2 \alpha_1}(t) k_0$, which definitely satisfies (\ref{33})
by (\ref{54a}) and (\ref{51}). Its uniqueness can be proved as in
the proof of Lemma 4.8 in \cite{KK}. \hfill{$\square$}

Before proceeding further, we prove some corollary of Lemma \ref{3lm}
related to the predual evolution in $\mathcal{G}_\alpha$, see
(\ref{35}). Let $S_\alpha$ be the semigroup as in Lemma \ref{2lm}.
For $\alpha'
>\alpha$, let $S_{\alpha \alpha'}(t)$ be the restriction of
$S_\alpha(t)$ to $\mathcal{G}_{\alpha'} \hookrightarrow
\mathcal{G}_\alpha$. Along with the operators defined in (\ref{37})
we consider the predual operators to $B^\Delta_\upsilon$,  see
(\ref{21}) and (\ref{46}). That is, they act
\begin{eqnarray*}
(B_1 G) (\eta)& =&  - \sum_{x\in \eta} G(\eta\setminus x) E^a (x,
\eta\setminus x), \\[.2cm] \nonumber
(B_{2,\upsilon} G) (\eta)& =& 2 \int_{(\mathds{R}^d)^2} \sum_{x\in
\eta} G(\eta\setminus x \cup y_1) b(x|y_1, y_2) d y_1 d y_2 +
\upsilon |\eta| G(\eta).
\end{eqnarray*}
By means of these expressions we can define bounded operators acting
from $\mathcal{G}_{\alpha}$ to $\mathcal{G}_{\alpha'}$ for $\alpha'<
\alpha$. It turns out that the estimate of the norm is exactly as in (\ref{47}),
that is,
\begin{equation*}
\|(B_\upsilon)_{\alpha' \alpha} \|= \frac{2 \langle b \rangle +
\upsilon + \langle a \rangle e^{\alpha'}}{e(\alpha - \alpha')}.
\end{equation*}
Recall that $\mathcal{A}(B^\Delta_\upsilon)$ is defined in
(\ref{54}). For $(\alpha_2, \alpha_1, t)\in
\mathcal{A}(B^\Delta_\upsilon)$, let $T<T(\alpha_2, \alpha_1)$ be
fixed. Pick  $\alpha\in [\alpha_1,\alpha_2)$ and $\delta<
\alpha_2-\alpha$ such that $T< T(\alpha_2, \alpha +\delta)$. Then,
for some $l\in \mathds{N}$, set, cf. (\ref{55}),
\begin{equation*}
\alpha_{2s} = \alpha_2 - \frac{s}{l+1}\delta - s \epsilon, \quad
\alpha^{2s+1} = \alpha_2 - \frac{s+1}{l+1}\delta - s \epsilon,
\end{equation*}
where $\epsilon = (\alpha_2 - \alpha - \delta)/ l$. For $0\leq
t_l\leq \cdots \leq t_1 \leq t_0:=t$ we then define, cf. (\ref{56}),
\begin{eqnarray*}
\Omega^{(l)}_{\alpha\alpha_2} (t, t_1 , \dots , t_n) & = & S_{\alpha
\alpha^{2l}} (t-t_1) (B_\upsilon)_{\alpha^{2l}\alpha^{2l-1}}
S_{\alpha^{2l-1}\alpha^{2l-2}} (t_1-t_2) \times \\[.2cm] & \times &
S_{\alpha^3\alpha^2}(t_{l-1}- t_l)
(B_\upsilon)_{\alpha^{2}\alpha^{1}} S_{\alpha^1\alpha_2}(t_{l}) .
\nonumber
\end{eqnarray*}
Set
\begin{equation}
  \label{Du4}
H_{\alpha \alpha_2} (t) = S_{\alpha\alpha_2}(t) + \sum_{l=1}^\infty
\int_0^t \int_0^{t_1}\cdots \int_0^{t_{l-1}}
\Omega^{(l)}_{\alpha\alpha_2} (t, t_1 , \dots , t_n) d t_l d t_{l-1}
\cdots d t_1.
\end{equation}
Then exactly as in the case of Lemma \ref{3lm} we prove the
following statement.
\begin{proposition}
  \label{Du1pn}
Each member of the family of operators $\{H_{\alpha\alpha_2}(t):
(\alpha_2 , \alpha, t)\in \mathcal{A}(B^\Delta_\upsilon)\}$ defined
in (\ref{Du4}) has the following properties:
\begin{itemize}
  \item[(a)] $H_{\alpha\alpha_2}(t) \in
\mathcal{L}(\mathcal{G}_{\alpha_2}, \mathcal{G}_{\alpha})$, the
operator norm of which satisfies
\[
\|H_{\alpha\alpha_2}(t)\| \leq
\frac{T(\alpha_2,\alpha)}{T(\alpha_2,\alpha)-t};
\]
\item[(b)] For each $k\in \mathcal{K}_\alpha$ and $G\in
\mathcal{G}_{\alpha_2}$, it follows that
\begin{equation}
  \label{Du5}
\langle \! \langle G, Q_{\alpha_2 \alpha}(t) k\rangle \!\rangle =
\langle \! \langle H_{\alpha \alpha_2}(t) G, k\rangle \!\rangle.
\end{equation}
\end{itemize}

\end{proposition}

\section{The Identification Lemma}

\label{Sec5}

Our aim now is to prove that the solution obtained in Lemma
\ref{1lm} has the property $k_t=k_{\mu_t}$ for a unique $\mu_t \in
\mathcal{P}_{\rm exp}(\Gamma)$. We call this \emph{identification}
since it allows us to identify the mentioned solutions as the
correlation functions of sub-Poissonian states.

Recall that $\upsilon$ and $\omega$ appear in Proposition \ref{2pn}
and $\mathcal{K}^\star_\alpha$ is defined in (\ref{32}).
\begin{lemma}[Identification]
  \label{ILlm}
For each $\alpha_2>\alpha_1>-\log \omega$, it follows that  $Q_{\alpha_2 \alpha_1}(t)=Q_{\alpha_2
\alpha_1}(t;B_\upsilon^\Delta): \mathcal{K}_{\alpha_1}^\star \to
\mathcal{K}_{\alpha_2}^\star$ for all $t \in [0, \tau (\alpha_2,
\alpha_1)]$ with $\tau(\alpha_2, \alpha_1) = T(\alpha_2, \alpha_1)/3$.
\end{lemma}
The proof consists in the following steps:
\begin{itemize}
  \item[(i)] constructing an approximation $k_t^{\rm app}$ of $k_t = Q_{\alpha_2 \alpha_1}(t)k_{0}$, $k_0\in \mathcal{K}_{\alpha_1}^\star$, such that
 $\langle \! \langle G, k_t^{\rm app}\rangle \! \rangle \geq 0$ for
 all $G\in B^\star_{\rm bs}(\Gamma_0)$;
 \item[(ii)] proving that $\langle \! \langle G, k_t^{\rm app}\rangle \!
 \rangle \to \langle \! \langle G, k_t\rangle \! \rangle$ as the
 approximation is eliminated.
\end{itemize}
\begin{figure}[h]
    \centering
    \includegraphics[scale=0.23]{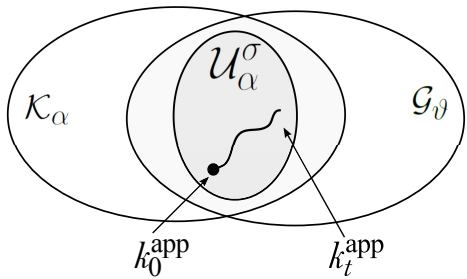}
    \caption{The evolution in spaces}
\end{figure}
Fig. 1 provides an illustration to the idea of how to realize step
(i). The origin of the inequality in question is in (\ref{15}) and
(\ref{16}). To relate $k_t$ with a positive measure one uses local
approximations of $\mu_0$, the densities of which (not necessarily
normalized) evolve $R^{\rm app}_0 \to R^{\rm app}_t$  in $L^1$-like
spaces according to Theorem \ref{1tm}. These approximations are
tailored in such a way that the corresponding correlation functions
(\ref{13}) (that have the desired property by construction) also
evolve $q_0^{\rm app}\to q_t^{\rm app}$ in $L^1$-like spaces
$\mathcal{G}_\vartheta$. The technique developed in Sect. \ref{Sec4}
allows for proving that $\langle \! \langle G, k_t^{\rm app}\rangle
\! \rangle$ converges to $\langle \! \langle G, k_t\rangle \!
\rangle$ only if $k_t^{\rm app}= Q_{\alpha \alpha_0}(t) q_0^{\rm
app}$. That is, at this stage there is no connection between the
evolutions $q_0^{\rm app}\to q_t^{\rm app}$ and $q_0^{\rm app}\to
k_t^{\rm app}$ as they take place in (different) spaces,
$\mathcal{G}_\vartheta$ and $\mathcal{K}_\alpha$, respectively. It
turns out, that these spaces have an intersection
$\mathcal{U}_\alpha^\sigma$ constructed with the help of some
objects dependent on a para,eter, $\sigma>0$. To  employ this fact
we use auxiliary models (indexed by $\sigma$), for which we prove
that both evolutions $q_0^{\rm app}\to q_t^{\rm app}$ and $q_0^{\rm
app}=k_0^{\rm app}\to k_t^{\rm app}$ take place in
$\mathcal{U}_\alpha^\sigma$ and thus coincide. That is $q_t^{\rm
app}=k_t^{\rm app}$ for $t\leq \tau$ with some positive $\tau$, that
yields the desired positivity of $k_t^{\rm app}$. Then step (ii)
includes also taking the limit $\sigma \to 0^+$.

\subsection{Auxiliary evolutions}
For $\sigma >0$ and $x\in \mathds{R}^d$, we set
\begin{eqnarray}
\label{U1} & & \phi_\sigma (x)  =  \exp\left(- \sigma |x|^2 \right),
\quad  \langle \phi_\sigma \rangle  =  \int
_{\mathds{R}^d}\phi_\sigma (x)
d x.\\[.2cm] \nonumber
& & b_\sigma (x|y_1 , y_2) =   b (x|y_1 , y_2) \phi_\sigma (y_1)
\phi_\sigma (y_2).
\end{eqnarray}
Consider
\begin{equation}
  \label{U2}
L^{\Delta, \sigma}=A^{\Delta, \sigma}+B^{\Delta,
\sigma}=A_\upsilon^{\Delta, \sigma}+B_\upsilon^{\Delta, \sigma},
\end{equation}
that is obtained from the corresponding operators in (\ref{21}) and
(\ref{45}), (\ref{46}) by replacing $b$ with $b_\sigma$ given in
(\ref{U1}). Since this substitution does not affect
$\mathcal{D}^\Delta_\alpha$, see (\ref{24}), we will use the latter
as the domain of the corresponding unbounded operators. Then we
repeat the construction as in the proof of Lemma \ref{3lm} and
obtain the family $\{Q^\sigma_{\alpha_2\alpha_1}(t): (\alpha_1 ,
\alpha_2 , t)\in \mathcal{A}(B^\Delta_\upsilon)\}$ corresponding to
the choice $\mathbf{B}= B_\upsilon^{\Delta, \sigma}$. Along with the
evolution $t \mapsto Q^\sigma_{\alpha_2\alpha_1}(t) k_0$ we will
consider two more evolutions in $L^\infty$- and $L^1$-like spaces.
The latter one will be positive in the sense of Proposition
\ref{1pn} by the very construction. The auxiliary $L^\infty$-like
space where we are going to construct $t\mapsto k_t^{\rm app}$ lies
in the intersection of the just mentioned $L^1$-like space with the
spaces $\mathcal{K}_\alpha$, see Fig. 1, and hence is also positive
in the sense of Proposition \ref{1pn}. These arguments will allow us
to realize item (i) of the program.

\subsubsection{$L^\infty$-like evolution}
For $u: \Gamma_0 \to \mathbb{R}$, we define the norm
\begin{equation}
  \label{U3}
\| u \|_{\sigma, \alpha} = \esssup_{\eta \in
\Gamma_0}\frac{|u(\eta)|\exp(-\alpha|\eta|)}{e(\phi_\sigma; \eta)},
\end{equation}
where $$e(\phi_\sigma; \eta)= \prod_{x\in \eta}\phi_\sigma
(x)=\exp\left( -\sigma\sum_{x\in \eta} |x|^2\right),$$ cf.
(\ref{11}). Then we consider the Banach space $\mathcal{U}_{\sigma,
\alpha} = \{u: \Gamma_0 \to \mathbb{R}: \| u \|_{\sigma, \alpha}
<\infty\}$. Clearly,
\begin{equation}
  \label{Jan4}
\mathcal{U}_{\sigma, \alpha}\hookrightarrow \mathcal{K}_\alpha,
\qquad \alpha \in \mathds{R}.
\end{equation}
The space predual to $\mathcal{U}_{\sigma, \alpha}$ is the
$L^1$-space equipped with the norm, cf. (\ref{35}), (\ref{36}),
\begin{equation}
  \label{U4}
  |G|_{\sigma, \alpha}=\int_{\Gamma_0} |G(\eta)|\exp (\alpha |\eta|)e(\phi_\sigma;\eta)\lambda(d
  \eta).
\end{equation}
In this space, we define $A^\sigma_{1,\upsilon}$ which acts exactly
as in (\ref{37}), and $A^\sigma_{2}$ which acts as in (\ref{37})
with $b$ replaced by $b_\sigma$. Their domain is the same
$\mathcal{D}_\alpha$. Then, like in (\ref{39}), by means of
(\ref{18}) and (\ref{U4}) we obtain
\begin{eqnarray*}
|A_2^\sigma G|_{\sigma,\alpha} & = & \int_{\Gamma_0} \left(\sum_{x
\in \eta} \int_{(\mathds{R}^d)^2}
|G(\eta \setminus x \cup \{y_1, y_2\})|b_\sigma(x|y_1,y_2)dy_1 dy_2 \right)\\[.2cm] & \times &\exp (\alpha |\eta|)e(\phi_\sigma;\eta)\lambda(d \eta)
\nonumber\\[.2cm] & = & e^\alpha \int_{\Gamma_0} \left( \int_{(\mathds{R}^d)^3}
|G(\eta \cup \{y_1, y_2\})| b_\sigma(x|y_1,y_2) \phi_\sigma (x) d x
d y_1 d y_2 \right)\nonumber\\[.2cm] & \times & \exp (\alpha |\eta|)e(\phi_\sigma;\eta) \lambda
(d\eta)
\nonumber\\[.2cm]
& \leq & e^\alpha \int_{\Gamma_0} \left( \int_{(\mathds{R}^d)^2}
|G(\eta \cup \{y_1, y_2\})| \beta(y_2-y_1)
e(\phi_\sigma;\eta\cup\{y_1,y_2\}) d y_1 d y_2 \right)
\nonumber\\[.2cm]
& \times & \exp (\alpha |\eta|) \lambda (d\eta) \nonumber\\[.2cm]
  & = & e^{-\alpha}\int_{\Gamma_0} E^b(\eta) |G(\eta)|  e^{\alpha
|\eta|} e(\phi_\sigma;\eta) \lambda(d
\eta)\nonumber\\[.2cm] \nonumber & \leq & (e^{-\alpha}/\omega) \int_{\Gamma_0} e^{\alpha |\eta|} \Psi_\upsilon (\eta)|G(\eta)|
e(\phi;\eta) \lambda(d
\eta)\nonumber\\[.2cm] \nonumber & = & (e^{-\alpha}/\omega)
|A_{1,\upsilon}^\sigma G|_{\sigma, \alpha}.
\end{eqnarray*}
This allows us to prove the following analog of Lemma \ref{2lm}.
\begin{proposition}
  \label{U1pn}
Let $\upsilon$ and $\omega$ be as in Proposition \ref{2pn} and
$A_{1,\upsilon}^\sigma$, $A_2^\sigma$ and $\mathcal{D}_\alpha$ be as
just described. Then for each $\alpha> - \log \omega$, the operator
$(A_{\upsilon}^\sigma , \mathcal{D}_\alpha):=(A^\sigma_{1,\upsilon}
+ A^\sigma_2, \mathcal{D}_\alpha)$ is the generator of a
sub-stochastic semigroup
$S_{\sigma,\alpha}=\{S_{\sigma,\alpha}(t)\}_{t\geq 0}$ on
$\mathcal{G}_{\sigma,\alpha}$.
\end{proposition}
Let $S_{\sigma,\alpha}^{\odot}$ be the sun-dual semigroup, the
definition of which is pretty analogous to that of
$S_{\alpha}^{\odot}$,  see Proposition \ref{Papn}. Then, for
$\alpha'< \alpha$, we define $\Sigma_{\alpha \alpha'}^\sigma (t)=
S^{\odot}_{\sigma, \alpha}(t)|_{\mathcal{U}_{\sigma.\alpha'}}$. As
in Proposition \ref{Papn} we then get that the map
\begin{equation*}
[0, +\infty) \ni t \mapsto \Sigma_{\alpha \alpha'}^\sigma (t) \in
\mathcal{L}(\mathcal{U}_{\sigma , \alpha'}, \mathcal{U}_{\sigma ,
\alpha})
\end{equation*}
is continuous and
\begin{equation*}
\|\Sigma_{\alpha, \alpha'}^\sigma (t)\| \leq 1, \qquad {\rm for}  \
{\rm all} \ t\geq 0.
\end{equation*}
The operators $B^{\Delta,\sigma}_{\upsilon}= B^{\Delta,\sigma}_{1}+
B^{\Delta,\sigma}_{2,\upsilon}$ act as in (\ref{46}) with $b$
replaced by $b_\sigma$. Then we define the corresponding bounded
operators and obtain, cf. (\ref{47}),
\begin{equation*}
 \|(B^{\Delta,\sigma}_{\upsilon})_{\alpha \alpha'} \| \leq \frac{2 \langle b \rangle + \upsilon + \langle a \rangle e^{\alpha'}}{e(\alpha
 -\alpha')}.
\end{equation*}
Thereafter, we take $\delta>0$ as in Lemma \ref{3lm} and the
division as in (\ref{55}), and then define
\begin{eqnarray*}
\Pi_{\alpha \alpha'}^{l,\sigma}(t, t_1, t_2, ... , t_l) & = &
\Sigma_{\alpha \alpha^{2l}}^{\sigma}(t-t_1)(B_\upsilon^{\Delta,
\sigma})_{\alpha^{2l}\alpha^{2l-1}}\cdots \Sigma_{\alpha^{2s+1}
\alpha^{2s}}^{\sigma}(t_{l-s}-t_{l-s+1})\\[.2cm] & \times &(B_\upsilon^{\Delta,
\sigma})_{\alpha^{2s}\alpha^{2s-1}}\cdots \Sigma_{\alpha^3
\alpha^{2}}^{\sigma}(t_{l-1}-t_l)(B_\upsilon^{\Delta,
\sigma})_{\alpha^{2}\alpha^{1}}\Sigma_{\alpha^1
\alpha'}^{\sigma}(t_l),
\end{eqnarray*}
As in the proof of Lemma \ref{3lm} we obtain the family $\{
U^\sigma_{\alpha_2 \alpha_1}(t):(\alpha_1 , \alpha_2, t)\in
\mathcal{A}(B^{\Delta}_\upsilon)\}$, see (\ref{54}), with members
defined by
\[
U^\sigma_{\alpha_2 \alpha_1}(t) = \Sigma_{\alpha_2 \alpha_1}^\sigma
(t) + \sum_{l=1}^\infty \int_{0}^{t}
\int_0^{t_1}...\int_0^{t_{l-1}}\Pi^{l, \sigma}_{\alpha_2
\alpha_1}(t, t_1, t_2, ... , t_l)dt_l...dt_1,
\]
where the series converges for $t< T(\alpha_2, \alpha_1)$ defined in
(\ref{34}), cf. (\ref{53}) and (\ref{Jan3}). For this family, the
following holds, cf. (\ref{54a}),
\begin{equation}
  \label{U9}
  \frac{d}{dt} U^\sigma_{\alpha_2 \alpha_1}(t) = L^{\Delta,
  \sigma}_{\alpha_2,u}  U^\sigma_{\alpha_2 \alpha_1}(t),
\end{equation}
where the action of of $L^{\Delta,\sigma}_{\alpha_2,u}$ is as in
(\ref{U2}) and the domain is
\begin{equation}
  \label{U9a}
\mathcal{D}^{\Delta,\sigma}_{\alpha_2 , u} = \{ u\in
\mathcal{U}_{\sigma, \alpha_2}: \Psi_\upsilon u \in
\mathcal{U}_{\sigma, \alpha_2}\}\subset
\mathcal{D}^\Delta_{\alpha_2},
\end{equation}
where the latter inclusion follows by (\ref{Jan4}) and (\ref{24}).
Then by (\ref{U9a}) we have that
\begin{equation}
  \label{Jan10}
L^{\Delta,\sigma}_{\alpha, u} u = L^{\Delta,\sigma}_{\alpha} u ,
\qquad u\in \mathcal{D}^{\Delta,\sigma}_{\alpha ,
  u} .
\end{equation}
Now by (\ref{U9}) we prove the following statement.
\begin{proposition}
  \label{U2pn}
For each $\alpha_2>\alpha_1 > - \log \omega$, the problem
 \begin{equation}
   \label{U10}
 \dot{u}_t = L^{\Delta,\sigma}_{\alpha_2,u} u_t, \qquad u_t|_{t=0}= u_0 \in \mathcal{U}_{\sigma,\alpha_1}
 \end{equation}
has a unique solution $u_t \in \mathcal{U}_{\sigma,\alpha_2}$ on the
time interval $[0, T(\alpha_2, \alpha_1))$. This solution is given
by $u_t = U^\sigma_{\alpha_2 \alpha_1}(t) u_0$.
\end{proposition}
\begin{corollary}
  \label{U1co}
Let $\alpha_2>\alpha_1 > - \log \omega$ be as in Proposition
\ref{U2pn} and $Q^\sigma_{\alpha_2\alpha_1}(t)$ be as described at
the beginning of this subsection. Then for each $t< T(\alpha_2 ,
\alpha_1)$ and $u_0 \in \mathcal{U}_{\sigma,\alpha_1}\subset
\mathcal{K}_{\alpha_1}$, it follows that
\begin{equation}
  \label{U11}
U^\sigma_{\alpha_2 \alpha_1}(t) u_0 = Q^\sigma_{\alpha_2
\alpha_1}(t) u_0.
\end{equation}
\end{corollary}
\begin{proof}
By (\ref{Jan10}) we get that the solution of (\ref{U10}) is also the
unique solution of the following ``$\sigma$-analog" of (\ref{33})
$$\dot{u}_t = L^{\Delta,\sigma}_{\alpha_2} u_t, \quad u_t|_{t=0}=
u_0,
$$ and hence is given by the right-hand side of (\ref{U11}). Then the
equality in (\ref{U11}) follows by the uniqueness just mentioned.
\end{proof}

\subsubsection{$L^1$-like evolution}
Now we take $L^{\Delta,\sigma}$ as given in (\ref{U2}) and define
the corresponding operator $L^{\Delta,\sigma}_\vartheta$ in
$\mathcal{G}_\vartheta$, $\vartheta \in \mathds{R}$, introduced in
(\ref{35}), (\ref{36}), with domain $\mathcal{D}_\vartheta$ given in
(\ref{37}). By (\ref{U2}) and (\ref{21}) we have that $A^\Delta_1 :
\mathcal{D}_\vartheta \to \mathcal{G}_\vartheta$. Next, for $q\in
\mathcal{D}_\vartheta$, we have
\begin{eqnarray}
  \label{Jan11}
|A^{\Delta, \sigma}_2 q |_\vartheta & \leq & \int_{\Gamma_0}
e^{\vartheta|\eta|}\left( \int_{\mathds{R}^d} \sum_{y_1\in
\eta}\sum_{y_2\in \eta\setminus y_1 } |q(\eta\cup x\setminus\{y_1,
y_2\}) | b_\sigma (x|y_1. y_2) d x \right) \lambda ( d \eta)
\qquad \quad \\[.2cm] \nonumber &\leq & \int_{\Gamma_0}
e^{\vartheta|\eta|+ 2 \vartheta} \int_{\mathds{R}^d} |q(\eta\cup x)|
\left( \int_{(\mathds{R}^d)^2} b(x|y_1 , y_2) d y_1 d y_2 \right) d
x \lambda ( d \eta) \\[.2cm] \nonumber &=& \langle b \rangle
e^\vartheta \int_{\Gamma_0} |\eta| e^{\vartheta |\eta|} |q(\eta)|
\lambda ( d \eta) \leq e^\vartheta \int_{\Gamma_0} \Psi(\eta)
e^{\vartheta |\eta|} |q(\eta)| \lambda ( d \eta),
\end{eqnarray}
see item (iii) of Assumption \ref{ass1} and (\ref{22}). Hence,
$A^{\Delta,\sigma}_2 : \mathcal{D}_\vartheta \to
\mathcal{G}_\vartheta$. Next, for the same $q$, we have
\begin{eqnarray}
  \label{Jan12}
|B^{\Delta}_1 q |_\vartheta & \leq & \int_{\Gamma_0}
e^{\vartheta|\eta|}\left( \int_{\mathds{R}^d} |q(\eta\cup x)| E^a(x,
\eta) d x \right) \lambda ( d \eta) \\[.2cm] \nonumber & = &
e^{-\vartheta} \int_{\Gamma_0} e^{\vartheta |\eta|}  E^a (\eta)
|q(\eta)| \lambda ( d \eta) \leq e^{-\vartheta} \int_{\Gamma_0}
\Psi(\eta) e^{\vartheta |\eta|} |q(\eta)| \lambda ( d \eta).
\end{eqnarray}
Hence, $B^\Delta_1 : \mathcal{D}_\vartheta \to
\mathcal{G}_\vartheta$. Finally,
\begin{eqnarray}
  \label{Jan9}
|B^{\Delta,\sigma}_2 q |_\vartheta & \leq & 2 \int_{\Gamma_0}
e^{\vartheta|\eta|} \left( \int_{(\mathds{R}^d)^2} \sum_{y_1\in
\eta} |q(\eta \cup x\setminus y_1)| b_\sigma (x|y_1, y_2) d y_2 d x
\right) \lambda (d\eta) \qquad \\[.2cm] \nonumber & \leq & 2
\int_{\Gamma_0} e^{\vartheta|\eta|+\vartheta} \left(
\int_{(\mathds{R}^d)^3}|q(\eta \cup x)|  b (x|y_1, y_2) d x d y_1 d
y_2\right)  \lambda (d\eta) \qquad \\[.2cm] \nonumber & = & 2
\langle b \rangle \int_{\Gamma_0} e^{\vartheta|\eta|} |\eta|
|q(\eta)| \lambda ( d \eta) \leq  \int_{\Gamma_0} \Psi(\eta)
e^{\vartheta |\eta|} |q(\eta)| \lambda ( d \eta).
\end{eqnarray}
Then by (\ref{Jan11}), (\ref{Jan12}) and (\ref{Jan9}) we conclude
that, for an arbitrary $\vartheta \in \mathds{R}$,
$L^{\Delta,\sigma}= A^{\Delta}_1 + A^{\Delta,\sigma}_2 +
B^{\Delta}_1 + B^{\Delta,\sigma}_2$ maps $\mathcal{D}_\vartheta$ to
$\mathcal{G}_\vartheta$ and hence can be used to define the
corresponding unbounded operator $(L^{\Delta,\sigma}_\vartheta,
\mathcal{D}_\vartheta)$. Let us then consider the corresponding
Cauchy problem
\begin{equation}
  \label{Jan8}
 \dot{q}_t = L^{\Delta,\sigma}_\vartheta q_t , \qquad q_t|_{t=0} =
 q_0 \in \mathcal{D}_\vartheta.
\end{equation}
Recall that $\mathcal{G}_{\vartheta'} \subset \mathcal{D}_\vartheta$
for each $\vartheta'> \vartheta$.
\begin{lemma}
  \label{Janln}
For a given $\vartheta >0$ and $\vartheta'> \vartheta$, assume that
the problem in (\ref{Jan8}) with $q_0 \in \mathcal{G}_{\vartheta'}$
 has a solution $q_t\in\mathcal{G}_\vartheta$ on a time interval
 $[0,\tau)$. Then this solution is unique.
\end{lemma}
\begin{proof}
Set
\[
w_t (\eta) = (-1)^{|\eta|}q_t (\eta).
\]
Then $|w_t|_\vartheta= |q_t|_\vartheta$ and $q_t$ solves
(\ref{Jan8}) if and only if $w_t$ solves the following equation
\begin{equation}
  \label{Jan13}
  \dot{w}_t = \left( A_1^{\Delta} - A_2^{\Delta,\sigma} -
  B_1^{\Delta} +  B_2^{\Delta,\sigma} \right) w_t.
\end{equation}
By Proposition \ref{TV0pn} we prove that $(A_1^{\Delta} -
B_1^{\Delta}, \mathcal{D}_\vartheta)$ generates a sub-stochastic
semigroup on $\mathcal{G}_\vartheta$. Indeed, $(A_1^{\Delta} ,
\mathcal{D}_\vartheta)$ generates a sub-stochastic semigroup defined
in (\ref{38}) with $\upsilon =0$, and $- B_1^{\Delta}$ is positive
and defined on $\mathcal{D}_\vartheta$, see (\ref{Jan12}). Also by
(\ref{Jan12}), for $w\in \mathcal{G}^{+}_\vartheta$ and $r\in
(0,1)$, we get
\begin{eqnarray*}
& & \int_{\Gamma_0} e^{\vartheta|\eta|}\left( \left(A_1^{\Delta} -
r^{-1}B_1^{\Delta}\right) w\right)(\eta) \lambda(d\eta) = -
\int_{\Gamma_0} e^{\vartheta|\eta|} \Psi(\eta) w(\eta) \lambda ( d
\eta) \\[.2cm] \nonumber & &  + r^{-1} \int_{\Gamma_0}
e^{\vartheta|\eta|}\left( \int_{\mathds{R}^d} w(\eta\cup x) E^a(x,
\eta) d x \right) \lambda ( d \eta) \\[.2cm] \nonumber & & =  -
\int_{\Gamma_0} e^{\vartheta|\eta|} \Psi(\eta) w(\eta) \lambda ( d
\eta) + r^{-1} e^{-\vartheta} \int_{\Gamma_0} e^{\vartheta |\eta|}
E^a (\eta) w(\eta) \lambda ( d \eta) \\[.2cm] \nonumber & & \leq - \left(1 - r^{-1}
e^{-\vartheta} \right) \int_{\Gamma_0} \Psi(\eta) e^{\vartheta
|\eta|} w(\eta) \lambda ( d \eta) \leq 0,
\end{eqnarray*}
where the latter inequality holds for $r\in (e^{-\vartheta}, 1)$.
Therefore, $(A_1^{\Delta} - B_1^{\Delta}, \mathcal{D}_\vartheta)=
(A_1^{\Delta} - r r^{-1} B_1^{\Delta}, \mathcal{D}_\vartheta)$
generates a sub-stochastic semigroup $V_\vartheta =\{V_\vartheta
(t)\}_{t\geq 0}$ on $\mathcal{G}_\vartheta$. For each $\vartheta''
\in (0, \vartheta)$, we have that $\mathcal{G}_\vartheta
\hookrightarrow \mathcal{G}_{\vartheta''}$. By the estimates in
(\ref{Jan11}) and (\ref{Jan9}), similarly as in (\ref{47}) we obtain
that
\begin{eqnarray*}
|A^{\Delta, \sigma}_2 w|_{\vartheta''} &\leq & \frac{\langle b
\rangle}{e(\vartheta - \vartheta'')} |w|_\vartheta, \\[.2cm]
|B^{\Delta, \sigma}_2 w|_{\vartheta''} & \leq & \frac{2\langle b
\rangle}{e(\vartheta - \vartheta'')} |w|_\vartheta,
\end{eqnarray*}
which we then use to define a bounded operator $C^{\Delta,
\sigma}_{\vartheta''\vartheta} :\mathcal{G}_\vartheta \to
\mathcal{G}_{\vartheta''}$. It acts as $- A^{\Delta,\sigma}_2 +
B^{\Delta,\sigma}_2$ and its norm satisfies
\begin{equation}
  \label{Jan15}
\|C^{\Delta, \sigma}_{\vartheta''\vartheta}\| \leq \frac{3\langle b
\rangle}{e(\vartheta - \vartheta'')}.
\end{equation}
Assume now that (\ref{Jan13}) has two solutions corresponding to the
same initial condition $w_0$. Let $v_t$ be their difference. Then it
solves (\ref{Jan13}) with the zero initial condition and hence
satisfies
\begin{equation}
  \label{Jan16}
 v_t = \int_0^t V_{\vartheta''} (t - s) C^{\Delta,
 \sigma}_{\vartheta'' \vartheta}
 v _s d s
\end{equation}
where $v_t$ in the left-hand side is considered as an element of
$\mathcal{G}_{\vartheta''}$ and $t>0$ will be chosen later. Now for
a given $n\in \mathds{N}$, we set $\epsilon = (\vartheta -
\vartheta'')/n$ and $\vartheta^l = \vartheta - l \epsilon$, $l=0,
\dots , n$. Next, we iterate (\ref{Jan16}) due times and get
\begin{eqnarray*}
v_t & = & \int_0^t \int_0^{t_1} \cdots \int_{0}^{t_{n-1}}
V_{\vartheta''} (t-t_1)  C^{\Delta,
 \sigma}_{\vartheta'' \vartheta^{n-1}} V_{\vartheta^{n-1}} (t_1-t_2)  C^{\Delta,
 \sigma}_{\vartheta^{n-1} \vartheta^{n-2}} \times \cdots \times
 \\[.2cm] & \times & V_{\vartheta^{1}} (t_{n-1}-t_n)  C^{\Delta,
 \sigma}_{\vartheta^{n-1} \vartheta} v_{t_n} d t_n \cdots d t_1.
\end{eqnarray*}
Then we take into account that $V_\vartheta$ is sub-stochastic,
$C^{\Delta,
 \sigma}_{\vartheta^l \vartheta^{l-1}}$ are positive and satisfy
(\ref{Jan15}), and thus obtain from the latter that $v_t$ satisfies
\[
|v_t|_{\vartheta''} \leq \frac{1}{n!} \left(\frac{n}{e} \right)^n
\left(\frac{3 t \langle b \rangle}{\vartheta - \vartheta''}\right)^n
\sup_{s\in [0,t]}|v_s|_{\vartheta}.
\]
Then, since $n$ is an arbitrary positive integer, for all $t<
(\vartheta - \vartheta'')/ 3\langle b \rangle$ it follows that $v_t
=0$. To prove that $v_t=0$ for all $t$ of interest one has to repeat
the above procedure appropriate number of times.
\end{proof}
Let us now take $u\in \mathcal{U}_{\sigma, \alpha}$ with some
$\alpha \in \mathds{R}$, for which by (\ref{U3}) we have
\[
|u(\eta)|\leq \|u\|_{\sigma,\alpha} e^{\alpha |\eta|} e(\phi_\sigma,
\eta).
\]
Then the norm of this $u$ in $\mathcal{G}_\vartheta$ can be
estimated as follows, see (\ref{U1}),
\begin{equation}
\label{x}
|u|_\vartheta\leq \|u\|_{\sigma,\alpha} \int_{\Gamma_0} \exp\left(
(\alpha + \vartheta) |\eta|\right) e(\phi_\sigma, \eta)\lambda
(d\eta) = \|u\|_{\sigma,\alpha} \exp\left( (\alpha + \vartheta)
\langle \phi \rangle\right).
\end{equation}
This means that $\mathcal{U}_{\sigma, \alpha}\hookrightarrow
\mathcal{G}_{\vartheta}$ for each pair of real $\alpha$ and
$\vartheta$. Moreover, for the operators discussed above this
implies, cf. (\ref{Jan10}),
\begin{equation}
  \label{U12}
  L^{\Delta,\sigma}_{\alpha,u}u = L_{\vartheta}^{\Delta, \sigma}u, \qquad u\in\mathcal{D}^{\Delta,\sigma}_{\alpha,u}.
\end{equation}
\begin{corollary}
  \label{Jan2co}
Let $\alpha_1$ and $\alpha_2$ be as in Proposition \ref{U2pn}. Then,
for each $q_0 \in \mathcal{U}_{\sigma,\alpha_1}$, the problem in
(\ref{Jan8}) has a unique solution $q_t \in
\mathcal{U}_{\sigma,\alpha_2}$, $t<T(\alpha_2, \alpha_1)$, which
coincides with the unique solution of (\ref{U10}).
\end{corollary}
\begin{proof}
By (\ref{U12}) we have that the unique solution of (\ref{U10}) $u_t$
solves also (\ref{Jan8}), and this is a unique solution in view of
Lemma \ref{Janln}.
\end{proof}

\subsection{Local approximations}

Our aim now is to prove that, cf. Proposition \ref{1pn}, the
following holds
\begin{equation}
  \label{Jan18}
\langle \! \langle G, Q^\sigma_{\alpha_2 \alpha_1}(t) k_0 \rangle \!
\rangle \geq 0, \qquad G\in B_{\rm bs}^\star (\Gamma_0),
\end{equation}
for suitable $t>0$. By Corollaries \ref{U1co} and \ref{Jan2co} to
this end it is enough to prove (\ref{Jan18}) with
$Q^\sigma_{\alpha_2 \alpha_1}(t) k_0$ replaced by $q_t$.

For $\mu_0\in \mathcal{P}_{\rm exp}(\Gamma)$ and a compact
$\Lambda$, let $\mu^\Lambda_0\in \mathcal{P}(\Gamma_\Lambda)$ be the
corresponding projection to $\Gamma_\Lambda$ defined in (\ref{Rel}).
Let $R^\Lambda_0$ be its Radon-Nikodym derivative, see (\ref{12}).
For $N\in \mathds{N}$ and $\eta\in \Gamma_0$, we then set
\begin{equation}
  \label{Jan19}
  R^{\Lambda, N}_0 (\eta) = \left\{ \begin{array}{ll}
 R^{\Lambda}_0(\eta) , \qquad &{\rm if} \ \ \eta\in \Gamma_\Lambda \
 \ {\rm and} \ \ |\eta|\leq N;\\[.2cm]
 0, \qquad &{\rm otherwise}. \end{array} \right.
\end{equation}
Until the end of this subsection, $\Lambda$ and $N$ are fixed.
Having in mind (\ref{13}) we introduce
\begin{equation}
  \label{Jan20}
  q_0^{\Lambda,N}( \eta) = \int_{\Gamma_0} R^{\Lambda,N}_0 (\eta\cup
  \xi) \lambda ( d \xi), \qquad \eta \in  \Gamma_{0}.
\end{equation}
For $G\in B^{\star}_{\rm bs}(\Gamma_0)$, by (\ref{15}), (\ref{18})
and (\ref{Jan20}) we have
\begin{equation}
  \label{Jan21}
\langle\!\langle G, q_0^{\Lambda,N} \rangle\!\rangle =
\langle\!\langle K G, R_0^{\Lambda,N} \rangle\!\rangle \geq 0.
\end{equation}
By (\ref{Jan19}) it follows that $R^{\Lambda, N}_0 \in
\mathcal{R}^{+}$ and $\|R^{\Lambda, N}_0\|_\mathcal{R}\leq 1$.
Moreover, for each $\kappa>0$, we have, see (\ref{8}),
\begin{eqnarray}
  \label{Jan22}
  \|R^{\Lambda, N}_0\|_{\mathcal{R}_{\chi^\kappa}} =
  \int_{\Gamma_\Lambda} e^{\kappa |\eta|} R^{\Lambda, N}_0 (\eta)
  \lambda ( d\eta) \leq e^{\kappa N} \|R^{\Lambda,
  N}_0\|_{\mathcal{R}} \leq e^{\kappa N}.
\end{eqnarray}
Let $S^\sigma_\mathcal{R}$ be the stochastic semigroup on
$\mathcal{R}$ constructed in the proof of Theorem \ref{1ftm} with
$b$ replaced by $b_\sigma$. Recall that $R_t = S_\mathcal{R}(t) R_0$
is the solution of (\ref{L9}). Set
\begin{eqnarray}
  \label{Jan23}
R^{\Lambda, N}_t & = & S_\mathcal{R}^\sigma(t) R^{\Lambda, N}_0, \qquad t>0, \\[.2cm]
q^{\Lambda, N}_t (\eta) & = & \int_{\Gamma_0} R^{\Lambda, N}_t
(\eta\cup \xi) \lambda ( d \xi), \quad \eta \in \Gamma_0. \nonumber
\end{eqnarray}
\begin{proposition}
  \label{JJ1pn}
For each $\vartheta\in \mathds{R}$ and $t \in [0, \tau_\vartheta)$,
$\tau_\vartheta:= [ e \langle b \rangle ( 1+ e^\vartheta)]^{-1}$, it
follows that $q^{\Lambda, N}_t \in \mathcal{G}_\vartheta^{+}$.
Moreover,
\begin{equation}
  \label{Jan24}
\langle\!\langle G, q_t^{\Lambda,N} \rangle\!\rangle \geq 0
\end{equation}
holding for each $G\in B^{\star}_{\rm bs}(\Gamma_0)$ and all $t>0$.
\end{proposition}
\begin{proof}
Since $S_\mathcal{R}^\sigma$ is stochastic and $R_0^{\Lambda,N}$ is
as in (\ref{Jan19}), then $R_t^{\Lambda,N}\in \mathcal{R}^{+}$ for
all $t>0$. Hence, $q^{\Lambda, N}_t(\eta) \geq 0$ for all those
$t>0$ for which the integral in the second line in (\ref{Jan23})
makes sense.  By (\ref{22T}) we have that $T(\kappa, \kappa')$, as a
function of $\kappa$, attains its maximum value $T_{\kappa'} =
e^{-\kappa'}/ e \langle b \rangle$ at $\kappa = \kappa'+1$. By
(\ref{Jan22}) we have that $R^{\Lambda, N}_0 \in
\mathcal{R}_{\chi^\kappa}$ for any $\kappa>0$. Then, for each
$\kappa>0$, by Proposition \ref{TVpn} it follows that $R^{\Lambda,
N}_t \in \mathcal{R}_{\chi^\kappa}$ for $t< T_\kappa$. Taking all
these fact into account we then get
\begin{eqnarray}
  \label{Jan25}
 |q_t^{\Lambda,N}|_\vartheta & = & \int_{\Gamma_0} e^{\vartheta|\eta|}
 q_t^{\Lambda,N} (\eta) \lambda ( d\eta)\\ & = & \nonumber \int_{\Gamma_0}
 \int_{\Gamma_0} e^{\vartheta|\eta|} R^{\Lambda, N}_t
(\eta\cup \xi) \lambda (d\eta)\lambda ( d \xi) \\ & = &
\int_{\Gamma_0} \left( 1 + e^\vartheta\right)^{|\eta|} R^{\Lambda,
N}_t (\eta) \lambda (d\eta) = \| R^{\Lambda,
N}_t\|_{\mathcal{R}_{\chi^\kappa}} \nonumber
\end{eqnarray}
with $\kappa = \log (1 +e^\vartheta)$. For these $\kappa$ and
$\vartheta$, we have that $T_\kappa = \tau_\vartheta$. Then
$q_t^{\Lambda,N}\in \mathcal{G}_\vartheta$ for $t< \tau_\vartheta$,
holding by (\ref{Jan25}). The existence of the integral in
(\ref{Jan24}) follows by the equality
\[
\langle \! \langle G, q^{\Lambda,N}_t \rangle \! \rangle = \langle
\! \langle K G, R^{\Lambda,N}_t \rangle \! \rangle,
\]
(\ref{10}) and the fact that $R^{\Lambda,N}_t \in
\mathcal{R}^{+}_{\chi_m}$ for all $t>0$ and $m\in \mathds{N}$, see
claims (a) and (c) of Theorem \ref{1ftm}. The validity of the
inequality in (\ref{Jan24}) is straightforward, cf. (\ref{Jan21}).
\end{proof}
\begin{corollary}
  \label{JJ1co}
For each $\alpha \in \mathds{R}$, it follows that
$q_0^{\Lambda,N}\in \mathcal{U}_{\sigma,\alpha}^{+}$.
\end{corollary}
\begin{proof}
Set $I_N(\eta)=1$ whenever $|\eta|\leq N$ and $I_N(\eta)=0$
otherwise. By (\ref{Jan19}), (\ref{Jan20}) and (\ref{13}) we have
that
\begin{eqnarray*}
q_0^{\Lambda,N}(\eta) & = &  I_N (\eta) \mathds{1}_{\Gamma_\Lambda}
(\eta) \int_{\Gamma_\Lambda}R_0^\Lambda(\eta
\cup\xi)\lambda (d \xi) \\[.2cm]  & = & k_0(\eta) I_N (\eta)
\mathds{1}_{\Gamma_\Lambda} (\eta)\leq  \varkappa^N I_N (\eta)
\mathds{1}_{\Gamma_\Lambda} (\eta).
\end{eqnarray*}
The latter estimate follows by the fact that $k_0=k_{\mu_0}$ for
some $\mu_0\in \mathcal{P}_{\rm exp}(\Gamma)$, and thus $k_0(\eta)
\leq
 \varkappa^{ |\eta|}$ for some $\varkappa>0$, see Definition \ref{0df} and (\ref{6c}). Then
$q_0^{\Lambda,N}\in \mathcal{U}_{\sigma,\alpha}$ by (\ref{U3}). The
stated positivity i immediate.
\end{proof}
By (\ref{x}) and Corollary \ref{JJ1co} we obtain that
$q_0^{\Lambda,N}\in \mathcal{G}_{\vartheta}^{+}$ for each
$\vartheta\in \mathds{R}$. Now we relate $q_t^{\Lambda,N}$ with
solutions of (\ref{Jan8}).
\begin{lemma}
  \label{JJ1lm}
For each $\vartheta \in \mathds{R}$, the map $[0, \tau_\vartheta)\ni
t\mapsto q_t^{\Lambda,N} \in \mathcal{G}_\vartheta$ is continuous
and continuously differentiable on $(0, \tau_\vartheta)$. Moreover,
$q_t^{\Lambda,N} \in \mathcal{D}_\vartheta$, see (\ref{37}), and
solves the problem in (\ref{Jan8}) on the time interval
$[0,\tau_\vartheta)$ with $q_0^{\Lambda,N}$ as the initial
condition.
\end{lemma}
\begin{proof}
Fix an arbitrary $\vartheta \in \mathds{R}$. The stated continuity
of $t\mapsto q_t^{\Lambda,N}$ follows by (\ref{Jan23}). Let us prove
that $q_t^{\Lambda,N}$ be differentiable in $\mathcal{G}_\vartheta$
on $(0,\tau_\vartheta)$ and the following holds
\begin{equation}
  \label{Mar1}
  \dot{q}^{\Lambda,N}_t (\eta) = \int_{\Gamma_0} \dot{R}^{\Lambda,N}_t (\eta\cup
  \xi)\lambda(d\xi).
\end{equation}
For small enough $|\tau|$, we have
\begin{eqnarray}
  \label{Mar2}
& & \frac{1}{\tau} \left(q^{\Lambda,N}_{t+\tau} (\eta)-
q^{\Lambda,N}_t (\eta)\right) - \int_{\Gamma_0}
\dot{R}^{\Lambda,N}_t (\eta\cup
  \xi)\lambda(d\xi) \\[.2cm] & & \qquad = \int_{\Gamma_0} \left[\frac{1}{\tau} \left(R^{\Lambda,N}_{t+\tau} (\eta\cup\xi)- R^{\Lambda,N}_t
(\eta\cup\xi)\right) - \dot{R}^{\Lambda,N}_t (\eta\cup
  \xi) \right]\lambda(d\xi). \nonumber
\end{eqnarray}
Then by (\ref{18}) we get
\begin{eqnarray*}
\left\vert {\rm LHS}(\ref{Mar2})\right\vert_\vartheta \leq
\int_{\Gamma_0} \left(1+e^\vartheta \right)^{|\eta|}
\left\vert\frac{1}{\tau} \left(R^{\Lambda,N}_{t+\tau} (\eta)-
R^{\Lambda,N}_t (\eta)\right) - \dot{R}^{\Lambda,N}_t (\eta)
\right\vert\lambda(d\eta),
\end{eqnarray*}
that proves (\ref{Mar1}), cf. (\ref{Jan25}). The continuity of
$t\mapsto \dot{q}^{\Lambda,N}_t$ follows by (\ref{Mar1}) and the
fact that $R^{\Lambda,N}_t = S^\sigma_{\mathcal{R}}
(t)R^{\Lambda,N}_0$, which also yields that
\begin{equation}
  \label{Jan27}
\dot{q}_t^{\Lambda,N} (\eta) = \int_{\Gamma_0}
\left(L^{\dagger,\sigma}_\vartheta R_t^{\Lambda,N} \right) (\eta\cup
\xi) \lambda ( d\xi),
\end{equation}
where $L^{\dagger,\sigma}_\vartheta$ is the trace of
$L^{\dagger,\sigma}$  (the generator of $S^\sigma_{\mathcal{R}}$) in
$\mathcal{R}_{\chi^\kappa}$ with $\kappa = \log(1+e^\vartheta)$. By
(\ref{37}) it follows that $\Psi_\upsilon (\eta) \leq C_\varepsilon
e^{\varepsilon |\eta|}$ holding for an arbitrary $\varepsilon >0$
and the corresponding $C_\varepsilon>0$. For each $t<T_\kappa =
\tau_\vartheta$, one can pick $\kappa'>\kappa$ such that
$R_t^{\Lambda, N}\in \mathcal{R}_{\chi^{\kappa'}}$. For these $t$
and $\kappa'$, we thus pick $\varepsilon>0$ such that $1+
e^{\vartheta + \varepsilon} = e^{\kappa'}$, and then obtain, cf.
(\ref{Jan25}),
\begin{eqnarray}
  \label{Mar4}
|\Psi_\upsilon q_t^{\Lambda,N}|_\vartheta \leq C_\varepsilon
\|R^{\Lambda,N}_t\|_{\mathcal{R}^{\chi^{\kappa'}}}.
\end{eqnarray}
Hence, $q_t^{\Lambda,N}\in \mathcal{D}_\vartheta$ for this $t$. Let
us now prove that $q_t^{\Lambda,N}$ solves (\ref{Jan8}). In view of
(\ref{Jan27}), (\ref{L}) and (\ref{Mar4}), to this end it is enough
to prove that
\begin{eqnarray}
  \label{Mar5}
\left(L^\Delta q_t^{\Lambda,N}\right)(\eta) & = & - \int_{\Gamma_0}
\Psi(\eta\cup\xi) R^{\Lambda,N}_t (\eta \cup \xi) \lambda (d\xi)
\\[.2cm] \nonumber & + &\int_{\mathds{R}^d} \int_{\Gamma_0} \left( m(x)+ E^a(x, \eta \cup
\xi)\right)R^{\Lambda,N}_t (\eta \cup \xi\cup x) \lambda (d\xi) dx \\[.2cm] \nonumber & + &\int_{\mathds{R}^d}
\int_{\Gamma_0}\sum_{y_1\in \eta \cup\xi} \sum_{y_2\in \eta
\cup\xi\setminus y_1} b(x|y_1 , y_2) R^{\Lambda,N}_t (\eta \cup
\xi\cup x\setminus\{y_1,y_2\} ) \lambda (d\xi)d x,
\end{eqnarray}
holding point-wise in $\eta\in\Gamma_0$. By (\ref{22}) and
(\ref{19}) we get
\begin{equation}
  \label{Mar6}
 \Psi (\eta \cup \xi) = \Psi (\eta)+ \Psi (\xi) + 2 \sum_{x\in
 \eta}\sum_{y\in \xi} a(x-y).
\end{equation}
Let $I_1 (\eta)$ denote the first summand in the right-hand side of
(\ref{Mar5}). By (\ref{18}) and (\ref{Mar6}) we then write it as
follows
\begin{eqnarray}
  \label{Mar7}
I_1 (\eta) & = & - \Psi(\eta)q_t^{\Lambda,N} (\eta) - 2
\int_{\mathds{R}^d} E^a (x, \eta) q_t^{\Lambda,N} (\eta\cup x) dx
\\[.2cm]\nonumber & - & \int_{\Gamma_0}\Psi (\xi)R^{\Lambda,N}_t (\eta \cup \xi) \lambda
(d\xi).
\end{eqnarray}
To calculate the latter summand in (\ref{Mar7}) we again use
(\ref{22}) and (\ref{18}) to obtain the following:
\begin{eqnarray}
  \label{Mar8}
 \int_{\Gamma_0}\left( \sum_{x\in \xi} m(x) \right) R^{\Lambda,N}_t (\eta \cup
 \xi) \lambda (d \xi)& = & \int_{\Gamma_0} \int_{\mathds{R}^d} m(x) R^{\Lambda,N}_t (\eta \cup
 \xi\cup x) \lambda (d \xi) dx \qquad \quad \\[.2cm] \nonumber & = & \int_{\mathds{R}^d}
 m(x) q_t^{\Lambda,N} (\eta\cup x) dx.
\end{eqnarray}
\begin{eqnarray}
  \label{Mar9}
 & & \int_{\Gamma_0}\left( \sum_{x\in \xi}\sum_{y\in \xi\setminus x} a(x-y) \right) R^{\Lambda,N}_t (\eta \cup
 \xi) \lambda (d \xi)\\[.2cm] \nonumber & & \quad =  \int_{\Gamma_0} \int_{(\mathds{R}^d)^2} a(x-y) R^{\Lambda,N}_t (\eta \cup
 \xi\cup\{ x,y\}) \lambda (d \xi) dx d y  \\[.2cm] \nonumber & & \quad  =  \int_{(\mathds{R}^d)^2}
 a(x-y) q_t^{\Lambda,N} (\eta\cup \{x,y\}) dx d y .
\end{eqnarray}
\begin{eqnarray}
  \label{Mar10}
 \int_{\Gamma_0}\left(\langle b \rangle \sum_{x\in \xi} 1 \right) R^{\Lambda,N}_t (\eta \cup
 \xi) \lambda (d \xi) =  \langle b \rangle \int_{\mathds{R}^d} q^{\Lambda,N}_t (\eta
 \cup x) dx.
\end{eqnarray}
In a similar way, we get the second $I_2$ (resp. the third $I_3$)
summands of the right-hand side of (\ref{Mar5}) as follows
\begin{eqnarray}
  \label{Mar11}
I_2(\eta) & = & \int_{\mathds{R}^d} \left(m(x) + E^a (x, \eta)
\right)q^{\Lambda,N}_t (\eta
 \cup x) dx \\[.2cm] \nonumber & + & \int_{(\mathds{R}^d)^2} a (x-y)
q^{\Lambda,N}_t (\eta
 \cup\{ x,y\}) dx d y.
\end{eqnarray}
\begin{eqnarray}
  \label{Mar12}
I_3(\eta) & = & \int_{\mathds{R}^d} \sum_{y_1\in \eta} \sum_{y_2\in
\eta\setminus y_1} b(x|y_1, y_2)q^{\Lambda,N}_t (\eta
 \cup x\setminus \{y_1,y_2\}) dx  \\[.2cm] & + & \nonumber 2
 \int_{(\mathds{R}^d)^2} \sum_{y_1\in \eta}  b(x|y_1,y_2)
q^{\Lambda,N}_t (\eta \cup x\setminus y_1) dx  d y_2 \\[.2cm] \nonumber & + & \langle b \rangle  \int_{\mathds{R}^d}
q^{\Lambda,N}_t (\eta
 \cup x) dx .
\end{eqnarray}
Now we plug (\ref{Mar8}), (\ref{Mar9}) and (\ref{Mar10}) into
(\ref{Mar7}), and then use it together with (\ref{Mar11}) and
(\ref{Mar12}) in the right-hand side of (\ref{Mar5}) to get its
equality with the left-hand side, see (\ref{21}). This completes the
proof.
\end{proof}
\begin{corollary}
  \label{JJ2co}
Let $\alpha_1> -\log \omega$ and $\alpha_2 >\alpha_1$ be chosen.
Then $k^{\Lambda,N}_t = Q^\sigma_{\alpha_2\alpha_1} (t)
q_0^{\Lambda,N}$ has the property
\begin{equation}
  \label{Jan32}
\langle\! \langle G, k_t^{\Lambda,N} \rangle\! \rangle \geq 0,
\end{equation}
holding for all $G\in B_{\rm bs}^\star (\Gamma_0)$ and
$t<T(\alpha_2,\alpha_1)$.
\end{corollary}
\begin{proof}
The proof of (\ref{Jan32}) will be done by showing that
$k^{\Lambda,N}_t= q^{\Lambda,N}_t$, for $t< T(\alpha_2,\alpha_1)$
and then by employing (\ref{Jan24}), which holds for all $t>0$.

By Corollary \ref{JJ1co} it follows that $q_0^{\Lambda,N}\in
\mathcal{U}_{\sigma,\alpha_1}$, and hence $u_t =
U^\sigma_{\alpha_2\alpha_1}(t)q_0^{\Lambda,N}$ is a unique solution
of (\ref{U10}), see Proposition \ref{U2pn}. By Lemma \ref{JJ1lm}
$q_t^{\Lambda,N}$ solves (\ref{Jan8}) on $[0,\tau_\vartheta)$, which
by Corollary \ref{Jan2co} yields $u_t=q_t^{\Lambda,N}$ for $t<
\min\{\tau_\vartheta; T(\alpha_2,\alpha_1)\}$. If $\tau_\vartheta <
T(\alpha_2,\alpha_1)$, we can continue $q_t^{\Lambda,N}$ beyond
$\tau_\vartheta$ by means of the following arguments. Since
$u_t=q_t^{\Lambda,N}$ lies in $\mathcal{U}_{\sigma,\alpha_2}$ for
all $t< \min\{\tau_\vartheta; T(\alpha_2,\alpha_1)\}$, by (\ref{x})
we get that $q_t^{\Lambda,N}$ lies in the initial space
$\mathcal{G}_{\vartheta'}$ and hence can further by continued. Thus,
$u_t=q_t^{\Lambda,N}$ for all $t< T(\alpha_2,\alpha_1)$. Now by
(\ref{U11}) we get $q_t^{\Lambda,N}=u_t= k^{\Lambda,N}_t$, that
completes the proof.
\end{proof}

\subsection{Taking the limits}

We prove that (\ref{Jan32}) holds when the approximation is removed.
Recall that $k_t^{\Lambda,N}$ in (\ref{Jan32}) depends on
$\sigma>0$, $\Lambda$ and $N$. We first take the limits $\Lambda\to
\mathds{R}^d$ and $N\to +\infty$. Below, by an exhausting sequence
$\{\Lambda_n\}_{n\in \mathds{N}}$ we mean a sequence of compact
$\Lambda_n$ such that: (a) $\Lambda_n\subset \Lambda_{n+1}$ for all
$n$; (b) for each $x\in \mathds{R}^d$, there exits $n$ such that
$x\in \Lambda_n$.
\begin{proposition}
  \label{JJ10pn}
Let $\alpha_1>-\log \omega$, $\alpha_2>\alpha_1$ and $k_0\in
\mathcal{K}^\star_{\alpha_1}$  be fixed. For these $\alpha_1$,
$\alpha_2$ and $t< T(\alpha_2 , \alpha_1)$, let $k_t^{\Lambda,N}$
and $Q^\sigma_{\alpha_2\alpha_1}(t)$ be the same as in Corollary
\ref{JJ2co} and (\ref{Jan18}), respectively. Then, for each $G\in
B_{\rm bs}(\Gamma_0)$ and any $t<T(\alpha_2,\alpha_1)$, the
following holds
\[
\lim_{n\to +\infty} \lim_{l\to +\infty} \langle \!\langle G,
k_t^{\Lambda_n, N_l}\rangle \!\rangle = \langle \!\langle G,
Q^\sigma_{\alpha_2\alpha_1} (t) k_0\rangle \!\rangle,
\]
for arbitrary exhausting $\{\Lambda_n\}_{n\in \mathds{N}}$ and
increasing $\{N_l\}_{l\in \mathds{N}}$ sequences of sets and
positive integers, respectively.
\end{proposition}
The proof of this statement can be performed by the literal
repetition of the proof of a similar statement given in Appendix of
\cite{Berns}.

Recall that, for $\alpha_2 > \alpha_1$,  $T(\alpha_2,\alpha_1)$ was
defined in (\ref{34}). For these, $\alpha_2$, $\alpha_1$, we set
\begin{equation}
  \label{Jan41}
\alpha= \frac{1}{3}\alpha_2 + \frac{2}{3}\alpha_1, \quad \ \alpha'=
\frac{2}{3}\alpha_2 + \frac{1}{4}\alpha_1.
\end{equation}
Clearly,
\begin{equation}
  \label{Jan40}
\tau(\alpha_2, \alpha_1):= \frac{1}{3} T(\alpha_2, \alpha_1) <
\min\{ T(\alpha_2 , \alpha');  T(\alpha , \alpha_1)\}.
\end{equation}
\begin{lemma}
  \label{JJ10lm}
Let $\alpha_1$, $\alpha_2$ and $k_0$ be as in Proposition
\ref{JJ10pn}, and let $k_t$ be the solution of (\ref{33}). Then for
each $G\in B_{\rm bs}(\Gamma_0)$ and $t\in
[0,\tau(\alpha_2,\alpha_1)]$, the following holds
\begin{equation}
  \label{Jan42}
\lim_{\sigma \to 0^+} \langle \!\langle G,
Q^\sigma_{\alpha_2\alpha_1} (t) k_0\rangle \!\rangle = \langle
\!\langle G,  k_t\rangle \!\rangle.
\end{equation}
\end{lemma}
\begin{proof}
We recall that the solution of (\ref{33}) is
$k_t=Q_{\alpha_2\alpha_1}(t)k_0$ with $Q_{\alpha_2\alpha_1}(t)$
given in (\ref{Jan3}) and $t\leq T(\alpha_2, \alpha_1)$, see Lemma
\ref{1lm}.  For $\alpha$ and $\alpha'$ as in (\ref{Jan41}) and
$t\leq \tau(\alpha_2, \alpha_1)$, write
\begin{eqnarray}
  \label{Jan43}
Q_{\alpha_2\alpha_1}(t) k_0 & = & Q^\sigma_{\alpha_2\alpha_1}(t) k_0
+
\Upsilon_{1}(t,\sigma) + \Upsilon_{2}(t,\sigma), \\[.2cm]
\nonumber \Upsilon_{1}(t,\sigma)& = & \int_0^t
Q_{\alpha_2\alpha'}(t-s) \left[(A^\Delta_2)_{\alpha'\alpha} -
(A^{\Delta.\sigma}_2)_{\alpha'\alpha}
\right]Q^\sigma_{\alpha\alpha_1}(s) k_0 d s,
, \\[.2cm]
\nonumber \Upsilon_{2}(t,\sigma)& = & \int_0^t
Q_{\alpha_2\alpha'}(t-s) \left[(B^\Delta_2)_{\alpha'\alpha} -
(B^{\Delta,\sigma}_2)_{\alpha'\alpha}
\right]Q^\sigma_{\alpha\alpha_1}(s) k_0 d s.
\end{eqnarray}
The validity of (\ref{Jan43}) is verified by taking the
$t$-derivatives from both sides and then by using e.g., (\ref{54b}).
Note that the norms of the operators $(A^\Delta_2)_{\alpha'\alpha}$.
$(B^\Delta_2)_{\alpha'\alpha}$,
$(A^{\Delta,\sigma}_2)_{\alpha'\alpha}$,
$(B^{\Delta.\sigma}_2)_{\alpha'\alpha}$ can be estimated as in
(\ref{48}). For $G$ as in (\ref{Jan42}), we then have
\begin{equation}
  \label{Jan44}
\langle \!\langle G, Q_{\alpha_2\alpha_1}(t) k_0\rangle \!\rangle -
\langle \!\langle G, Q^\sigma_{\alpha_2\alpha_1}(t)  k_0\rangle
\!\rangle = \langle \!\langle G, \Upsilon_1(t,\sigma) \rangle
\!\rangle + \langle \!\langle G, \Upsilon_2(t,\sigma) \rangle
\!\rangle.
\end{equation}
By (\ref{Du5}) and (\ref{Jan43}) it follows that
\begin{eqnarray*}
\langle \!\langle G, \Upsilon_1(t,\sigma) \rangle \!\rangle & = &
\int_0^t \langle \!\langle G, Q_{\alpha_2\alpha'}(t-s)
\left[(A^\Delta_2)_{\alpha'\alpha} -
(A^{\Delta.\sigma}_2)_{\alpha'\alpha}
\right]Q^\sigma_{\alpha\alpha_1}(s) k_0   \rangle \!\rangle d s
\qquad \\[.2cm] \nonumber &=& \int_0^t \langle \!\langle H_{\alpha'\alpha_2}(t-s)
G, v_s^\sigma  \rangle \!\rangle d s = \int_0^t \langle \!\langle
G_{t-s}, v_s^\sigma  \rangle \!\rangle d s,
\end{eqnarray*}
where
\begin{equation*}
v^\sigma_s =\left[(A^\Delta_2)_{\alpha'\alpha} -
(A^{\Delta.\sigma}_2)_{\alpha'\alpha} \right]k^\sigma_s :=
\left[(A^\Delta_2)_{\alpha'\alpha} -
(A^{\Delta.\sigma}_2)_{\alpha'\alpha}
\right]Q^\sigma_{\alpha\alpha_1}(s) k_0 \in \mathcal{K}_{\alpha'},
\end{equation*}
and
\begin{equation}
  \label{Jan46a}
G_{t-s} = H_{\alpha'\alpha_2}(t-s) G \in \mathcal{G}_{\alpha'},
\end{equation}
which makes sense since obviously $G\in \mathcal{G}_{\alpha_2}$. In
view of (\ref{21}) we then get
\begin{eqnarray}
  \label{Jan47}
& & \int_0^t \langle \!\langle G_{t-s}, v_s^\sigma  \rangle
\!\rangle d s =  \int_{\Gamma_0} G_{t-s} (\eta)
\bigg{(}\int_{\mathds{R}^d} \sum_{y_1\in\eta} \sum_{y_2\in
\eta\setminus y_1} k^\sigma_s (\eta
\cup x\setminus\{y_1,y_2\})  \\[.2cm] \nonumber & & \qquad \times  \left[1-\phi_\sigma (y_1) \phi_\sigma (y_2)\right] b(x|y_1, y_2)dx
\bigg{)} \lambda ( d\eta)   \\[.2cm] \nonumber & &\qquad = \int_{\Gamma_0}
\bigg{(} \int_{(\mathds{R}^d)^3}G_{t-s} (\eta\cup\{y_1,y_2\})
k^\sigma_s (\eta \cup x)\\[.2cm] \nonumber & & \qquad \times  \left[1-\phi_\sigma (y_1) \phi_\sigma
(y_2)\right] b(x|y_1, y_2) d x dy_1 dy_2 \bigg{)} \lambda (d\eta).
\end{eqnarray}
Since $k^\sigma_s = Q^\sigma_{\alpha\alpha_1}(s) k_0$ is in
$\mathcal{K}_\alpha$, we have that
\begin{equation}
  \label{Jan48}
|k^\sigma_s (\eta \cup x)| \leq \|k^\sigma_s\|_\alpha e^{\alpha
|\eta| +\alpha} \leq  e^{\alpha |\eta| +\alpha} \frac{T(\alpha,
\alpha_1) \|k_0\|_{\alpha_1}}{ T(\alpha, \alpha_1) -
\tau(\alpha_2,\alpha_1)},
\end{equation}
where $\alpha $ is as in (\ref{Jan41}) and $s\leq t \leq
\tau(\alpha_2 , \alpha_1)$. Now for $s\leq t$, we set
\begin{equation}
  \label{Jan49}
g_s (y_1, y_2) = \int_{\Gamma_0} e^{\alpha |\eta|} |G_{s}
(\eta\cup\{y_1,y_2\})| \lambda (d \eta).
\end{equation}
Let us show that $g_s \in L^1((\mathds{R}^d)^2)$. By (\ref{Jan46a})
we have
\begin{eqnarray}
  \label{Jan50}
& & \int_{(\mathds{R}^d)^2} g_s (y_1, y_2) d y_1 d y_2  =
e^{-2\alpha} \int_{\Gamma_0} |\eta|(|\eta|-1)
e^{-(\alpha'-\alpha)|\eta|} |G_s(\eta)| e^{\alpha'|\eta|} \lambda
(d\eta)\qquad  \\[.2cm]\nonumber & & \quad \qquad  \leq\frac{4 e^{-2\alpha -
2}}{(\alpha'- \alpha)^2} |G_s|_{\alpha'} \leq \frac{4 e^{-2\alpha -
2}T(\alpha_2 , \alpha')|G|_{\alpha_2}}{(\alpha'-
\alpha)^2[T(\alpha_2, \alpha')-\tau(\alpha_2 ,\alpha_1)]}.
\end{eqnarray}
Turn now to (\ref{Jan47}). By means of  item (iv) of Assumption
\ref{ass1} and by (\ref{Jan48}) and (\ref{Jan49}) we get
\begin{eqnarray*}
& & \int_0^t \left\vert\langle \!\langle G_{t-s}, v_s^\sigma \rangle
\!\rangle\right\vert d s \\[.2cm] \nonumber & & \qquad  \leq \beta^* C(\alpha_2,
\alpha_1)\|k_0\|_{\alpha_1}\int_0^t \int_{(\mathds{R}^d)^2} g_s
(y_1, y_2)\left[1-\phi_\sigma (y_1) \phi_\sigma (y_2)\right] ds d
y_1 d y_2,
 \end{eqnarray*}
where we have taken into account that $\alpha$ and $\alpha'$ are
expressed through $\alpha_2$ and $\alpha_1$, see (\ref{Jan41}). Then
the function under the latter integral is bounded from above by
$g_s(y_1, y_2)$ which by (\ref{Jan50}) is integrable on $[0,t]\times
(\mathds{R}^d)^2$. Since this function converges point-wise to $0$
as $\sigma \to 0^+$, by Lebesgue's dominated convergence theorem we
get that
\begin{equation*}
\langle \!\langle G, \Upsilon_1(t,\sigma) \rangle \!\rangle \to 0 ,
\qquad {\rm as} \ \ \sigma\to 0^{+}.
\end{equation*}
The proof that the second summand in the right-hand side of
(\ref{Jan44}) vanishes in the limit $\sigma\to 0^{+}$ is pretty
analogous.
\end{proof}
{\it Proof of Lemma \ref{ILlm}.} By (\ref{32}) and Proposition
\ref{1pn} we have that each $k_0\in \mathcal{K}^\star_{\alpha_1}$ is
the correlation function of some $\mu_0\in \mathcal{P}_{\rm
exp}(\Gamma_0)$. By (\ref{21}) we readily conclude that
\[
\dot{k}_t (\varnothing) = (L^\Delta_{\alpha_2} k_t)(\varnothing)=0.
\]
Hence, $k_t(\varnothing)=k_0(\varnothing)=1$. At the same time, for
$t\leq\tau(\alpha_2, \alpha_1)$ given in (\ref{Jan40}), we have that
$$\langle \!\langle G, k_t \rangle\!\rangle = \lim_{\sigma\to
0^+}\lim_{n\to +\infty} \lim_{l\to +\infty} \langle \!\langle G,
k_t^{\Lambda_n, N_l} \rangle\!\rangle,$$ that follows by Lemma
\ref{JJ10lm} and Proposition \ref{JJ10pn}. Then $\langle \!\langle
G, k_t \rangle\!\rangle \geq 0$ by (\ref{Jan32}) that completes the
proof. \hfill{$\square$}

\section{The Global Solution}

\label{Sec6}

In this section, we continue the solution obtained in Lemma
\ref{1lm} to all $t>0$ and thus prove that it satisfies the upper
bound following from property (i) in Theorem \ref{1tm}.

\subsection{Comparison statements}

Note that the time bound $T(\alpha, \alpha_1)$ defined in (\ref{34})
is a bounded function of $\alpha >\alpha_1$. Then the solution
obtained in Lemma \ref{1lm} may abandon the scale of spaces
$\{\mathcal{K}_\alpha\}_{\alpha \in \mathds{R}}$ in finite time. To
overcome this difficulty we compare $k_t$ with some auxiliary
functions.
\begin{lemma}
  \label{complm}
Let $\alpha_2$, $\alpha_1$ and $\tau(\alpha_2,\alpha_1)$ be as in
Lemma \ref{ILlm}. Then for each $t \in [0, \tau
(\alpha_2,\alpha_1)]$ and arbitrary $k_0 \in
\mathcal{K}_{\alpha_1}^\star$, the following holds
\begin{equation}
\label{59} 0 \le (Q_{\alpha_2 \alpha_1}(t;
B_\upsilon^\Delta)k_0)(\eta) \le (Q_{\alpha_2 \alpha_1}(t;
B_{2,\upsilon}^\Delta)k_0)(\eta), \qquad \eta \in \Gamma_0.
\end{equation}
\end{lemma}
\begin{proof}
The left-hand side inequality follows by Lemma \ref{ILlm} and
(\ref{32a}). By the second line in (\ref{54a}) we conclude that $w_t
= Q_{\alpha_2 \alpha_1}(t; B_{2,\upsilon}^\Delta) k_0$ is the unique
solution of the equation
\begin{equation*}
\dot{w}_t = ((A^\Delta_\upsilon)_{\alpha_2} +
(B^\Delta_{2,\upsilon})_{\alpha_2}) w_t, \qquad w_t|_{t=0}= k_0,
\end{equation*}
on the time interval $[0,
T(\alpha_2,\alpha_1;B_{2,\upsilon}^\Delta))\supset [0,
T(\alpha_2,\alpha_1;B_{\upsilon}^\Delta))$ since
$T(\alpha_2,\alpha_1;B_\upsilon^\Delta)\le
T(\alpha_2,\alpha_1;B_{2,\upsilon}^\Delta)$. Then we have that $w_t
-k_t \in \mathcal{K}_{\alpha_2}$ for all
$t\leq\tau(\alpha_2,\alpha_1)$. Now we choose $\alpha', \alpha \in
[\alpha_1, \alpha_2]$ according to (\ref{Jan41}) so that
(\ref{Jan40}) holds, and then write
\begin{eqnarray}
  \label{61}
  w_t -k_t & = & (Q_{\alpha_2 \alpha_1}(t;
B_{2,\upsilon}^\Delta)k_0)(\eta) - (Q_{\alpha_2 \alpha_1}(t;
B_{\upsilon}^\Delta)k_0)(\eta) \\[.2cm] \nonumber & = &
\int_0^t
Q_{\alpha_2\alpha'}(t-s;B_{2,\upsilon}^\Delta)(-B_1^\Delta)_{\alpha'\alpha}k_s
ds, \qquad t<\tau(\alpha_2,\alpha_1),
\end{eqnarray}
where the operator $(-B_1^\Delta)_{\alpha'\alpha}$ is positive with
respect to the cone $\mathcal{K}_\alpha^+$ defined in (\ref{32a}).
In the integral in (\ref{61}), for all $s\in [0,
\tau(\alpha_2,\alpha_1)]$, we have that $k_s \in
\mathcal{K}_{\alpha}$ and
$Q_{\alpha_2\alpha'}(t-s;B_{2,\upsilon}^\Delta) \in
\mathcal{L}(\mathcal{K}_{\alpha'},\mathcal{K}_{\alpha_2})$ is
positive. We also have that $k_s \in \mathcal{K}_{\alpha}^\star
\subset \mathcal{K}_{\alpha}^+$ (by Lemma \ref{ILlm}). Therefore
$w_t -k_t \in \mathcal{K}_{\alpha_2}^+$ for $t\le
\tau(\alpha_2,\alpha_1)$, which yields (\ref{59}).
\end{proof}
The next step is to compare $k_t$ with
\begin{equation}
  \label{63}
r_t(\eta) = \|k_0\|_{\alpha_1}\exp\left( (\alpha_1 + c
t)|\eta|\right),
\end{equation}
where $\alpha_1$ is as in Lemma \ref{complm} and
\begin{equation}
  \label{64}
c  = \langle b \rangle + \upsilon  - m_*, \qquad m_*= \inf_{x\in
\mathds{R}^d} m(x).
\end{equation}
Let us show that $r_t\in \mathcal{K}_\alpha$ for $t\leq
\tau(\alpha_2, \alpha_1)$, where $\alpha$ is given in (\ref{Jan41}).
In view of (\ref{nk}), this is the case if the following holds
\begin{equation}
  \label{64a}
\alpha_1 + c \tau(\alpha_2, \alpha_1) \leq \frac{1}{3}\alpha_2 +
\frac{2}{3}\alpha_1,
\end{equation}
which amounts to $c \leq 2\langle b \rangle + \upsilon + \langle a
\rangle e^{\alpha_2}$, see (\ref{Jan40}) and (\ref{34}). The latter
obviously holds by (\ref{64}).
\begin{lemma}
  \label{comp1lm}
Let $\alpha_1$, $\alpha_2$ and $k_t = Q_{\alpha_2\alpha_1}(t)k_0$ be
as in Lemma \ref{complm}, and $r_t$ be as in (\ref{63}), (\ref{64}).
Then $k_t (\eta) \leq r_t(\eta)$ for all $t\leq
\tau(\alpha_2,\alpha_1)$ and $\eta\in \Gamma_0$.
\end{lemma}
\begin{proof}
The idea is to show that $w_t (\eta) \leq r_t(\eta)$ and then to
apply the estimate obtained in Lemma \ref{complm}. Set $\tilde{w}_t
=Q_{\alpha_2 \alpha_1}(t; B_{2,\upsilon}^\Delta) r_0$. Since $k_0
\in \mathcal{K}_{\alpha_1}$, we have that $k_0 \leq r_0$. Then  by
the positivity discussed in Remark \ref{Jan10rk} we obtain $w_t \leq
\tilde{w}_t$, and hence $k_t \leq \tilde{w}_t$, holding for all $t
\leq \tau(\alpha_2, \alpha_1)$. Thus, it remains to prove that
$\tilde{w}_t (\eta) \leq r_t(\eta)$. To this end we write, cf.
(\ref{61}),
\begin{equation}
  \label{65}
\tilde{w}_t - r_t = \int_0^t Q_{\alpha_2 \alpha'}(t-s;
B_{2,\upsilon}^\Delta) D_{\alpha'\alpha} r_s d s,
\end{equation}
where $\alpha'$ and $\alpha$ are as in (\ref{Jan41}) and the bounded
operator $D_{\alpha'\alpha}$ acts as follows: $D= A^\Delta_\upsilon
+ B^\Delta_{2,\upsilon} - J_{c}$, where $(J_{c}k)(\eta) = c |\eta|
k(\eta)$ with $c$ as in (\ref{64}). The validity of (\ref{65}) can
be established by taking the $t$-derivative of both sides and then
taking into account (\ref{63}) and (\ref{54a}). Note that $r_s$ in
(\ref{65}) lies in $\mathcal{K}_\alpha$, as it was shown above. By
means of (\ref{21}) the action of $D$ on $r_s$ can be calculated
explicitly yielding
\begin{eqnarray}
  \label{66}
(D r_t)(\eta) & = & - \Psi_\upsilon (\eta) r_t(\eta) +
\int_{\mathds{R}^d}\sum_{y_1\in \eta} \sum_{y_2\in \eta\setminus
\eta_1} r_t(\eta\cup x\setminus \{y_1,y_2\}) b(x|y_1,y_2) d x \\[.2cm] \nonumber&
+ & \upsilon |\eta| r_t(\eta) + 2\int_{(\mathds{R}^d)^2}\sum_{y_1\in
\eta} r_t(\eta\cup x\setminus y_1)  b(x|y_1,y_2) d x  d y_2 -
c|\eta| r_t(\eta) \\[.2cm]\nonumber & = & \bigg{(} - M(\eta) -
E^a(\eta) - \langle b \rangle |\eta|  + e^{-\alpha_1 - c t}E^b
(\eta) + 2 \langle b \rangle |\eta| - c|\eta| \bigg{)} r_t (\eta).
\end{eqnarray}
Since $\alpha_1 >-\log \omega$, by Proposition \ref{2pn} we have
that
\[
- E^a(\eta) + e^{-\alpha_1 - c t}E^b (\eta) \leq \upsilon |\eta|,
\]
by which and (\ref{64}) we obtain from (\ref{66}) that $(D
r_t)(\eta) \leq 0$. We apply this in (\ref{65}) and obtain
$\tilde{w}_t \leq r_t$ which completes the proof. \end{proof}
\begin{remark}
  \label{JK10rk}
By (\ref{64}) we obtain that $c\leq 0$ (and hence $k_t\in
\mathcal{K}_{\alpha_1}$) whenever
\[
m_* \geq  \langle b \rangle + \upsilon.
\]
In the short dispersal case, see Remark \ref{1rk}, one can take
$\upsilon =0$. In the long dispersal case, by Proposition \ref{pfi}
one can make $\upsilon$ as small as one wants by taking small enough
$\omega$ and hence big enough $\alpha_1$. Then, the evolution of
$k_t$ leaves the initial space invariant if the following holds
\begin{equation}
  \label{67}
m_*  >  \langle b \rangle.
\end{equation}
In the short dispersal case, one can allow equality in (\ref{67}).
\end{remark}

\subsection{Completing the proof}

The choice of the initial space should satisfy the condition
$\alpha_1 > -\log \omega$. At the same time, the parameter
$\alpha_2>\alpha_1$ can be taken arbitrarily. In view of the
dependence of $T(\alpha_2, \alpha_1)$ on $\alpha_2$, see (\ref{34}),
the function $\alpha_2 \mapsto T(\alpha_2, \alpha_1)$ attains
maximum at $\alpha_2 = \alpha_1 + \delta (\alpha_1)$, where
\begin{equation}
 \label{Jan26}
 \delta (\alpha)  =  1+ W\left(\frac{2\langle b \rangle + \upsilon}{\langle a \rangle} e^{-\alpha -1} \right),
\end{equation}
Here $W$ is Lambert's function, see \cite{W}. Then we have
\begin{equation}
  \label{68}
T_{\max} (\alpha_1) = \max_{\alpha_2 >\alpha_1} T(\alpha_2,
\alpha_1) = \exp\left( -\alpha_1 - \delta(\alpha_1) \right) /\langle
a \rangle.
\end{equation}
{\it Proof of Theorem \ref{1tm}.} Fix $\upsilon$ and then find small
$\omega$ (see Proposition \ref{pfi}) such that the inequality in
Proposition \ref{2pn} holds true. Thereafter, take $\alpha_0 >-\log
\omega$ such that $k_{\mu_0}\in \mathcal{K}_{\alpha_0}$. Then take
$c$ as given in (\ref{64}) with this $\upsilon$. Next, set $T_1 =
T_{\max} (\alpha_0)/3$, see (\ref{68}), and also $\alpha^*_1 =
\alpha_0+ cT_1$, $\alpha_1 = \alpha_0 + \delta(\alpha_0)$, see
(\ref{Jan26}). Clearly, $\alpha_1^* < \alpha_1$ that can be checked
similarly as in (\ref{64a}). By Lemma \ref{ILlm} it follows that,
for $t\leq T_1$, $k_t = Q_{\alpha_1 \alpha_0} (t) k_{\mu_0}$ lies in
$\mathcal{K}^\star_{\alpha_1}$, whereas by Lemma \ref{comp1lm} we
have that $k_t\in \mathcal{K}^\star_{\alpha_t}$ with $\alpha_t =
\alpha_0 +c t\leq \alpha_1^*$. Clearly, for $T<T_1$, the map $[0,
T)\ni t \mapsto k_t\in \mathcal{K}_{\alpha_T}$ is continuous and
continuously differentiable, and both claims (i) and (ii) are
satisfied since (by construction) $\dot{k}_t = L^\Delta_{\alpha_1}
k_t = L^\Delta_{\alpha_T} k_t$, see (\ref{29}). Now, for $n\geq 2$,
we set
\begin{gather}
  \label{70}
T_n  =  T_{\rm max} (\alpha^*_{n-1}) /3, \quad \alpha_n^* =
\alpha^*_{n-1} + c T_n , \\[.2cm] \nonumber \alpha_n  =  \alpha^*_{n-1} +
\delta(\alpha^*_{n-1}).
\end{gather}
As for $n=1$, we have that $\alpha_n^* < \alpha_n$ and $T_n<
T(\alpha_n,\alpha_{n-1}^*)$ holding for all $n\geq 2$. Thereafter,
set
\begin{equation*}
k_t^{(n)} = Q_{\alpha_n \alpha_{n-1}^*} (t)
k^{(n-1)}_{T_{n-1}},\qquad t \in [0,T (\alpha_n,\alpha_{n-1}^*)),
\end{equation*}
where $k^{(1)}_t = Q_{\alpha_1 \alpha_0} (t) k_{\mu_0}$. Then, for
each $T<T_n$ both maps $[0,T)\ni t \mapsto k^{(n)}_t \in
\mathcal{K}_{\bar{\alpha}_{n-1}(T)}$ and $[0,T)\ni t \mapsto
L^\Delta_{\bar{\alpha}_{n-1}(T)} k^{(n)}_t \in
\mathcal{K}_{\bar{\alpha}_{n-1}(T)}$ are continuous, where
$\bar{\alpha}_{n-1}(T) := \alpha_{n-1}^* + cT$. The continuity of
the latter map follows by the fact that $k^{(n)}_t \in
\mathcal{K}_{\bar{\alpha}_{n-1}(t)}\hookrightarrow
\mathcal{K}_{\bar{\alpha}_{n-1}(T)}$ and that
$L^\Delta_{\bar{\alpha}_{n-1}(T)}|_{\mathcal{K}_{\bar{\alpha}_{n-1}(t)}}
= L^\Delta_{\bar{\alpha}_{n-1}(T)\bar{\alpha}_{n-1}(t)}$, see
(\ref{29}). Moreover $k^{(n)}_0 = k^{(n-1)}_{T_{n-1}}$ and
$L^\Delta_{\alpha^*_{n-1}+\varepsilon } k^{(n)}_0 =
L^\Delta_{\alpha_{n-1}^*+\varepsilon } k^{(n-1)}_{T_{n-1}}$ holding
for each $\varepsilon>0$. Then the map in question $t\mapsto k_t$ is
\begin{equation*}
  k_{t+T_1 \cdots + T_{n-1}} = k_t^{(n)}, \qquad t\in [0,T_n],
\end{equation*}
provided that the series $\sum_{n\geq 1}T_n$ is divergent. By
(\ref{68}) we have
\begin{equation}
  \label{71}
\sum_{n\geq 1}T_n = \frac{1}{3 \langle a \rangle} \sum_{n\geq
1}\exp\left( -\alpha_{n-1}^* - \delta(\alpha_{n-1}^* ) \right).
\end{equation}
For the convergence of the series in the right-hand side it is
necessary that $\alpha_{n-1}^* + \delta(\alpha_{n-1}^*)\to +\infty$,
and hence  $\alpha_{n-1}^* \to +\infty$ as $n\to +\infty$, since
$\delta(\alpha)$ is decreasing. By (\ref{70}) we have $\alpha_n^*=
\alpha_0 + c(T_1+\cdots +T_n)$. Then the convergence of $\sum_{n\geq
1}T_n$ would imply that $\alpha_n^*\leq \alpha^*$ for some number
$\alpha^*>0$ that contradicts the convergence of the right-hand side
of (\ref{71}). \hfill{$\square$} \vskip.1cm \noindent {\it Proof of
Corollary \ref{Jaco}.} For a compact $\Lambda$, let us show that
$\mu^\Lambda_t\in \mathcal{D}$, that is, $R_{\mu_t}^\Lambda \in
\mathcal{D}^\dagger$, see (\ref{L2}). For  $k_t =k_{\mu_t}$
described in Theorem \ref{1tm}, by (\ref{13}) we have
\begin{equation*}
R_{\mu_t}^{\Lambda}(\eta)=\int_{\Gamma_{\Lambda}}
(-1)^{|\xi|}k_t(\eta \cup \xi)\lambda(d\xi).
\end{equation*}
Let $\alpha >\alpha_0$ be such that $k_t\in \mathcal{K}_{\alpha}$. Then using (\ref{nk}), (\ref{8}), (\ref{22}) and (\ref{25}) we calculate
\begin{eqnarray*}
\int_{\Gamma_{\Lambda}}\Psi(\eta) R_{\mu_t}^{\Lambda}(\eta)
\lambda(d\eta)&=& \int_{\Gamma_{\Lambda}}\Psi(\eta)
\int_{\Gamma_{\Lambda}} (-1)^{|\xi|}k_t(\eta \cup
\xi)\lambda(d\xi)\lambda(d\eta)\\[.2cm]
&\le &  \int_{\Gamma_{\Lambda}}\Psi(\eta) \|k\|_{\alpha}e^{\alpha|\eta|}\lambda(d\eta) \int_{\Gamma_{\Lambda}}e^{\alpha|\xi|}
\lambda(d\xi)\\[.2cm]
&\le& \|k\|_{\alpha} (m^*+a^*+\langle b\rangle) \int_{\Gamma_{\Lambda}} |\eta|^2e^{\alpha|\eta|}\lambda(d\eta)
\exp\left(|\Lambda|e^{\alpha}\right)\\[.2cm]
&=& \|k\|_{\alpha} (m^*+a^*+\langle b\rangle) |\Lambda|e^\alpha
\left( 2 +
|\Lambda|e^\alpha\right)\exp\left(2|\Lambda|e^{\alpha}\right),
\end{eqnarray*}
where $|\Lambda|$ is the Euclidean volume of $\Lambda$. That yields
$\mu_t^{\Lambda} \in \mathcal{D}$. The validity of (\ref{Ja})
follows by (\ref{11}). \hfill{$\square$}

\section*{Acknowledgment}
The authors are grateful to Krzysztof Pilorz for valuable assistance
and discussions. In the period 2016-17, the research of both authors
related to this paper was supported by the European Commission under
the project STREVCOMS PIRSES-2013-612669. In March 2017, during his
stay in Bucharest Yuri Kozitsky was supported by Research Institute
of the University of Bucharest. In 2018, he was supported by
National Science Centre, Poland, grant 2017/25/B/ST1/00051. All
these supports are cordially acknowledged.
\appendix
\setcounter{secnumdepth}{1}

\section{The proof of Proposition \ref{2pn}}
According to Assumption \ref{ass1}, $\beta$ is Riemann integrable,
then for an arbitrary $\varepsilon >0$, one can divide
$\mathbb{R}^d$ into equal cubic cells $E_l$, $l\in \mathbb{N}$, of
side $h>0$ such that the following holds
\begin{equation}
\label{pz} h^d\sum_{l=1}^{+\infty} \beta_l \le \langle b
\rangle+\varepsilon, \qquad \beta_l:=\sup_{x\in E_l}\beta(x).
\end{equation}
For $r>0$, set $K_r(x)=\lbrace y\in \mathbb{R}^d:|x-y|<r \rbrace$,
$x\in \mathbb{R}^d$, and
\begin{equation}
\label{ar}
a_r = \inf_{x\in K_{2r}(0)}a(x).
\end{equation}
Then we fix $\varepsilon$ and pick $r>0$ such that $a_r>0$.  For
$r$, $h$ and $\varepsilon$ as above, we prove the statement by the
induction in the number of points in $\eta$. By (\ref{fi}) we
rewrite (\ref{2pnN}) in the form
\begin{equation}
\label{u} U_{\omega}(\eta):=\upsilon|\eta|+\Phi_{\omega}(\eta)\ge 0,
\end{equation}
and, for some $x\in \eta$, consider
\begin{eqnarray*}
U_{\omega}(x,\eta \setminus x)&:= & U_{\omega}(\eta)- U_{\omega}(\eta \setminus x)\\
&=& \upsilon +2 \left( \sum_{y\in \eta \setminus x} a(x-y)-\omega \sum_{y\in \eta \setminus x}\beta(x-y)  \right).
\end{eqnarray*}
Set $c_d=|K_1|$ and let $\Delta(d)$ be the packing constant for
rigid balls in $\mathbb{R}^d$, cf. \cite{gro}. Then set
\begin{equation}
\label{del} \delta=\max \lbrace \beta^*; (\langle b
\rangle+\varepsilon)g_d(h,r), \rbrace,
\end{equation}
where
$$g_d(h,r)=\frac{\Delta(d)}{c_d}\left( \frac{h+2r}{hr} \right)^d.$$
Next, assume that $\upsilon$ and $\omega$ satisfy, cf. (\ref{ar}),
\begin{equation}
\label{ome} \omega \le \min \left\{ \frac{\upsilon}{2\delta};
\frac{a_r}{\delta} \right\}.
\end{equation}
Let us show that
\begin{itemize}
\item[(i)] for each $\eta=\lbrace x,y \rbrace$, (\ref{ome}) implies (\ref{u});
\item[(ii)] for each $\eta$, one finds $x\in \eta$ such that $U_{\omega}(x,\eta \setminus x)\ge 0$ whenever (\ref{ome}) holds.
\end{itemize}
To prove (i) by (\ref{ome}) and (\ref{del}) we get
\begin{eqnarray*}
U_{\omega}(\lbrace x,y \rbrace)&=&2 \upsilon +2a(x-y)-2\omega\beta(x-y)\\
&\geq & (\upsilon - 2\omega\beta^*)+2a(x-y)\ge 0.
\end{eqnarray*}
To prove (ii), for $y\in \eta$, we set
\begin{equation}
\label{s}
s=\max_{y\in \eta} |\eta \cap K_{2r}(y)|.
\end{equation}
Let also $x\in \eta$ be such that $|\eta \cap K_{2r}(x)|=s$. For
this $x$, by $E_l(x)$, $l\in \mathbb{N}$, we denote the
corresponding translates of $E_l$ which appear in (\ref{pz}). Set
$\eta_l=\eta \cap E_l(x)$ and let $l_* \in \mathbb{N}$ be such that
$\eta \subset \bigcup_{l\le l_*}E_l(x)$ which is possible since
$\eta$ is finite. For a given $l$, a subset $\zeta_l \subset \eta_l$
is called $r-$admissible if for each distinct $y,z\in \zeta_l$, one
has that $K_r(y)\cap K_r(z)= \emptyset$. Such a subset $\zeta_l$ is
called maximal $r-$admissible if $|\zeta_l|\ge |\zeta'|$ for any
other $r-$admissible $\zeta_l'$. It is clear that
\begin{equation}
\label{etal}
\eta_l \subset \bigcup_{z\in \zeta_l}K_{2r}(z).
\end{equation}
Otherwise, one finds $y\in \eta_l$ such that $|y-z|\ge 2r$, for each
$z\in \zeta_l$, which yields that $\zeta_l$ is not maximal. Since
all the balls $K_r(z)$, $z\in \zeta_l$, are contained in the
$h-$extended cell
\begin{equation*}
E_l^h(x):=\lbrace y\in \mathbb{R}^d: \inf_{z\in E_l(x)}|y-z|\le h \rbrace,
\end{equation*}
their maximum number - and hence $|\zeta_l|$ - can be estimated as follows
\begin{equation}
\label{zetal}
|\zeta_l|\le \Delta(d)V(E_l^h(x))/c_dr^d=h^d\frac{\Delta(d)}{c_d}\left( \frac{h+2r}{hr} \right)^d=h^dg_d(h,r),
\end{equation}
where $c_d$ and $\Delta(d)$ are as in (\ref{del}). Then by (\ref{s}) and (\ref{etal}) we get
\begin{equation*}
\sum_{y\in \eta \setminus x}\beta(x-y)\le \sum_{l=1}^{l_*} \sum_{z\in \zeta_l} \sum_{y\in K_{2r}(z)\cap \eta_l}\beta_l.
\end{equation*}
The cardinality of $K_{2r}(z)\cap \eta_l$ does not exceed $s$, see
(\ref{s}), whereas the cardinality of $\zeta_l$ satisfies
(\ref{zetal}). Then
\begin{equation}
\label{ogr} \sum_{y\in\eta\setminus x}\beta(x-y)\le s
g_d(h,r)\sum_{l=1}^{\infty}\beta_l h^d \le sg_d(h,r)(\langle b
\rangle +\varepsilon)\le s\delta.
\end{equation}
On other hand, by (\ref{ar}) and (\ref{s}) we get
\begin{equation*}
\sum_{y\in\eta\setminus x}a(x-y)\ge \sum_{y\in (\eta\setminus x)\cap K_{2r}(x)}a(x-y) \ge (s-1)a_r.
\end{equation*}
We use this estimate and (\ref{ogr}) in (\ref{u}) and obtain
$$U_{\omega}(x, \eta \setminus x)\ge 2\delta
\left[ \left( \frac{\upsilon}{2\delta}-\omega
\right)+(s-1)\left(\frac{a_r}{\delta}-\omega \right) \right]\ge 0,$$
see (\ref{ome}). Thus, (ii) also holds and the proof follows by the
induction in $|\eta|$.

\end{document}